\documentclass[12pt,a4paper]{article}
\usepackage[utf8]{inputenc}

\usepackage{amsfonts,amsmath, amssymb}
\usepackage{amsthm}
\usepackage{graphicx}
\usepackage{longtable, rotating}
\usepackage{float, mathtools}
\usepackage{array}
\usepackage{enumerate}
\usepackage{threeparttable}
\usepackage{mathrsfs}
\usepackage[text={15cm,24cm},top=2cm,margin=2cm]{geometry} 
\usepackage{listings}
\usepackage{url}
\usepackage{subfigure}
\usepackage{multirow}
\usepackage{textcomp,booktabs}
\usepackage[usenames,dvipsnames]{color}
\usepackage{colortbl}
\usepackage{lscape}
\usepackage{bm,bbm}
\usepackage[numbers,square]{natbib}
\usepackage[colorlinks=true,urlcolor=black,citecolor=blue,linkcolor=blue,bookmarks=true]{hyperref} 
\usepackage{verbatim}
\usepackage{thmtools}
\usepackage{thm-restate}
\usepackage{cleveref}
\usepackage{ulem}
\usepackage{tikz}
\usetikzlibrary{positioning,arrows}
\tikzset{
	block/.style={
		draw, 
		rectangle, 
		minimum height=1cm, 
		minimum width=2cm, align=center
	}, 
	line/.style={->,>=latex'}
}

\definecolor{mygray}{gray}{.9}
\usepackage[labelfont={bf,it},textfont={bf,it}]{caption}
\DeclareGraphicsExtensions{.jpg,.pdf,.gif,.png}

\newtheorem{proposition}{Proposition}

\newtheorem{lemma}{Lemma}

\newtheorem{corollary}{Corollary}

\newcommand{\R}{\mathbb{R}}

\numberwithin{equation}{section}
\numberwithin{definition}{section}
\numberwithin{remark}{section}
\numberwithin{theorem}{section}
\numberwithin{proposition}{section}
\numberwithin{lemma}{section}
\numberwithin{remark}{section}
\numberwithin{example}{section}
\numberwithin{figure}{section}
\numberwithin{conjecture}{section}
\numberwithin{table}{section}
\numberwithin{corollary}{section}

\begin{document}
\pagestyle{plain}
\title{\bf Approximating a spatially-heterogeneously mass-emitting
object by multiple point sources in a diffusion model}
\author{Qiyao Peng$^{1,2*}$, Sander C. Hille$^2$}
\date{{\footnotesize \it
    $^1$ Mathematics for AI in Real-world Systems, School of Mathematical Sciences, Lancaster University, LA1 4YF, Lancaster, United Kingdom. qiyao.peng[AT]lancaster.ac.uk\\%
    $^2$ Mathematical Institute, Faculty of Science, Leiden University. Einsteinweg 55, 2333 CC, Leiden, The Netherlands}\\[2ex]%
    \today
}
\maketitle

\begin{abstract} Various biological cells secrete diffusing chemical compounds into their environment for communication purposes. Secretion usually takes place over the cell membrane in a spatially heterogeneous manner. Mathematical models of these processes will be part of more elaborate models, e.g. of the movement of immune cells that react to cytokines in their environment. Here, we compare two approaches to modelling of the secretion-diffusion process of signalling compounds. The first is the so-called {\it spatial exclusion model}, in which the intracellular space is excluded from consideration and the computational space is the extracellular environment. The second consists of {\it point source models}, where the secreting cell is replaced by one or more non-spatial point sources or sinks, using -- mathematically --  Dirac delta distributions. We propose a multi-Dirac approach and provide explicit expressions for the intensities of the Dirac distributions. We show that two to three well-positioned Dirac points suffice to approximate well a temporally constant but spatially heterogeneous flux distribution of compound over the cell membrane, for a wide range of variation in flux density and diffusivity. The multi-Dirac approach is compared to a single-Dirac approach that was studied in previous work. Moreover, an explicit Green's function approach is introduced that has significant benefits in circumventing numerical instability that may occur when the Dirac sources have high intensities. 
\vskip 0.1cm

\end{abstract}

\section{Introduction}
\noindent
Living organisms, from single cells to animals and plants, communicate with each other, for various purposes, by means of secreting and detecting chemical compounds in their environment, which get spread by diffusive or advective transport, or both. For example, plants may signal others warning for predating insects and herbivores \citep{Furstenberg_Hagg_2013}; animals may secrete pheromones to attract partners for mating \citep{menini2009neurobiology}; single cells may thus coordinate and accomplish tasks that they cannot do it alone, for instance, in the development of tissues, organs and whole organisms \citep{Bender:2020,Karban:2015,Perbal_2003}, and for regulating cell behaviours, as in the immune system where immune cells are attracted to target areas by means of cytokines \citep{Harvanov__2023, janeway2001immunobiology, Neitzel2014Cell}.

Means of communication are not limited to the diffusing compounds that are the subject of modelling in this paper. At the cellular level, various types of communication have been identified \citep{cooper2023cell}.  One, other than by signalling molecules, is cell-matrix interaction. In any case, communication is realized through the extracellular environment, with the cell membrane, which defines the boundary of the cell and separates the intracellular and extracellular environments, in particular specific proteins embedded within, realizing detection and secretion of compounds. Over the cell membrane, receptors, ion channel proteins and transporter proteins are distributed inhomogeneously \citep{cooper2022cell}. The latter two are responsible for the permeability of the membrane for various types of molecules, such as ions, cytokines, nutrients, etc. Compounds can be secreted as well by means of exocytosis of vesicles that contain these compounds at high concentration, in the intermediate `protein free' membrane space. This process is complex and highly regulated \citep{mierke2020cellular}. For a more detailed review of the cell membranes, see e.g. \citet{buehler2015cell, cooper2022cell}.

Thus, when studying cell-to-cell communication using mathematical modelling one considers a (reaction-) diffusion equation for the compound in the environment, possibly with an advective term added if that type of transport is relevant. Here, we shall limit attention to diffusion only. 

Depending on the spatial scale of a study, one encounters models in which either the individual cells are small in number and viewed as non-negligibly spatially extended, or there is a (very) large number of cells, of negligible size. 
The latter leads to a continuum description on a spatial domain $\Omega$ with reaction terms in the diffusion equation on $\Omega$ that represent detection and secretion of signalling compound essentially `everywhere'. The former leads to exclusion from $\Omega$ of the space taken by the interior and cell membrane of all cells, represented by the closed domain $\overline{\Omega}_C$, say. We call this a {\it spatial exclusion model}; see Figure~\ref{Fig_schematic}(a) for a schematic representation. Secretion of compound is modelled then by flux conditions on the boundary $\partial\Omega_C$, i.e. the totality of all cell membranes. Because of the heterogeneity of localisation of proteins and exocytosis events on the cell membrane, the associated flux densities in the model are expected to be spatially inhomogeneous over each cellular part (i.e. connected component) of $\partial\Omega_C$.

Biological cells have varying shapes. For mathematical convenience, each individual cell may be given a regular, spherical, shape in a spatial exclusion model, as an approximation of a model with detailed cell shape. However, a model with such regular shape and inhomogeneous flux density distribution over the boundary may yield an acceptable approximating solution to a model with detailed, realistic cellular shapes (and inhomogeneous flux density); see Figure~\ref{Fig_schematic}(b). In this paper, although it is an interesting research question, we shall not consider the quality of approximation of solutions between spatial exclusion models with regular and realistic shaped cells. Instead, we shall be concerned with studying approaches towards approximation of a spatial exclusion model with a regularly shaped cell and inhomogeneous flux density over its boundary by a so-called {\it point source model}. The cell is then replaced by a finite number of point sources (see Figure~\ref{Fig_schematic}(b) -- blue points) or sinks (red points).

\begin{figure}[h!]
    \centering
    \subfigure[Illustration of the computational domain in a spatial exclusion model with circular cell]{\includegraphics[width = 0.48\textwidth]{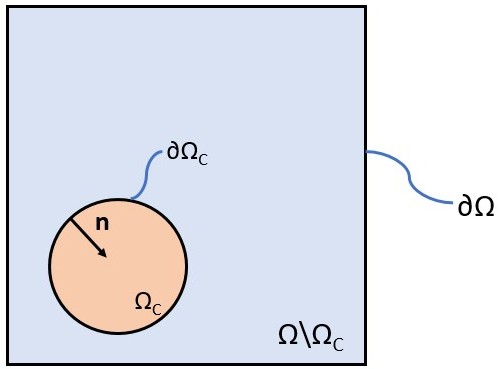}}\hfill
    \subfigure[Model-wise conversion of an irregular-shaped cell into a circular cell, then to mass-emitting point(s)]{
    \includegraphics[width = 0.48\textwidth]{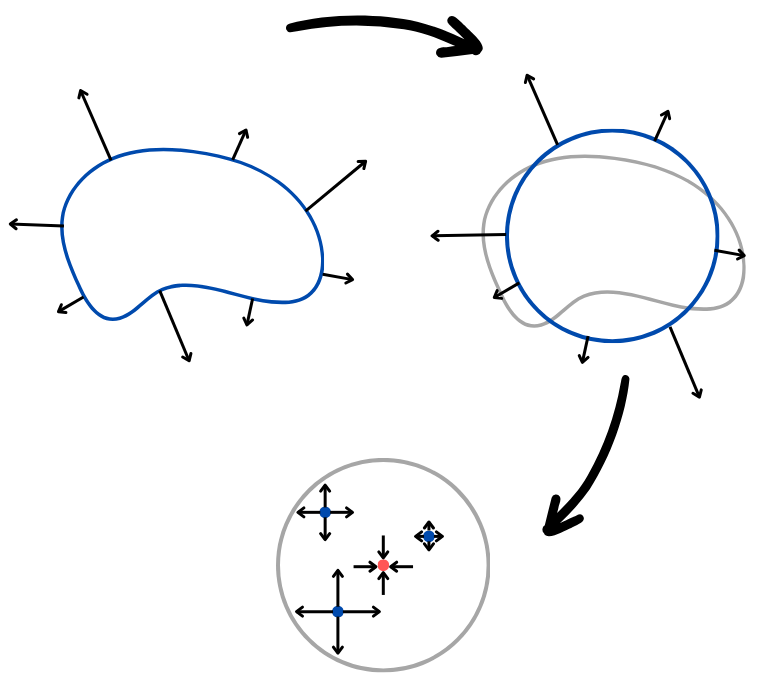}
    }
    \caption{Geometric configuration of the spatial domain in the models. Panel (a): A schematic presentation of one circular cell $\Omega_C$ embedded in the domain $\Omega$. The boundaries of the cell and the entire domain are denoted by $\partial\Omega_C$ and $\partial\Omega$, respectively. We define the normal vector $\boldsymbol{n}$ of $\partial\Omega_C$ pointing towards the center of the cell. This figure is taken from \citet{Peng2023}. Panel (b): A biological cell has a more irregular shape.  The heterogeneous flux density over its boundary can be converted into a circular cell with (different) heterogeneous flux density. Then, using the method proposed in this article, the circular cell can be converted to mass-emitting points, i.e. sources (blue) and sinks (red), with the cell boundary now virtual. }
    \label{Fig_schematic}
\end{figure}

  The quality of approximation is assessed quantitatively by numerical analysis. In the latter model, the cell is replaced by a small number of point sources and sinks with particular intensities for mass exchange, positioned within the part of the full domain $\Omega$ that belonged to the cell $\Omega_C$ in the spatial exclusion model. In the model, this is realized by adding Dirac delta distributions with particular intensities as reaction terms to the diffusion equation on the full domain $\Omega$. 

The replacement by a point source model with Dirac delta distributions is of interest for mathematical convenience as well. It helps to improve computational efficiency, in particular when there are many cells present, but not as many that a continuum approximation is acceptable. Moreover, in a setting where cells start moving, the application of a finite-element method (FEM) for numerical analysis will require regeneration of the mesh over time steps, which will be a large computational burden. Moving point sources are expected to be less computationally demanding. Computational efficiency will be relevant in situations where quick, adequate simulations are needed, for example for digital twins in a medical setting \citep{Haleem_2023, Sun_2023}. 

It then becomes relevant to determine the adequacy of the point source model in approximating a spatial exclusion model, with inhomogenous boundary conditions. A comparison of solutions have so far been made theoretically in \citet{HMEvers2015} and numerically in our previous work \citet{Peng2023}, to the best of our knowledge. This paper starts investigating approximation when inhomogeneous flux densities over the cell membrane occur.

Note that a point source -- by necessity -- will spread mass (or remove mass for a sink) in symmetric fashion from its neighbourhood. In this paper we propose to replace a spatially extended cell in that case by a finite number of well-located point sources (and sinks). A key question is, to determine intensities for the associated Dirac delta distributions in the model. We describe in Section~\ref{Sec_Multi} a method to do so and then validate this approach numerically, quantitatively, in the following sections. We consider a two-dimensional setting for the domain $\Omega$. Any inhomogenous flux density over the circular boundary $\Omega_C$ is a spatially $2\pi-$periodic function in angular coordinate and can hence be decomposed as a superposition of Fourier modes. We consider the first few modes only. We expect that high frequency modes, due to the regularizing effect of diffusion, will not contribute much to improvement of quality of approximation. It remains a practical question for application how many modes are sufficient, given diffusion constant and size of cells. This is a topic for further research.
\vskip 2mm

This paper is structured as follows: Section~\ref{Sec_math_formulation} presents the mathematical formulation of the spatial exclusion and point source model, and a corresponding phrasing of the main question. Section~\ref{Sec_Multi} proposes the idea of using multiple Dirac points to represent the cell, if the flux density is a spatially inhomogeneous function, particularly the sinusoidal function. As a consequence, a natural question is how to select the intensity of these Dirac points. In Section~\ref{Sec_Amplitude}, we start with the case of polarized flux density, represented by a $2\pi-$periodic sinusoidal function to investigate the locations and the intensity of the off-centre Dirac points, then later we extend this to the general case. Section~\ref{Sec_Comp_Setup} introduces the explicit Green's function approach to solve the point source model, and all models are non-dimensionalized. Afterwards, Section~\ref{Sec_Homogeneity} introduces a homogeneity indicator to determine the relative difference between the solutions to the homogeneous and inhomogeneous flux density in the spatial exclusion model. Numerical results are presented in Section~\ref{Sec_Results} and Monte Carlo simulations, investigating the effects on the solutions of changes in parameters, are conducted in Section~\ref{Sec_MC_n_2}. In Section~\ref{Sec_Mesh_Comparison}, we compare the classical FEM and the explicit Green's function approach to solve the point source model, on a fine and coarse mesh, respectively. Finally, Section~\ref{Sec_Conclu} delivers conclusions and discussion.

\section{Mathematical formulation of the question}\label{Sec_math_formulation}
\noindent
In our previous work \citep{Peng2023}, we considered the situation that the cell boundary releases diffusive compounds into the direct environment in a homogeneous manner, i.e. the flux density over the cell boundary is constant. Two models were discussed (namely, \textit{the spatial exclusion model} and \textit{the point source model}) that describe the secretion and subsequent diffusion of compounds in the environment of  cells, particularly the consistency between the solutions to each model. The spatial exclusion model separates the intracellular and extracellular environment, prescribes flux conditions on the boundary separating these, and the intracellular part is excluded from the computational domain. The point source model replaces each spatially extended cell by a point source, by means of a Dirac delta distribution (also called a Dirac measure). The location of this point source we call a {\it Dirac point}. 

From the numerical results, we observed that for a relatively small diffusion coefficient, there is a systematic time delay between the two approaches. Hence, to compensate for this discrepancy, in the point source approach, we used the Gaussian kernel (which is also the form of the fundamental solution to the diffusion equation) as the initial condition of the intracellular environment. Furthermore, we investigated how to determine the intensity and the variance of this Gaussian kernel.

Here, the set-up is similar to \citet{Peng2023}, but we will now consider approximating spatially inhomogeneous flux densities over the cell boundary in the spatial exclusion model, instead of constant flux densities. These prescribed flux densities are still time-independent. Focus point is the quality of approximation that can be reached by a point source model.

\subsection{Spatial exclusion and point source models}
\noindent
We consider an open bounded two-dimensional domain $\Omega\subset \mathbb{R}^2$, where there are a few non-overlapping cells denoted by $\Omega_C$. Then the domain $\Omega\setminus\bar{\Omega}_C$ is considered as the extracellular environment; see Figure~\ref{Fig_schematic}(b). The key difference between the spatial exclusion and point source model is whether taking the cell domain $\Omega_C$ as part of the entire domain or not.

The spatial exclusion model excludes the cell domain from the entire domain, and then the secretion of the compounds is described as a Neumann boundary condition over $\partial\Omega_C$. The model is given by
\begin{equation}
	\label{Eq_BVP_hole}
	(\rm BVP_S)\quad\left\{
	\begin{aligned}
	\frac{\partial u_S(\boldsymbol{x},t)}{\partial t} - D\Delta u_S(\boldsymbol{x},t) &= 0, &\mbox{in $\Omega\setminus\bar{\Omega}_C, t>0$,}\\
	D\nabla u_S(\boldsymbol{x},t)\cdot\boldsymbol{n} &= \phi(\boldsymbol{x},t), &\mbox{on $\partial\Omega_C, t>0$,}\\
	D\nabla u_S(\boldsymbol{x},t)\cdot\boldsymbol{n} &= 0, &\mbox{on $\partial\Omega, t>0$,}\\
	u_S(\boldsymbol{x}, 0) &= u_0(\boldsymbol{x}), &\mbox{in $\Omega\setminus\bar{\Omega}_C$,}
	\end{aligned}
	\right.
\end{equation}
where $\boldsymbol{n}$ is the outward pointing unit normal vector to the domain boundary of $\Omega\setminus\bar{\Omega}_C$, see Figure~\ref{Fig_schematic}(b). We assume that the flux density $\phi(\boldsymbol{x},t)$ is non-negative at $\boldsymbol{x}\in\partial\Omega_C$ where there is a flux of compound into the environment $\Omega\setminus\bar{\Omega}_C$. This ensures that solutions with positive initial condition remain positive for all time. In contrast to \citet{Peng2023}, we consider specifically fluxes that are spatially inhomogeneous.

For the numerical simulations, on time domain $[0,T]$, Finite-Element Method (FEM) \citep{van2005numerical} is used. Hence, we consider the weak solution concept for $(\rm BVP_S)$, which is given by 
\begin{equation*}
	    (\rm WF_S)\left\{
	    \begin{aligned}
	    &\text{Find $u_S(\boldsymbol{x},t)\in  L^2\bigl((0,T),H^1(\Omega\setminus\bar{\Omega}_C)\bigr)\cap H^1\bigl((0,T), H^1(\Omega\setminus\bar{\Omega}_C)^*\bigr)$, such that}\\
	    &\int_{\Omega\setminus\bar{\Omega}_C}\frac{\partial u_S(\boldsymbol{x},t)}{\partial t}v_1(\boldsymbol{x},t)d\Omega + \int_{\Omega\setminus\bar{\Omega}_C} D\nabla u_S(\boldsymbol{x},t)\cdot\nabla v_1(\boldsymbol{x},t) d\Omega\\
	    &\qquad - \int_{\partial\Omega_C}\phi(\boldsymbol{x},t) v_1(\boldsymbol{x},t)d\Gamma=0, \text{ for any $v_1(\boldsymbol{x},t)\in  L^2\bigl([0,T], H^1(\Omega\setminus\bar{\Omega}_C)\bigr)$.}
	    \end{aligned}
	\right.
\end{equation*}
(See \citet{Yang-Peng-Hille:2024} Theorem 1).
Here $d\Omega$ is the restriction of Lebesgue measure on $\R^2$ to $\Omega$ and $d\Gamma$ denotes the surface measure on $\partial\Omega_C$, which is so normalized that the Divergence Theorem \citep{keski2022mathematical} holds without additional constant. 

Furthermore, the initial-boundary value problem defined by point sources is given by
\begin{equation}
	\label{Eq_BVP_dirac}
	(\rm BVP_P)\quad\left\{
	\begin{aligned}
	\frac{\partial u_P(\boldsymbol{x},t)}{\partial t} - D\Delta u_P(\boldsymbol{x},t) &= \sum_{i = 1}^{N}\Phi_i(t)\delta(\boldsymbol{x}-\boldsymbol{x}^{(i)}), &\mbox{in $\Omega, t>0$,}\\
	D\nabla u_P\cdot\boldsymbol{n} &= 0, &\mbox{on $\partial\Omega, t>0$,}\\
	u_P(\boldsymbol{x}, 0) &= \bar{u}_0(\boldsymbol{x}), &\mbox{in $\Omega, t=0$.}
	\end{aligned}
	\right.
\end{equation}
Here, the source intensity function $\Phi_i(t)$ describes the flux of mass per unit time from the source at location $\boldsymbol{x}^{(i)}$ when it is positive (otherwise it is a sink).  The index $i$ identifies a point source or sink, of which there will be multiple associated to the same cell in the spatial exclusion model later.

Again, we consider weak solutions. In this case, $u_P(\boldsymbol{x},t)$ cannot be in $H^1(\Omega)$, because of the singularities at the Dirac points. It can be in $W^{1,p}(\Omega)$ for any $p\in(1,2)$, see \citet{Yang-Peng-Hille:2024} for details, Theorem 2 in particular. The restriction of $u_P(\boldsymbol{x},t)$ to $\boldsymbol{x}$ in the environment $\Omega\setminus\Omega_C$ is $H^1$ though \cite{HMEvers2015}. This amounts to the following weak form -- formulated for multiple Dirac points and some $p\in (1,2)$ (with $q$ the usual conjugate exponent: $\displaystyle\frac{1}{p}+\frac{1}{q}=1$): 
\begin{equation*}
    (WF_P)\left\{
    \begin{aligned}
    &\text{Find $u_P(\boldsymbol{x},t)\in  C^0\bigl([0,T], W^{1,p}(\Omega)\bigr)$, such that}\\
    &\int_{\Omega}\frac{\partial u_P(\boldsymbol{x},t)}{\partial t}v_2(\boldsymbol{x},t)d\Omega 
    +  \int_{\Omega} D\nabla u_P(\boldsymbol{x},t)\cdot\nabla v_2(\boldsymbol{x},t) d\Omega \\
    & \qquad =\int_{\Omega}\sum_{i = 1}^N\Phi_i(t)\delta(\boldsymbol{x} - \boldsymbol{x}^{(i)})v_2(\boldsymbol{x},t)d\Omega,\\
    &\text{for any $v_2(\boldsymbol{x},t)\in  L^2\bigl([0,T], W^{1,q}(\Omega)\bigr)$.}
    \end{aligned}
\right.
\end{equation*}
Note, that according to the Sobolev Imbedding Theorem (see e.g. \citet{Adams-Fournier:2003}, Theorem 4.12, p.85) for the indicated $p$, one has $q>2$. Then, $W^{1,q}(\Omega)$ can be viewed as consisting of continuous functions on $\overline{\Omega}$. Hence, evaluation of test functions at Dirac points is well-defined. Existence and uniqueness of (weak) solutions to the point source model for $\bar{u}_0(\boldsymbol{x})\in W^{1,p}(\Omega)$ is furnished by \cite{Amann:2001,Amann-Quittner:2004}; see \citet{Yang-Peng-Hille:2024} for details.

\subsection{Measures for quantitative comparison}
\noindent
In \citet{Peng2023}, a proposition is proven that reveals the relationship between the difference in the two models and the boundary flux over $\partial\Omega_C$, and the proposition has been proved analytically in the setting of FEM. The proposition is stated as follows:
\begin{proposition}
    \label{Prop_Condition}
	Denote by $u_S(\boldsymbol{x},t)$ and $u_P(\boldsymbol{x},t)$ the weak solutions to the spatial exclusion model $(\rm BVP_S)$ and the point source model $(\rm BVP_P)$, respectively, and let $\partial\Omega_C$ be the boundary of the cells, from which the compounds are released, with normal vector $\boldsymbol{n}$ pointing into $\Omega_C$. Then
    \begin{align}
        \frac{1}{2}\frac{d}{dt} \bigl\| u_S-u_P \bigr\|^2_{L^2(\Omega\setminus\Omega_C)}\ &=\ -D \int_{\Omega\setminus\Omega_C} \bigl| \nabla(u_S -u_P) \bigr|^2 d\Omega \label{eq:L2-norm differnce}\\
        &\qquad\qquad +\ \int_{\partial\Omega_C} (u_s-u_P)(\phi-D\nabla u_P\cdot \boldsymbol{n})\,d\Gamma.  \nonumber
    \end{align}
    Assume moreover, that $u_S(\cdot,0)=u_P(\cdot,0)$ a.e. on $\Omega\setminus\Omega_C$. Then, $u_S(\boldsymbol{x},t) = u_P(\boldsymbol{x},t)$ a.e. in $\Omega\setminus\bar{\Omega}_C\times [0,\infty)$ if and only if 
    \begin{equation}
    \label{Eq_PropCondition}
    \phi(\boldsymbol{x},t) - D\nabla u_P(\boldsymbol{x},t)\cdot\boldsymbol{n}= 0, \qquad\mbox{ a.e. on $\partial\Omega_C\times[0,\infty)$.}
    \end{equation}
\end{proposition}

In this manuscript, the measures that we are interested in for comparison of the solutions to the two models are firstly the difference in $L^2-$ and $H^1-$norm on the extracellular environment $\Omega\setminus\Omega_C$. In view of Proposition~\ref{Prop_Condition}, we also consider the difference between prescribed flux density on the cell boundary $\partial \Omega_C$ and flux density on this set in the point source model in $L^2$-norm, i.e. $\bigl\| \,\phi(\cdot,t) - D\nabla u_P(\cdot,t)\cdot\boldsymbol{n}\bigr\|_{L^2(\partial\Omega_c)}$, as a function of time. Moreover, we investigate the total boundary flux deviation $c^*(t)$ over the time interval $[0,t]$, given by
\begin{equation}
    c^*(t) := \int_0^t  \bigl\| \,\phi(\cdot,s) - D\nabla u_P(\cdot,s)\cdot\boldsymbol{n}\bigr\|_{L^2(\partial\Omega_c)}\, ds,
\end{equation}
which was defined and studied in \citet{HMEvers2015}.

In summary, to maintain good agreement between the two solutions in the extracellular environment, the deviation in flux over $\partial\Omega_C$ needs to be minimized. In the point source model the flux over the -- now virtual -- cell boundary is a result of diffusion from the source. Hence, a natural question is, how to define the intensities $\Phi_i(t)$ of the sources at the Dirac points in the point source model such that the flux condition in Equation~\eqref{Eq_PropCondition} can be best approximated.

\section{Approximation with Multiple Dirac Points per Cell}\label{Sec_Multi}
\label{Sec_Amplitude}
\noindent
In \citet{Peng2023}, we replaced a cell with constant homogeneous flux over the boundary in the spatial exclusion model by a single Dirac point in the point source model. As a next step, because in reality cells have an irregular shape and secrete compounds spatially inhomogenously over their membrane, we now consider how to best represent an inhomogenous flux over the boundary by Dirac sources and sinks in a point source model. Thus, in this study, we consider a prescribed spatially inhomogeneous -- but time-independent -- flux density distribution $\phi(\boldsymbol{x},t)=\phi(\boldsymbol{x})$ over the circular boundary of a single cell centred at $\boldsymbol{x}^{(0)}=(x_C,y_C)$ of radius $R$. We assume $\phi(\boldsymbol{x})$ to be non-negative everywhere. 

Locally negative flux density, representing uptake of compound, requires special care. One should realize, that in any realistic model, positivity of solution must be guaranteed. This holds, loosely speaking, when mass can only be taken up where there is mass. Hence, the uptake flux cannot be simply prescribed, but must depend on the solution close to the Dirac point. Hence, a well-designed approach is required. What works `best' in approximating the spatial exclusion model when also uptake of compounds is occurring, is a question for further investigation (started in e.g. \citet{Yang-Peng-Hille:2024}).

Any sufficiently regular function on the circle can be approximated by a superposition of periodic functions by Fourier analysis. Let $\theta\in[0,2\pi)$ denote the angular coordinate on the circular boundary $\partial\Omega_C$ and 
$\boldsymbol{x}_\theta$ the corresponding point, i.e. $\boldsymbol{x}_\theta := \boldsymbol{x}^{(0)} + R(\cos\theta,\sin\theta)\in\partial\Omega_C$.
Hence, in view of Fourier decomposition, we investigate here first fluxes of the 
form 
\begin{equation}
   \label{Eq_phi}
   \phi_n(\boldsymbol{x}_\theta) = \phi_n(\theta) = \phi_0+A\sin(n\theta) = \phi_0(1+\rho\sin(n\theta)),\qquad \theta\in[0,2\pi),
\end{equation}
where $n$ is a positive integer, constants $\phi_0, A>0$ and their ratio $\displaystyle\rho:=\frac{A}{\phi_0}\leqslant 1$ in order to maintain positivity of flux density.

If $n$ is very large, the spatial fluctuation of the flux density are on a very small spatial scale. Then, we expect that diffusion will `smooth out' this variation and the prescribed flux density may be considered homogeneous, effectively. 
Therefore, we are currently interested in small $n$. The question remains to determine from which Fourier mode on one can start neglecting their contribution to $\phi$, depending on $D$ and the spatial size $R$ of the cell.

Furthermore, the ratio $\rho$ plays a significant role: if $\rho$ is very small, then one can still consider using a single Dirac point in the cell centre, since the size of fluctuations is small; however, when $\rho$ is close to $1$, the inhomogeneity is non-negligible. Putting one Dirac point at the cell centre may then not be enough to obtain a good approximation of the spatial exclusion model by the point source model.

{\it\begin{quote}
In order to approximate a prescribed spatially inhomogeneous flux distribution of the form in Equation~\eqref{Eq_phi}, we propose to locate multiple Dirac points in specific configuration within a cell of the spatial exclusion model. We propose intensities for these Dirac points. Under certain conditions, we show that the emergent boundary flux in the point source model as a result of this positioning of the Dirac points and these intensities converges to the predefined flux density of the spatial exclusion model.
\end{quote}}

The idea to obtain approximately an inhomogeneous flux distribution $\phi_n$ on $\partial\Omega_C$ in the point source model is by putting one point source at the centre $\boldsymbol{x}_0$, which realises a spatialy homogenous flux distribution on the boundary. The spatial variation is then superimposed by adding off-centre Dirac points. For the sake of keeping the model simple and subsequent numerical simulations efficient, we want to have as small a number of point sources as possible, while retaining a good approximation of the spatial exclusion model with inhomogeneous flux $\phi_n$ over $\partial\Omega_C$.

Suppose that the additional $N$ off-centre Dirac point sources are located at $\boldsymbol{x}^{(i)}\in\Omega_C$, $i=1,\dots, N$. Let $\Phi_i(t)$ be the intensity of the point source (or sink, if $\Phi_i(t)<0$) at $\boldsymbol{x}^{(i)}$. Notice that in comparison to \citet{Peng2023} we not only allow for off-centre Dirac points, but also for time-dependence in the intensities. We assume the $\Phi_i(t)$ to be locally integrable. If $\Omega=\mathbb{R}^2$ (or $\partial\Omega$ is remote from the cell and $t$ is small such that reflection of mass from the boundary $\partial\Omega$ can still be neglected) the diffusion system defined in~\eqref{Eq_BVP_dirac} with $\bar{u}_0\equiv 0$, yields a flux density over the cell boundary -- for zero initial condition -- that is given for any point $\boldsymbol{x}\in\partial\Omega_C$ and $t>0$ by
\begin{align}
     \phi_P(\boldsymbol{x},t)\ &:=\ D\nabla u_P(\boldsymbol{x},t)\cdot\boldsymbol{n} \nonumber\\ &=\ \sum_{i = 0}^N \int_0^t \frac{\Phi_i(s)}{4\pi D(t-s)}\exp\left\{-\frac{\|\boldsymbol{x} - \boldsymbol{x}^{(i)}\|^2}{4D(t-s)}\right\}\frac{(\boldsymbol{x} - \boldsymbol{x}_C)\cdot (\boldsymbol{x} - \boldsymbol{x}^{(i)})}{2R(t-s)} ds.\label{Eq_bnd_flux_general_s}
\end{align}
Note that $\boldsymbol{n}\ =\ - \frac{1}{R} (\boldsymbol{x}-\boldsymbol{x}_C)$, with the conventions mentioned above.

According to Proposition~\ref{Prop_Condition}, the $N$ intensity functions $\Phi_i(t)$ should be chosen such that the resulting $\phi_P(\boldsymbol{x},t)$ is close to the desired $\phi_n(\boldsymbol{x},t)$.
Perfect agreement cannot hold
at all $\boldsymbol{x}\in\partial\Omega_C$ at all time. Not even if $A=0$, see \citet{Peng2023}. Thus, an approximation of some sort is required. From the practical point of view of numerical simulations, we want to have {\it explicit expressions} for $\Phi_i(t)$ that can be computed well while yielding a good approximation of the corresponding solution of the point source model to that of the spatial exclusion model. 

The convolution integrals in Equation~\eqref{Eq_bnd_flux_general_s} cause concern for the analysis. Typically, one would handle such a system with Fourier or Laplace transform in the time domain. Straightforward application of Fourier transform is obstructed by the constant-in-time flux distribution $\phi_n(\boldsymbol{x})$ not being in $L^2(\R_+)$. Hence, a distributional approach must be taken. Laplace transform can be directly applied to this locally integrable function, though. However, both do not yield (much desired) easily computable explicit expressions. 

To that end, we take an {\it ad hoc} approach.  For any finite $t>0$ (large), one can write -- in various ways:
\begin{equation}\label{eq:Phi deviation mean}
    \Phi_i(s)\ =\ \widetilde{\Phi}_i(t) \ +\ \Delta_i(s;t),\qquad s\in[0,t].
\end{equation}
That is, one can view $\Phi_i(s)$ on $[0,t]$ as a deviation $\Delta_i(s;t)$ from a `mean' $\widetilde{\Phi}_i(t)$. In fact, one could define $\widetilde{\Phi}_i(t) := \frac{1}{t}\int_0^t \Phi_i(s)ds$, but this particular choice is not required. What is important is, that we assume that the splitting can be made such that replacing $\Phi_i(s)$ in Equation~\eqref{Eq_bnd_flux_general_s} by $\Delta_i(s;t)$ results in values for the convolution integral that are small compared to $\hat{\phi}(\boldsymbol{x},t)$, e.g. for $t$ sufficiently large, where $\hat{\phi}(\boldsymbol{x},t)$ is obtained by replacing $\Phi_i(s)$ by $\widetilde{\Phi}_i(t)$ in Equation~\eqref{Eq_bnd_flux_general_s} and performing the integration: 
\begin{equation}\label{eq:expresion tilde Phi}
    \hat{\phi}(\boldsymbol{x},t) := \sum_{i=0}^N \frac{\widetilde{\Phi}_i(t)}{2\pi R}  \frac{(\boldsymbol{x} - \boldsymbol{x}_C)\cdot(\boldsymbol{x}-\boldsymbol{x}^{(i)})}{\|\boldsymbol{x} - \boldsymbol{x}^{(i)}\|^2}   \exp\left\{-\frac{\|\boldsymbol{x} - \boldsymbol{x}^{(i)}\|^2}{4Dt}\right\}.
\end{equation}
We then expect that $\widetilde{\Phi}_i(t)$ captures the `trend' of the behaviour of $\Phi_i(t)$ and -- moreover -- that $\hat{\phi}(\boldsymbol{x},t)$ is an adequate approximation of 
$\phi_P(\boldsymbol{x},t)$ for $t$ sufficiently large. This heuristic argument will be investigated with more mathematical rigour in follow-up work. Here, we employ it to obtain explicit, computable expressions for $\widetilde{\Phi}_i(t)$ by transferring conditions on $\phi_P$ to $\hat{\phi}$, given by Equation~\eqref{eq:expresion tilde Phi}. We then investigate the quality of approximation when using these functions as intensities. We write $\widetilde{\phi}_P(\cdot,t)$ for the flux density over the curve $\partial\Omega_C$ for the solution to the point source model {\it on the bounded domain} $\Omega$ when we take for the intensities $\Phi_i(s)$ in Equation~\eqref{Eq_bnd_flux_general_s} the trends $\widetilde{\Phi}_i(s)$. We compare $\phi(\boldsymbol{x},t)$ to $\widetilde{\phi}_P(\boldsymbol{x},t)$ -- and to $\hat{\phi}(\boldsymbol{x},t)$ in order to check for which $t$ the approximations are appropriate. Note that $\hat{\phi}(\boldsymbol{x},t)$ was computed from the solution on the {\it unbounded} space $\R^2$, using the explicit Green's functions.

Conditions imposed on $\phi_P(\cdot,t)$ vary according to which `metric of comparison' between $\phi_P(\cdot,t)$ and $\phi_n$ is chosen, like the $L^p$-distance over $\partial\Omega_C$, with $p=1,2$ or $\infty$ preferred. These may be transferred to $\widetilde{\phi}_P$. For obtaining explicit expressions, we prefer however requiring that the approximate flux density $\hat{\phi}(\boldsymbol{x},t)$ has the same value at the location of the maxima and minima of $\phi_n$. That is, for $t>0$ we impose
\begin{equation}\label{eq:conditions for tilde_Phi}
    \hat{\phi}(\boldsymbol{x}_{\theta_k},t) = \phi_n(\boldsymbol{x}_{\theta_k}) = \begin{cases} \phi_0 + A, & \mbox{if $k$ is odd},\\
    \phi_0 - A, & \mbox{if $k$ is even},
    \end{cases}
\end{equation}
for $k=1,2,\dots, 2n$. Thus one obtains $2n$ equations for $N+1$ unknown $\widetilde{\Phi}_i(t)$. Thus, with general locations $\boldsymbol{x}_i$ for the point sources in $\Omega_C$ one needs $N=2n-1$ points to be able to solve System~\eqref{eq:conditions for tilde_Phi} for the functions $\widetilde{\Phi}_i(t)$. However, exploiting symmetry by specifically localizing the points allows for doing with fewer points.

The variation in $\phi_n(\boldsymbol{x}_\theta)$ around $\phi_0$, i.e. $A\sin(n\theta)$, takes extreme values at $\theta_k:=(k-\frac{1}{2})\frac{\pi}{n}$, $k=1,2,\dots,2n$ with a maximum at odd $k$ and a minimum at even $k$. We make the choice of locating the off-centre Dirac points each on one of the $n$ line segments connecting the centre $\boldsymbol{x}_0$ to the location $\boldsymbol{x}_{\theta_{k}}$ of the maximum of $\phi_n(\boldsymbol{x})$, for $k$ odd, all at the same distance $r>0$ from the centre; see Figure~\ref{Fig_general_Dirac} for a schematic presentation (for $n=1,2$). By symmetry, the $n$ off-centre Dirac points are expected to have the same intensity, say $\Phi_D(t)$. Denote the intensity of the centre point by $\Phi_C(t)$. Let $\widetilde{\Phi}_D(t)$ and $\widetilde{\Phi}_C(t)$ be respective reference (`mean') values on $[0,t]$ for each, as discussed above. 

Further in this section, we analyse the case of $n=1,2$ and use Equation~\eqref{eq:conditions for tilde_Phi} to determine reference intensities $\widetilde{\Phi}_D(t)$ and $\widetilde{\Phi}_C(t)$. For the general case and more detailed analytical investigation, we preparing a theoretically oriented manuscript.

\subsection{Polarized Flux Distribution: $n=1$}\label{Subsec_1_Dirac}
\noindent
We start with the case $n = 1$, that has a flux distribution that is maximal at $\boldsymbol{x}_\theta$ with $\displaystyle\theta=\frac{\pi}{2}$, and minimal at the opposite point; see Figure~\ref{Fig_general_Dirac}\subref{Fig_1_Dirac}. It represents a `{\it polarized flux distribution}'. In this case, the prescribed flux density in the spatial exclusion model is expressed as 
\begin{equation}
    \label{Eq_flux_den_1}
    \phi(\boldsymbol{x}_\theta,t) = \phi_1(\boldsymbol{x}_\theta) = \phi_0(1 + \rho\sin(\theta)) = \phi_0\bigl(1 + \rho\,\frac{y-y_C}{R}\bigr),\qquad \rho=\frac{A}{\phi_0},
\end{equation}
where $y$ is the $y-$coordinate of any point over the cell boundary, $y_C$ is the $y-$coordinate of the cell center and $R$ is the cell radius. Thus, the cell is divided into two semicircles, of which the upper boundary releases more compounds while the lower boundary releases less. Following the idea in Section~\ref{Sec_Multi}, one off-centre Dirac point is located in the upper semicircle with distance $r$ to the cell centre, see Figure~\ref{Fig_general_Dirac}\subref{Fig_1_Dirac}. The coordinate of the extra off-centre Dirac point is 
\[
    \boldsymbol{x}^{(1)} = (x_C, y_C+r).
\]
\begin{figure}
    \centering
    \subfigure[Case $n=1$]{
        \includegraphics[width = 0.4\textwidth]{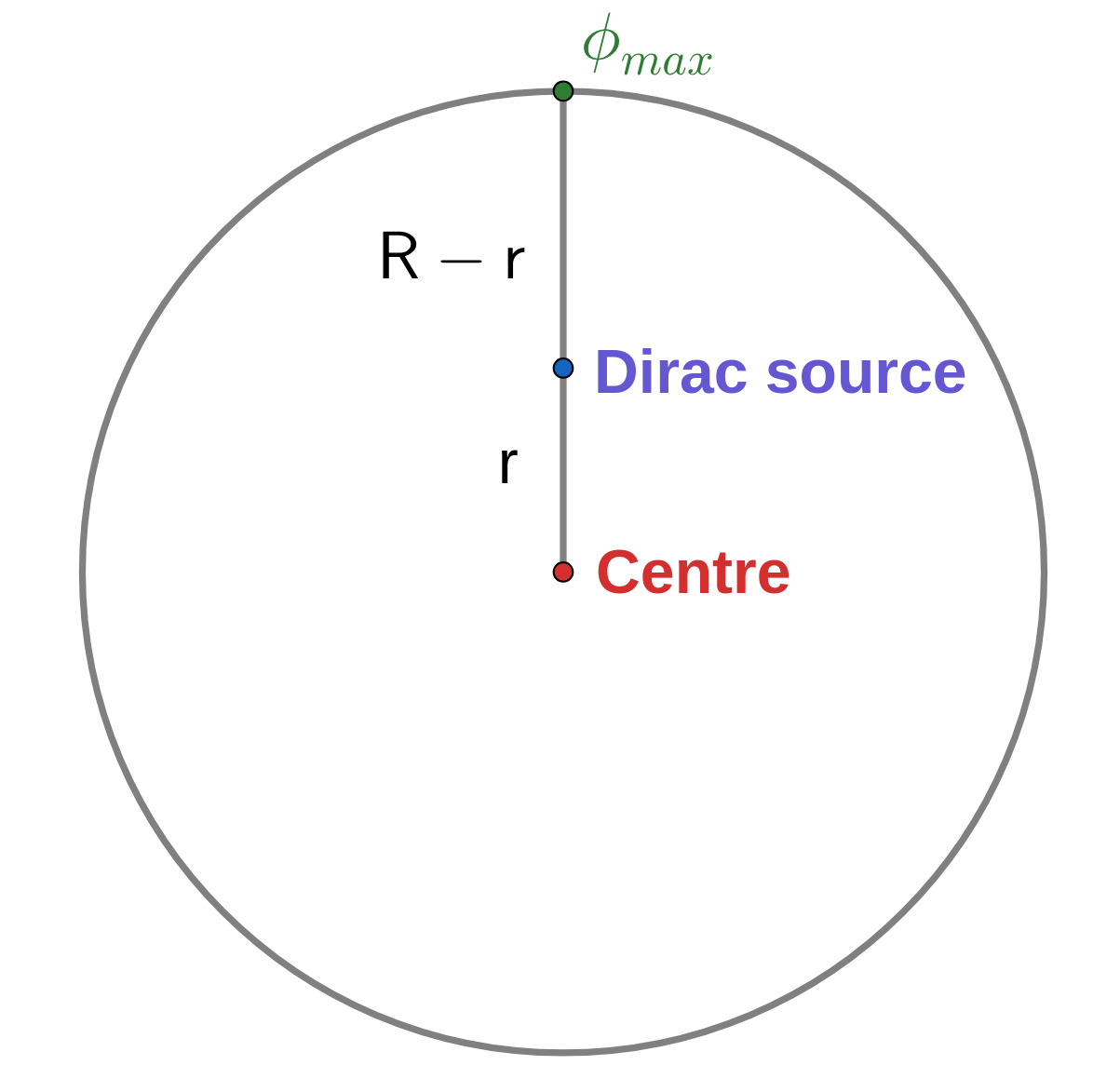}
        \label{Fig_1_Dirac}
    }\hfill
    \subfigure[Case $n=2$]{
        \includegraphics[width = 0.4\textwidth]{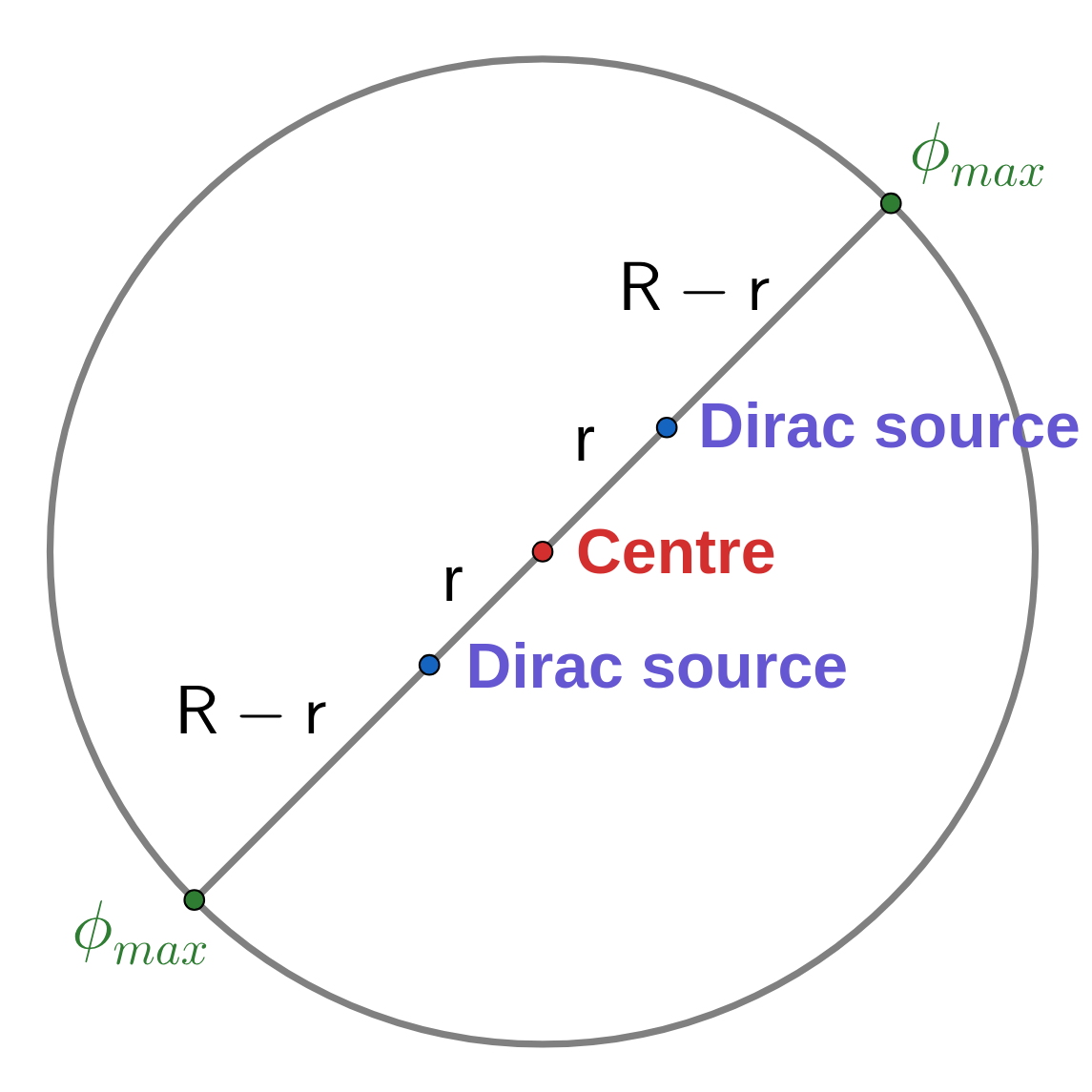}
        \label{Fig_2_Diracs}
    }
    \caption{The locations of the Dirac points for approximating the inhomogenoues flux density $\phi_n(\boldsymbol{x}_\theta)=\phi_0 + A\sin(n\theta)$ over the circular cell boundary of radius $R$ are shown for the cases (a) of a polarized flux distribution($n=1$), and (b) of an axially oriented flux distribution ($n=2$). The off-centre Dirac points (blue) are each on a line segment connecting the cell's centre point (red) to a point on the cell boundary where $\phi_n(\boldsymbol{x})$ attains its maximum value $\phi_{max}$ (green). The distance between the cell centre and any off-centre Dirac point(s) is $r$ ($0<r<R$).}
    \label{Fig_general_Dirac}
\end{figure}
Then, the boundary flux $\hat{\phi}(\boldsymbol{x},t)$ defined in Equation~\eqref{eq:expresion tilde Phi}, becomes 
\begin{align}
    \hat{\phi}(\boldsymbol{x}_\theta,t)\ &=\ \frac{\widetilde{\Phi}_D(t)(R-r\sin(\theta))}{2\pi(R^2+r^2-2rR\sin(\theta))}\exp\left\{-\frac{R^2+r^2-2rR\sin(\theta))}{4Dt}\right\} \nonumber \\
    &\qquad +\ \frac{\widetilde{\Phi}_C(t)}{2\pi R}\exp\left\{-\frac{R^2}{4Dt}\right\}.\label{Eq_flux_1}
\end{align}

To determine the value of $\widetilde{\Phi}_D(t)$ and $\widetilde{\Phi}_C(t)$, we require that for any $t>0$, the minimum and maximum of $\hat{\phi}(\boldsymbol{x}_\theta,t)$ in Equation~\eqref{Eq_flux_1} should be the same as those of $\phi_1(\boldsymbol{x}_\theta)$. In Appendix~\ref{app:same location extrema} it is shown that that for any $t>0$, $\theta\mapsto\hat{\phi}(\boldsymbol{x}_\theta,t)$ attains its maximum at $\displaystyle\theta=\frac{\pi}{2}$ and its minimum at $\displaystyle\theta=\frac{3\pi}{2}$. Hence, we obtain the conditions (for all $t>0$) 
\begin{equation}
    \hat{\phi}\left(\mbox{$\frac{\pi}{2}$},t\right) = \phi_1\left(\mbox{$\frac{\pi}{2}$}\right) = \phi_0+A \qquad\mbox{and}\qquad \hat{\phi}\left(\mbox{$\frac{3\pi}{2}$},t\right) = \phi_1\left(\mbox{$\frac{3\pi}{2}$}\right) = \phi_0-A.
\end{equation}
Then, $\widetilde{\Phi}_D(t)$ and $\widetilde{\Phi}_C(t)$ must satisfy the system
\begin{equation}
\label{Eq_dipole}
    \left\{
    \begin{aligned}
        &\frac{\widetilde{\Phi}_D(t)}{2\pi (R-r)}\exp\left\{-\frac{(R-r)^2}{4Dt}\right\}+\frac{\widetilde{\Phi}_C(t)}{2\pi R}\exp\left\{-\frac{R^2}{4Dt}\right\} = \phi_0+A,\\
        &\frac{\widetilde{\Phi}_D(t)}{2\pi (R+r)}\exp\left\{-\frac{(R+r)^2}{4Dt}\right\}+\frac{\widetilde{\Phi}_C(t)}{2\pi R}\exp\left\{-\frac{R^2}{4Dt}\right\} = \phi_0-A.
    \end{aligned}
    \right.
\end{equation}
The exact solutions to System (\ref{Eq_dipole}) are 
\begin{equation}
    \label{Eq_dipole_sol}
    \left\{
    \begin{aligned}
        \widetilde{\Phi}_D(t) &= \frac{4\pi A}{\frac{1}{R-r}\exp\left\{-\frac{(R-r)^2}{4Dt}\right\}-\frac{1}{R+r}\exp\left\{-\frac{(R+r)^2}{4Dt}\right\}},\\
        \widetilde{\Phi}_C(t) &= 2\pi R\exp\left\{\frac{R^2}{4Dt}\right\}\left(\phi_0+A-\frac{2A(R+r)}{(R+r)-(R-r)\exp\left\{-\frac{Rr}{Dt}\right\}}\right). 
    \end{aligned}
    \right.
\end{equation}
From the expressions, we notice that: (1)  solutions exist for $0<r< R$; (2) if $r = R$, then System~\eqref{Eq_dipole} is not defined, in the sense that not all coefficients are finite. Nevertheless, $\tilde{\Phi}_D(t)\to 0$ as $r\uparrow R$ in Expression~\eqref{Eq_dipole_sol}, but this does not have an apparent meaning; (3) neither of the functions are integrable at $t=0$. So, the functions $\widetilde{\Phi}_C(t)$ and $\widetilde{\Phi}_D(t)$ are not locally integrable on $\R_+$, as was assumed of the intensities $\Phi_i(t)$ at the start; (4) $\widetilde{\Phi}_D(t)$ is always non-negative, while $\widetilde{\Phi}_C(t)$ can be negative. In particular, if $A = \phi_0$ (i.e. $\rho = 1$), then $\widetilde{\Phi}_C(t)$ is negative. Thus, the off-centre Dirac point represents a source, always, while the centre point can be a sink or a source, which depends on the ratio between $\phi_0$ and $A$. In Appendix~\ref{App_Phi_D_Phi_C}, we provide a brief discussion of key characteristics of $\widetilde{\Phi}_D$ and $\widetilde{\Phi}_C$.

Note that the functions $\widetilde{\Phi}_D(t)$ and $\widetilde{\Phi}_C(t)$ converge very rapidly to $\infty$ as $t\downarrow 0$. In particular, they are not locally integrable near $0$, contrary to what is assumed of intensities. In computations, these functions are never evaluated at $0$, or closer to 0 than (halve a) time step $\tau$ in the numerical solver. Thus, effectively, they may be considered as given by Expression~\eqref{Eq_dipole_sol} and `truncated' near 0 by the value at $\tau$. Thus, one obtains locally integrable intensities.

Substituting Expression~\eqref{Eq_dipole_sol} into Equation~\eqref{Eq_flux_1}, we obtain
\begin{equation}
    \label{Eq_1_Dirac_bnd_flux_t}
    \begin{aligned}
        \hat{\phi}_1(\theta,t) &= \phi_0 + A + C_1(t)\exp\left\{-\frac{R^2+r^2-2Rr\sin(\theta)}{4Dt}\right\}\frac{R-r\sin(\theta)}{R^2+r^2-2Rr\sin(\theta)}\\
        &\qquad-\frac{2A(R+r)}{(R+r)-(R-r)\exp\left\{-\frac{Rr}{Dt}\right\}},
    \end{aligned}
\end{equation}
where 
\begin{equation}\label{eq: C1}
    C_1(t) = \frac{2A}{\frac{1}{R-r}\exp\left\{-\frac{(R-r)^2}{4Dt}\right\}-\frac{1}{R+r}\exp\left\{-\frac{(R+r)^2}{4Dt}\right\}}> 0, \quad \mbox{for all}\  r\in(0, R).
\end{equation}

In Figure~\ref{Fig_sine_1_Dirac}(a), we compare the prescribed flux density $\phi(\boldsymbol{x}_\theta) = \phi_1(\boldsymbol{x}_\theta)$ with $\hat{\phi}(\theta,t)=\hat{\phi}_1(\theta,t)$ for various time points. Note that the location and value of the maximum and minimum of $\phi(\boldsymbol{x}_\theta)$ and $\hat{\phi}_1(\theta, t)$ coincide, as expected (see Appendix~\ref{App_proof_1_Dirac_extremum}). Since the minimum of $\hat{\phi}_1(\theta, t)$ equals $\phi_0-A$, which is always non-negative, the positive concentration in the computational domain is ensured. Moreover, with a fixed value of $r$, increasing $t$ reduces the difference between $\hat{\phi}(\theta, t)$ and $\phi$. 

The distance $r$ of the off-centre Dirac point to the centre is a parameter we still can choose in the range $(0,R)$. To investigate its effect on $\hat{\phi}_1(\theta,t)$ we consider now the dependence on $r$ of the limit of this function as `$t\to\infty$'. We have for $0<r<R$:
\begin{equation}
    \label{Eq_flux_stst_1}
    \hat{\phi}_{\infty}(\theta;r)\ :=\  \lim_{t\to\infty} \hat{\phi}_1(\theta,t) \ =\ \phi_0+A-\frac{A(R+r)^2(1-\sin(\theta))}{R^2-2Rr\sin(\theta)+r^2}.
\end{equation}
Figure~\ref{Fig_sine_1_Dirac}(b) plots $\hat{\phi}_\infty$ for multiple values of $r$. Observe that the dashed curve ($\hat{\phi}_\infty$) and the solid curve ($\phi=\phi_1$) overlap even more when $r$ is decreasing. The analytical proof of this observed convergence is given in Appendix~\ref{App_proof_Dirac}. Furthermore, when $r$ decreases, $\widetilde{\Phi}_D$ increases while $\widetilde{\Phi}_C$ decreases without any upper or lower bound. In other words, with $r\rightarrow0^+$, $\widetilde{\Phi}_D(t)\rightarrow+\infty$ and $\widetilde{\Phi}_C(t)\rightarrow-\infty$ for any $t>0$. Thus, the off-centre Dirac point is a source while the centre point is a sink. The pair of Dirac delta distributions then form a system that resembles a dipole system in electromagnetism \citep{Ellingson_2018}. 
\begin{figure}[h!]
    \centering
    \subfigure[Flux over the cell boundary with $r=0.05$ at different $t$]{
    \includegraphics[width = 0.9\textwidth]{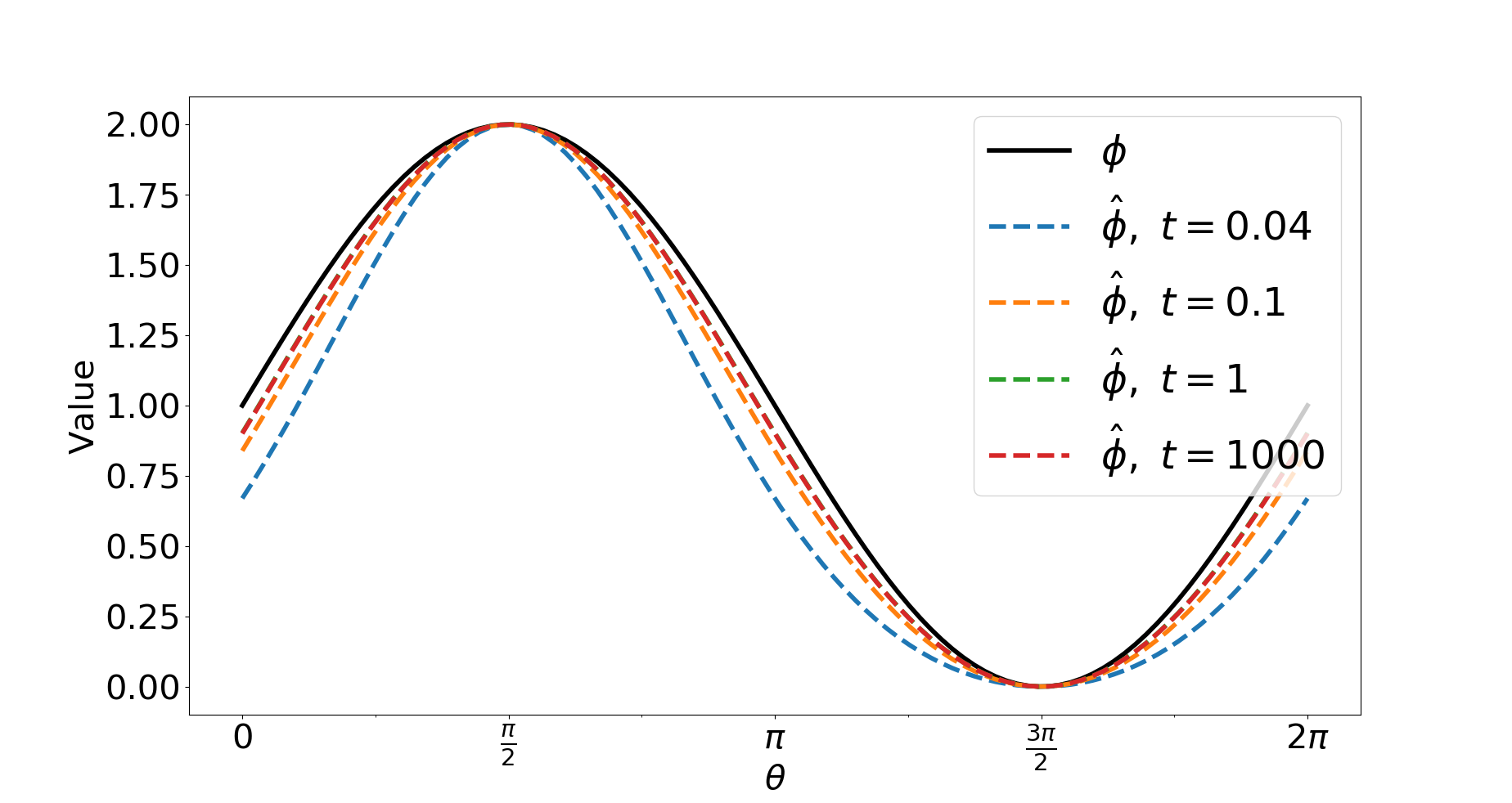}}
    \subfigure[Flux over the cell boundary with various values of $r$ when $t\rightarrow +\infty$]{
    \includegraphics[width = 0.9\textwidth]{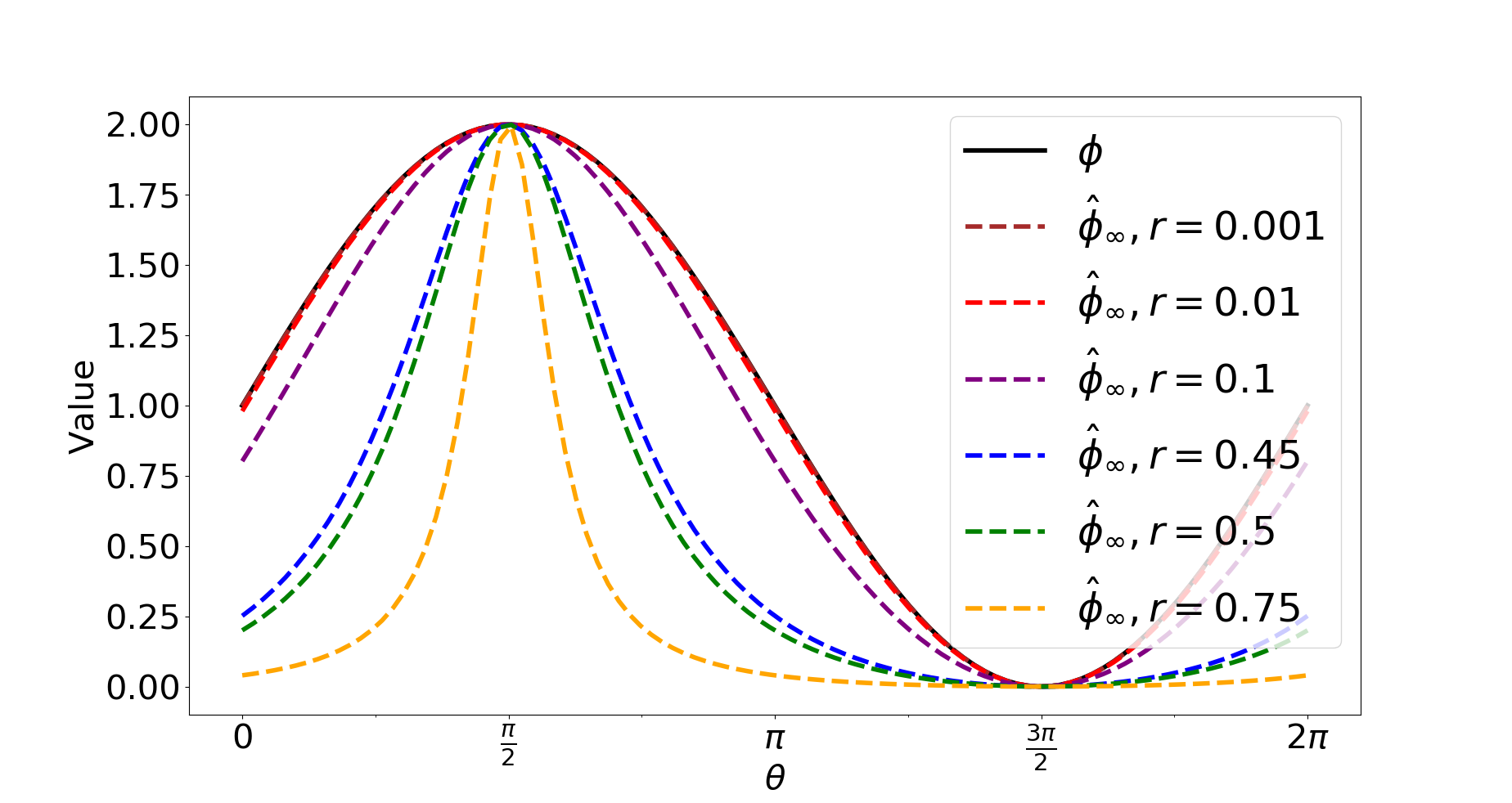}}
    \caption{For $n=1$, $\phi_0 = A = 1.0$ and cell radius $R=1.0$ the desired flux density $\phi(\boldsymbol{x}_\theta) = 1+\sin(\theta)$ is compared as function of the angle $\theta$ with (a) $\widetilde{\phi}(\boldsymbol{x}_\theta,t)$, computed from Equation~\eqref{eq:expresion tilde Phi} for indicated times, and (b) the limit of $\widetilde{\phi}(\boldsymbol{x}_\theta,t)$ as $t\to\infty$, given by Equation~\eqref{Eq_flux_stst_1} for various values of $r$. The solid curve is the prescribed flux density $\phi_1$ in the spatial exclusion approach and the dashed curves show the approximants $\widetilde{\phi}(\cdot,t)$ or $\widetilde{\phi}_\infty$; different colours of the curves represent different values of time $t$ or the distance between the off-centre Dirac point and the cell centre $r$.} 
    \label{Fig_sine_1_Dirac}
\end{figure}

\subsection{Axially Oriented Flux Distribution: $n=2$}
\noindent
When $n=2$, there are two complete periods of the flux density over the cell boundary. Hence, $\theta\mapsto\phi_2(\boldsymbol{x}_\theta)$ has two maxima, at $\displaystyle\theta = \frac{\pi}{4}$ and $\displaystyle\theta = \frac{5\pi}{4}$. Thus, secretion of compound is high around the axis through the two associated points on the boundary and low orthogonal to this line. Hence we call this the case of an `{\it axially oriented flux distribution}'. These two values will also be the angles of the locations of the extra off-centre Dirac points; see Figure~\ref{Fig_general_Dirac}\subref{Fig_2_Diracs} for a schematic presentation. Accordingly, $\boldsymbol{x}^{(0)}=\boldsymbol{x}_C$ will be the centre point, while $\boldsymbol{x}^{(1)}$ corresponds to $\displaystyle\theta=\frac{\pi}{4}$ and $\displaystyle\boldsymbol{x}^{(2)}$ to $\displaystyle\theta=\frac{5\pi}{4}$. The minima are located at $\displaystyle\theta=\frac{3\pi}{4}$ and $\displaystyle\theta=\frac{7\pi}{4}$.

With the chosen location of the Dirac points, a direct computation yields
\begin{equation}
    \frac{(\boldsymbol{x}_\theta - \boldsymbol{x}_C)\cdot (\boldsymbol{x}_\theta - \boldsymbol{x}^{(i)})}{\| \boldsymbol{x}_\theta - \boldsymbol{x}^{(i)} \|^2}\ =\ \left\{ \begin{aligned}
        1,\qquad\qquad\qquad\quad & \quad\mbox{if}\ i=0,\\
        \frac{R^2 - \frac{1}{2}\sqrt{2} rR(\cos\theta+\sin\theta)}{R^2 + r^2 - \sqrt{2} rR(\cos\theta+\sin\theta)}, & \quad\mbox{if}\ i=1,\\
        \frac{R^2 + \frac{1}{2}\sqrt{2} rR(\cos\theta+\sin\theta)}{R^2 + r^2 + \sqrt{2} rR(\cos\theta+\sin\theta)}, & \quad\mbox{if}\ i=2.\\
    \end{aligned} \right.
\end{equation}
The flux density approximation then becomes 

\begin{equation}
\label{Eq_phi_hat_2}
\begin{aligned}
    \hat{\phi}_2(\theta, t)&= \frac{\widetilde{\Phi}_D(t)}{2\pi R}\frac{R^2-\frac{\sqrt{2}}{2}Rr(\sin(\theta)+\cos(\theta))}{R^2+r^2-\sqrt{2}Rr(\sin(\theta)+\cos(\theta))}\exp\left\{-\frac{R^2+r^2-\sqrt{2}Rr(\sin(\theta)+\cos(\theta))}{4Dt}\right\}\\
    &+\frac{\widetilde{\Phi}_D(t)}{2\pi R}\frac{R^2+\frac{\sqrt{2}}{2}Rr(\sin(\theta)+\cos(\theta))}{R^2+r^2+\sqrt{2}Rr(\sin(\theta)+\cos(\theta))}\exp\left\{-\frac{R^2+r^2+\sqrt{2}Rr(\sin(\theta)+\cos(\theta))}{4Dt}\right\}\\
    &+\frac{\widetilde{\Phi}_C(t)}{2\pi R}\exp\left\{-\frac{R^2}{4Dt}\right\}
\end{aligned}
\end{equation}
We impose in this case too the condition that the maxima and minima of the approximate flux density $\hat{\phi}_2(\boldsymbol{x}_\theta,t)$ should have the same location and value as those of $\phi(\boldsymbol{x}_\theta)$ at all times $t>0$. According to Appendix~\ref{app:same location extrema}, the two functions in the given configuration of Dirac points and intensities do have the same fixed location in time of extremes. Hence, we solve $\widetilde{\Phi}_D(t)$ and $\widetilde{\Phi}_C(t)$ from
\begin{equation}
    \label{Eq_2_Dirac_Phi_t}
    \left\{
    \begin{aligned}
    &\frac{\widetilde{\Phi}_D(t)}{2\pi(R-r)}\exp\left\{-\frac{(R-r)^2}{4Dt}\right\}+\frac{\widetilde{\Phi}_D(t)}{2\pi(R+r)}\exp\left\{-\frac{(R+r)^2}{4Dt}\right\}+\frac{\widetilde{\Phi}_C(t)}{2\pi R}\exp\left\{-\frac{R^2}{4Dt}\right\} = \phi_0+A,\\
    &\frac{\widetilde{\Phi}_D(t)R}{\pi(R^2+r^2)}\exp\left\{-\frac{R^2+r^2}{4Dt}\right\}+\frac{\widetilde{\Phi}_C(t)}{2\pi R}\exp\left\{-\frac{R^2}{4Dt}\right\} = \phi_0-A.
    \end{aligned}
    \right.
\end{equation}

The solutions to System~\eqref{Eq_2_Dirac_Phi_t} are given by
\begin{equation}\label{Eq_tripole_sol}
    \widetilde{\Phi}_D(t)= \frac{4\pi A}{ \frac{1}{R-r} \exp\left\{ -\frac{(R-r)^2}{4Dt}\right\} +  \frac{1}{R+r} \exp\left\{ -\frac{(R+r)^2}{4Dt}\right\} - 2\frac{R}{R^2 +r^2} \exp\left\{ -\frac{r^2+R^2}{4Dt}\right\}  }\ =:\ 2\pi C_2(t)
\end{equation}
and
\begin{align}
    &\widetilde{\Phi}_C(t)\ =\ 2\pi R\exp\left\{\frac{R^2}{4Dt}\right\} \nonumber\\ &\quad\times
    \left[  \phi_0-A - \frac{4 A\exp\left\{-\frac{R^2+r^2}{4Dt}\right\}}{\frac{R^2+r^2}{R(R-r)}\exp\left\{-\frac{(R-r)^2}{4Dt}\right\}+\frac{R^2+r^2}{R(R+r)}\exp\left\{-\frac{(R+r)^2}{4Dt}\right\}-2\exp\left\{-\frac{R^2+r^2}{4Dt}\right\}} \right].\label{eq:phiC triple points}
\end{align}
Note that $C_2(t)>0$ for all $r\in(0,R)$.

Then, the approximate boundary flux is given by 
\begin{equation}
    \label{Eq_2_Dirac_bnd_flux_t}
    \begin{aligned}
    \hat{\phi}_2(&\theta,t)\ =\ \phi_0-A\\
    &+\ C_2(t)\left[\frac{R-\frac{\sqrt{2}}{2}r(\sin(\theta)+\cos(\theta))}{R^2+r^2-\sqrt{2}Rr(\sin(\theta)+\cos(\theta))}\exp\left\{-\frac{R^2+r^2-\sqrt{2}Rr(\sin(\theta)+\cos(\theta))}{4Dt}\right\}\right.\\
    &\quad+\ \frac{R+\frac{\sqrt{2}}{2}r(\sin(\theta)+\cos(\theta))}{R^2+r^2+\sqrt{2}Rr(\sin(\theta)+\cos(\theta))}\exp\left\{-\frac{R^2+r^2+\sqrt{2}Rr(\sin(\theta)+\cos(\theta))}{4Dt}\right\} \\
    &\left.\quad\quad-\ \frac{2R}{R^2+r^2}\exp\left\{-\frac{R^2+r^2}{4Dt}\right\}\right] 
\end{aligned}
\end{equation}

\begin{figure}[h!]
    \centering
    \subfigure[Flux over the cell boundary with $r=0.1$ at different $t$]{
    \includegraphics[width = 0.9\textwidth]{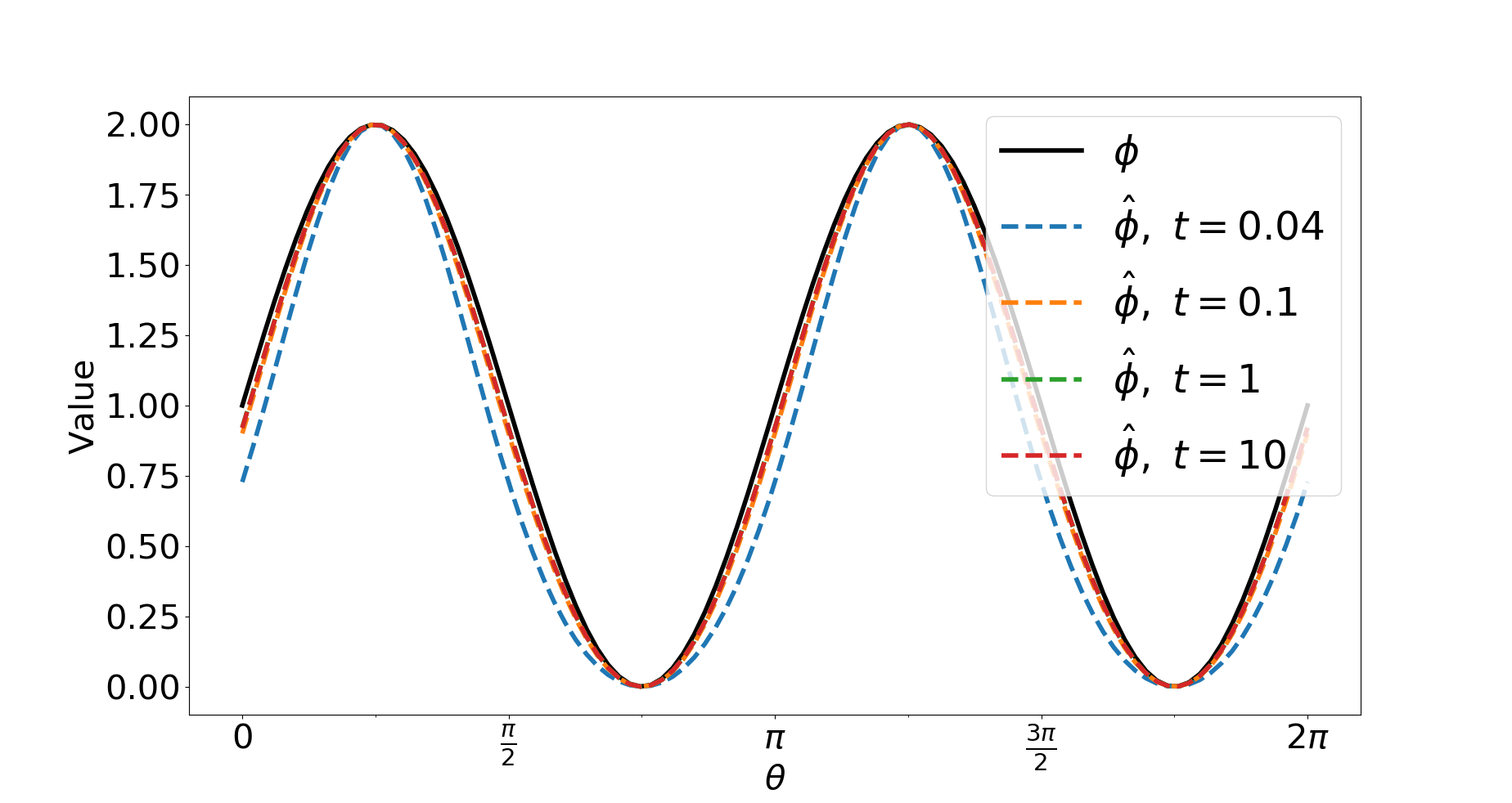}}
    \subfigure[Flux over the cell boundary with various values of $r$ when $t\rightarrow +\infty$]{
    \includegraphics[width = 0.9\textwidth]{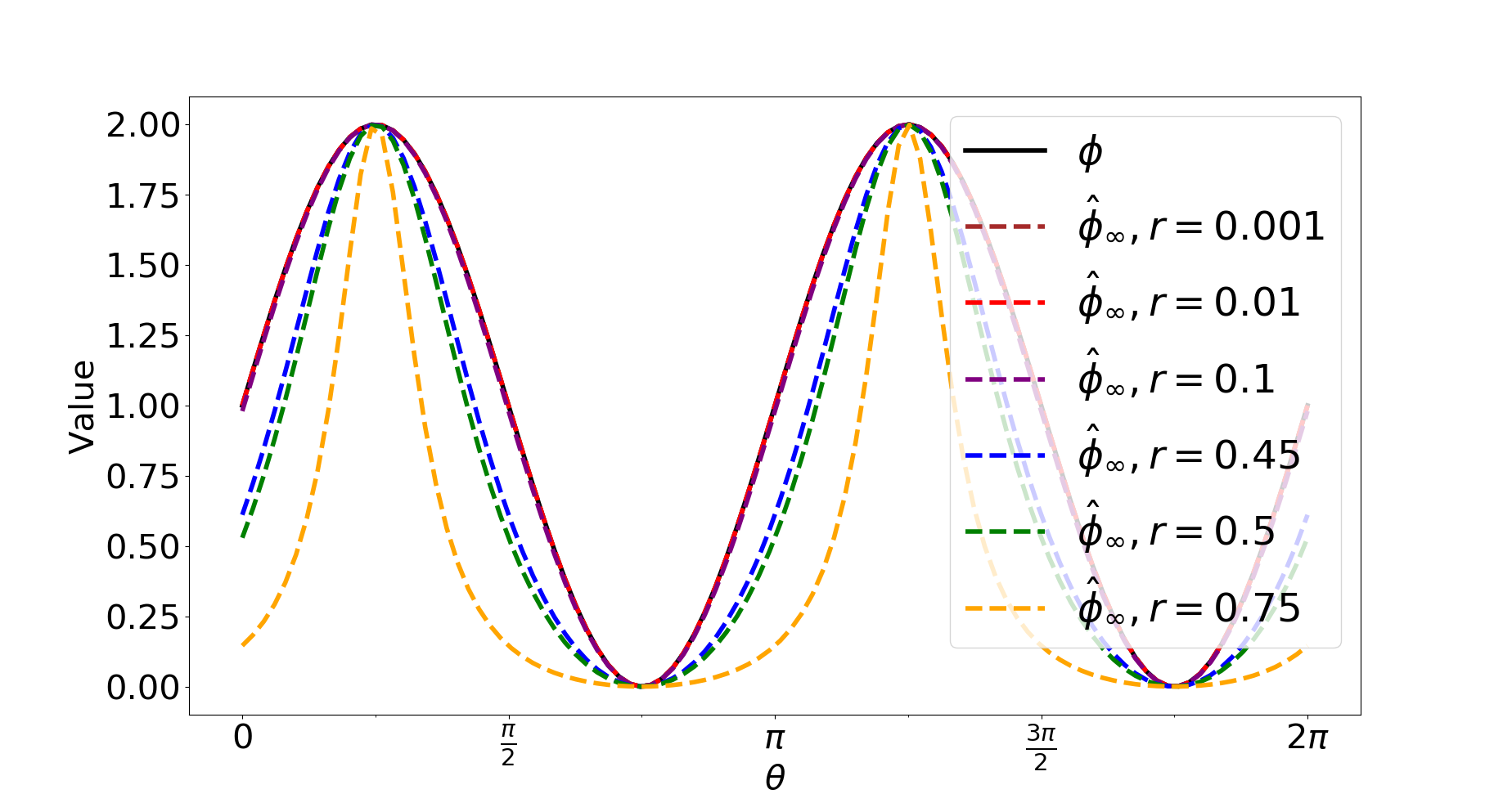}}
    \caption{Selecting $\rho = 1.0$ and cell radius $R=1.0$ for $\phi = 1+\sin(2\theta)$, the flux over the cell boundary is shown versus the angle $\theta$. The solid curve is the predefined flux density in the spatial exclusion approach and the dashed curves are the flux density computed by Equation~\eqref{Eq_2_Dirac_bnd_flux_t} with $n=2$; different colours of the curves represent different values of the distance between the off-centre Dirac point and the cell centre $r$. }
    \label{Fig_sine_2_Dirac}
\end{figure}

In Appendix~\ref{app:same location extrema} we argue that $\hat{\phi}_2$ and $\phi_2$ have the same location for the extreme points that are fixed in time. The imposed conditions assure that their values agree too. Numerically, it is confirmed that $\hat{\phi}_2$ and $\phi_2$ have the same extreme points; see Figure~\ref{Fig_sine_2_Dirac}. 

The approximate boundary flux $\hat{\phi}_2(\theta,t)$ has limit as $t\to\infty$ given by
\begin{equation}
    \label{Eq_2_Dirac_bnd_flux_stst}
    \hat{\phi}_{2,\infty}(\theta;r) = \phi_0-A-\frac{A(R^2-r^2)}{r^2}\ +\ \frac{A(R^2-r^2)(R^2+r^2)(R^2-r^2\sin(2\theta))}{r^2(R^4+r^4-2R^2r^2\sin(2\theta))}
\end{equation}
As can be seen in Figure~\ref{Fig_sine_2_Dirac}(b), as $r\rightarrow0^+$, $\hat{\phi}_{2,\infty}$ converges to $\phi_2(\theta) = 1+\sin(2\theta)$ point-wise. The analytical proof can again be found in Appendix~\ref{App_proof_Dirac}.

\section{Computational Set-up}
\label{Sec_Comp_Setup}
To visualize the deviation between the two approaches with inhomogeneous flux density, numerical simulations are needed. 
In this manuscript, finite-element methods (FEM) and backward Euler are used for the numerical simulations, for the spatial discretization and time integration respectively. Particularly, we use Python 3.10 and \texttt{FEniCS} package \citep{AlnaesEtal2015} version 2019.2.0.dev0. We bear in mind that in the implementation, instead of a smooth circle, the cell region is constructed by a series of mesh points as a polygon. 

The spatial exclusion model is well-defined in the FEM setting (see $(\rm BVP_S)$). In a FEM implementation of the point source model $(\rm BVP_P)$ the singular Delta distributions appear as point evaluation of the basis functions in the FEM at the Dirac points and multiplication with the corresponding source's intensity (see $(\rm WF_P)$). However, this singularity may lead to an unstable and unreliable numerical solution, particularly with a large intensity of a Dirac delta distribution. To counter this, we exploit the availability of the explicit Green's function on $\R^2$ and replace $(\rm BVP_P)$ by a diffusion equation on $\R^2$ with singular Dirac distributions in the reaction term, coupled to an appropriate correction on the bounded domain $\Omega$. The first now has a semi-explicit representation formula in terms of an integral, while the correction is given by the solution to a diffusion equation on $\Omega$ without reaction term and with regular data in its boundary conditions. We will call this the `{\it explicit Green's function approach}'. Details are given below. It is a common singularity removal technique \citep{Gjerde_2019, Lowry_1989, peng2019_DIAM_19_02}. 

It is our objective to quantitatively and qualitatively compare the spatial exclusion and point source model in different parameter regimes. Therefore, we use non-dimensionalized descriptions for both, such that only `essential (combinations of) parameters' appear in the equations. The explicit scaling is given in Section~\ref{subsec:non-dimensional models}.

\subsection{Explicit Green's function approach}\label{Sec_Green_approach}

In Section~\ref{Sec_Amplitude}, we showed a way to determine location and intensity of Dirac point sources, such that the resulting flux density over $\partial\Omega_C$ approximates the prescribed $\phi(\boldsymbol{x},t)$. In that approach, the centre point has positive intensity $\tilde{\Phi}_C(t)$ at time $t$ and the off-centre Dirac points have equal intensity $\tilde{\Phi}_D(t)$, which may not have a fixed sign. The system is similar to an electromagnetic dipole \citep{griffiths2005introduction}. 

Theoretically, to obtain a small discrepancy between the spatial exclusion approach and the point source approach, we need a small $r>0$, which will result in large absolute values of $\widetilde{\Phi}_C$ and $\widetilde{\Phi}_D$. Then, the numerical simulation with FEM might become unstable.

To resolve this issue, we remove the singularity of the Dirac delta distribution from the reaction term from the PDE that is solved numerically, in the following way. The explicit expression for the Green's function for the diffusion equation on $\R^2$ yields an integral representation for the solution of the equation 
\begin{equation}
    \label{Eq_unbnd_diffusion}
    \left\{
    \begin{aligned}
    \frac{\partial\hat{u}(\boldsymbol{x},t)}{\partial t}-D\Delta\hat{u}(\boldsymbol{x},t) &= \sum_{i=1}^{N} \widetilde{\Phi}_i(t) \delta(\boldsymbol{x} - \boldsymbol{x}^{(i)}), \qquad &\mbox{in}\ \R^2,\ t>0,\\
    \hat{u}(\boldsymbol{x},0) &= 0, \qquad &\mbox{in}\ \R^2,\ t=0,
    \end{aligned}
    \right.
\end{equation}
where $\boldsymbol{x}^{(i)}$ is the locations of the Dirac points, $\widetilde{\Phi}_i$ is the intensity of the Dirac point $\boldsymbol{x}_i$. In fact,
\begin{equation}
    \label{Eq_Green}
    \hat{u}(\boldsymbol{x},t) = \int_0^t\frac{1}{4\pi D(t-s)}\sum_{i=1}^{N} \widetilde{\Phi}_i(s)\exp\left\{-\frac{\|\boldsymbol{x}-\boldsymbol{x}^{(i)}\|^2}{4D(t-s)}\right\}ds, \qquad\mbox{for}\ \boldsymbol{x}\in\R^2,\ t\geq 0.
\end{equation}
Note that $\hat{u}(\boldsymbol{x},t)$ can be computed from Expression~\eqref{Eq_Green}, employing appropriate techniques -- not involving FEM -- that ensure its accuracy, even for high intensities.

Let $u_P(\boldsymbol{x},t)$ be the solution to the boundary value problem of the point source approach in Equation~\eqref{Eq_BVP_dirac} with zero initial conditions. Then we may consider $u_P$ as a correction to 
$\hat{u}$, due to reflection of `mass' back into the domain $\Omega$ at the boundary $\partial\Omega$. Define this correction by
\[
    v(\boldsymbol{x},t):=u_P(\boldsymbol{x},t)-\hat{u}(\boldsymbol{x},t).
\]
Since $u_P(\boldsymbol{x},t)=v(\boldsymbol{x},t)+\hat{u}(\boldsymbol{x},t)$, substituting this into $(\rm BVP_P)$, we conclude that $v(\boldsymbol{x},t)$ solves the boundary value problem
\begin{equation}
    \label{Eq_BVP_v}
    (\rm BVP_v)\left\{
    \begin{aligned}
    \frac{\partial v(\boldsymbol{x},t)}{\partial t} - D\Delta v(\boldsymbol{x},t) &= 0, &\mbox{in $\Omega, t>0$,}\\
	D\nabla v\cdot\boldsymbol{n} &= -D\nabla \hat{u}\cdot\boldsymbol{n}, &\mbox{on $\partial\Omega, t>0$,}\\
	v(\boldsymbol{x}, 0) &= \bar{u}_0(\boldsymbol{x}), &\mbox{in $\Omega, t=0$.}
	\end{aligned}
	\right.
\end{equation}

Hence, the solution to $(\rm BVP_P)$ can be `post-processed' after first determining $\hat{u}(\boldsymbol{x},t)$, then solving for $v(\boldsymbol{x},t)$ defined by $(\rm BVP_v)$ in which there is no singular reaction term -- with $\hat{u}$ now given and appearing in the boundary condition -- by means of e.g. FEM without a singular reaction term and then, finally, summing the two contributions. This overall step-wise approach we call the {\it explicit Green's function approach}.

The weak form of $(\rm BVP_v)$, used in the numerical simulations, is given by
\begin{equation*}
(\rm WF_v)\left\{
 \begin{aligned}
    &\text{Find $v\in H^1(\Omega)$, such that}\\
    &\int_\Omega\frac{\partial v}{\partial t}\varphi + D\nabla v\cdot \nabla\varphi d\Omega+\int_{\partial\Omega}D\varphi\nabla\hat{u}\cdot\boldsymbol{n}d\Gamma = 0, \text{ for any $\varphi\in H^1(\Omega)$.}\\
\end{aligned}  
    \right.
\end{equation*}

\subsection{Non-dimensional models}\label{subsec:non-dimensional models}

\noindent
The scalings for non-dimensionalisation are similar to those used in \citet{Peng2023}. The provided description is for the general case of any positive integer $n$ in Equation~\eqref{Eq_phi}. The rescalings are:
\begin{align*}
    \boldsymbol{\xi} := \frac{\boldsymbol{x}}{R},\ \tau := \frac{t}{\tau_0},\ \hat{D}:= \frac{D\tau_0}{R^2},\ \gamma:=\frac{u}{u^*},\ \frac{\phi_0\tau_0}{Ru^*} = 1,
\end{align*}
where the time-scale $\tau_0$ is chosen in relation to the chosen reference density $u^*$, such that the last parameter combination, the `non-dimensional flux density' becomes $1$. $R$ is the cell radius in the dimensional setting.

With this scaling the boundary value problem $(\rm BVP_S)$ becomes
\begin{equation*}
	(\rm BVP'_S)\left\{
	\begin{aligned}
	\frac{\partial \gamma_S(\boldsymbol{\xi},\tau)}{\partial \tau} - \hat{D}\Delta_{\boldsymbol{\xi}} \gamma_S(\boldsymbol{\xi},\tau) &= 0, &\mbox{in $\hat{\Omega}\setminus\bar{\hat{\Omega}}_C, \tau>0$,}\\
	\hat{D}\nabla_{\boldsymbol{\xi}} \gamma_S(\boldsymbol{\xi},\tau)\cdot\boldsymbol{n}_{\boldsymbol{\xi}} &= 1+\rho\sin(n\theta), &\mbox{on $\partial\hat{\Omega}_C, \tau>0$,}\\
	\hat{D}\nabla_{\boldsymbol{\xi}} \gamma_S(\boldsymbol{\xi},\tau)\cdot\boldsymbol{n}_{\boldsymbol{\xi}} &= 0, &\mbox{on $\partial\hat{\Omega}, \tau>0$,}\\
	\gamma_S(\boldsymbol{\xi}, 0) &= \frac{u_0(\boldsymbol{\xi})}{u^*}, &\mbox{in $\hat{\Omega}\setminus\bar{\hat{\Omega}}_C$.}
	\end{aligned}
	\right.
\end{equation*}
Similarly, we obtain the dimensionless system of $(\rm BVP_P)$, which reads as 
\begin{equation*}
	(\rm BVP'_P)\left\{
	\begin{aligned}
	\frac{\partial \gamma_P(\boldsymbol{\xi},\tau)}{\partial \tau} - \hat{D}\Delta_{\boldsymbol{\xi}} \gamma_P(\boldsymbol{\xi},\tau) &= \hat{\widetilde{\Phi}}_C\delta(\boldsymbol{\xi}-\boldsymbol{\xi}_C)+\sum_{i = 1}^{n}\hat{\widetilde{\Phi}}_D\delta(\boldsymbol{\xi}-\boldsymbol{\xi}_i), &\mbox{in $\hat{\Omega}, \tau>0$,}\\
	\hat{D}\nabla_{\boldsymbol{\xi}} \gamma_P\cdot\boldsymbol{n}_{\boldsymbol{\xi}} &= 0, &\mbox{on $\partial\hat{\Omega}, \tau>0$,}\\
	\gamma_P(\boldsymbol{\xi}, 0) &= \frac{\bar{u}_0(\boldsymbol{\xi})}{u^*}, &\mbox{in $\hat{\Omega}$.}
	\end{aligned}
	\right.
\end{equation*}
In the explicit Green's function approach, the dimensionless semi-explicit solution on $\R^2$ is then $$\hat{\gamma} = \frac{\hat{u}}{u^*}.$$ The dimensionless boundary value problem for the correction term becomes
\begin{equation*}
    (\rm BVP'_v)\left\{
    \begin{aligned}
    \frac{\partial \gamma_v(\boldsymbol{\xi},\tau)}{\partial \tau} - \hat{D}\Delta_{\boldsymbol{\xi}} \gamma_v(\boldsymbol{\xi},\tau) &= 0, &\mbox{in $\hat{\Omega}, \tau>0$,}\\	\hat{D}\nabla_{\boldsymbol{\xi}}\gamma_v\cdot\boldsymbol{n}_{\boldsymbol{\xi}} &= -\hat{D}\nabla_{\boldsymbol{\xi}}\hat{\gamma}\cdot\boldsymbol{n}_{\boldsymbol{\xi}}, &\mbox{on $\partial\hat{\Omega}, \tau>0$,}\\
	\gamma_v(\boldsymbol{\xi}, 0) &= 0, &\mbox{in $\hat{\Omega}, \tau=0$.}
	\end{aligned}
	\right.
\end{equation*}
`Post-processing' yields the rescaled solution on $\hat{\Omega}$, $\gamma_P:= u_P/u^*$, given by $$\gamma_P(\boldsymbol{\xi},\tau) = \hat{\gamma}(\boldsymbol{\xi},\tau) + \gamma_v(\boldsymbol{\xi},\tau).$$

The numerical simulations have been conducted in this dimensionless setting. The dimensionless parameter values that were used are listed in Table~\ref{tab:para_all}, unless indicated otherwise. 

\begin{table}\footnotesize
    \centering
    \caption{Standard parameter values used in  numerical analyses in Section~\ref{Sec_Homogeneity} and~\ref{Sec_Results}, corresponding to the dimensionless systems presented in Sections~\ref{subsec:non-dimensional models}.} 
    \begin{tabular}{m{2cm}m{2cm}m{11cm}}
    \toprule
    {\bf Parameter} & {\bf Value} & {\bf Description}  \\
    \midrule
    $\hat{D}$ & $1$ & Diffusion coefficient \\
    $L/R$ & $10$ & Size of the computational domain\\
    $\Delta \tau$ & $0.04$ & Time step \\
    $T$ & $40$ & Total time\\
    $h$ & $0.0875$ & Average mesh size\\
    $\boldsymbol{x}^{(0)}$ & $(0,0)$ & Cell centre \\
    $r$ & $0.01$ & Distance between the off-centre Dirac points and the cell centre\\
    $u_0(\boldsymbol{x})$ & $0$ & Initial condition in $\mathrm{(BVP_S)}$ in $\Omega\setminus\Omega_C$\\
    $\bar{u}_0(\boldsymbol{x})$ & $0$ & Initial condition in $\mathrm{(BVP_P)}$ in $\Omega$ and $\mathbb{R}^2$ in explicit Green's function approach \\
    \bottomrule
    \end{tabular}
    \label{tab:para_all}
\end{table}

\noindent From this point onwards, we abuse notation and return to the notationally more convenient dimensional versions of variables and parameters, considering them as non-dimensional, though.

\section{Comparison of Approaches}
In \citet{Peng2023}, we compared a spatial exclusion model with constant (in time and space) flux distribution over the boundary with a point source model with single Dirac point per cell, with constant-in-time intensity. We found then, that in a setting with small diffusivity a systematic time delay occurs. It was investigated how a choice of initial condition could reduce this delay and hence error in the approximation. Here, we shall not take an initial condition other than the constant 0 on the full domain. However, we employ time-varying intensities for the Dirac delta distributions. As will be shown, this will help in resolving the time delay too.

In this section, we shall investigate, to what extent more complicated versions of spatial exclusion or point source model can be approximated well by simpler versions. `Simpler' may mean: e.g. ignoring spatial inhomogeneity of prescribed the flux density over the boundary, ignoring higher frequency variations in space in this density, or using a smaller number of Dirac points in a point source model.

\subsection{Significance of Heterogeneity in Flux Density}\label{Sec_Homogeneity}
\noindent
Intuitively, when $\rho$ (i.e. the ratio between $A$ and $\phi_0$) is very small then the inhomogeneity is negligible, when there is sufficiently strong diffusion. The diffusion of compounds will then quickly level-off density differences, in particular those originating from the inhomogeneity in flux density. Then, we expect that it is not needed to consider spatial variation in the flux density even in the spatial exclusion model. Hence, the use of multiple Dirac points to represent the cell would not be needed as well. However, when the diffusion coefficient is small spatial variation must be taken into account and the use of multiple Dirac points may be required to get a good approximation by a point source model.

In Figure~\ref{Fig_hole_comp_D_1}, we investigate the impact of the ratio $\rho$ on the solution to $(\rm BVP_S)$ when the inhomogeneous flux density is used. We show the relative error between the solution with homogeneous flux density and solutions with inhomogeneous flux density with the same spatial average, but different amplitude and frequency of variations around this mean. It provides an indicator of the level of spatial difference in the comparison of the solutions in the extracellular domain over time. This inhomogeneity indicator is defined as a relative deviation between the solutions to the spatial exclusion model using the homogeneous and inhomogeneous flux density, respectively:
\begin{equation}
    \label{Eq_hole_comp_rela_error}
    \mathcal{H}(t) = \frac{\|u_S^{homo}(t) - u_S^{inhomo}(t)\|}{\|u_S^{homo}(t)\|},
\end{equation}
where $u^{homo}_S$ and $u^{inhomo}_S$ are the solutions to $(\rm BVP_S)$ with homogeneous and inhomogeneous flux density, respectively. The ratio $\rho$ here is either $0.001$ (blue curve) or $1$ (red curve). As expected, a small $\rho$ indicates a very small fluctuation of the flux density, thus, the inhomogeneity can be neglected here. As a result, instead of using multiple Dirac points to describe the inhomogeneous flux over the cell boundary, we can keep using the cell centre to represent the entire cell as the source of the compounds, when the $\rho$ is small. Furthermore, when the inhomogeneity is not negligible, for instance, when $\rho = 1$, setting the source point only at the cell centre is not enough, in particular, in the early stage of the time domain. Nevertheless, as time proceeds and $n$ increases, the inhomogeneity indicated by $\mathcal{H}(t)$ is decaying significantly. 
\begin{figure}[h!]
    \centering
    \subfigure[$\displaystyle\mathcal{H}_{L^2} = \frac{\|u^{homo}_S(\boldsymbol{x}) - u^{inhomo}_S\|_{L^2(\Omega\setminus\Omega_C)}}{\|u^{homo}_S(\boldsymbol{x})\|_{L^2(\Omega\setminus\Omega_C)}}$]{
    \includegraphics[width = 0.48\textwidth]{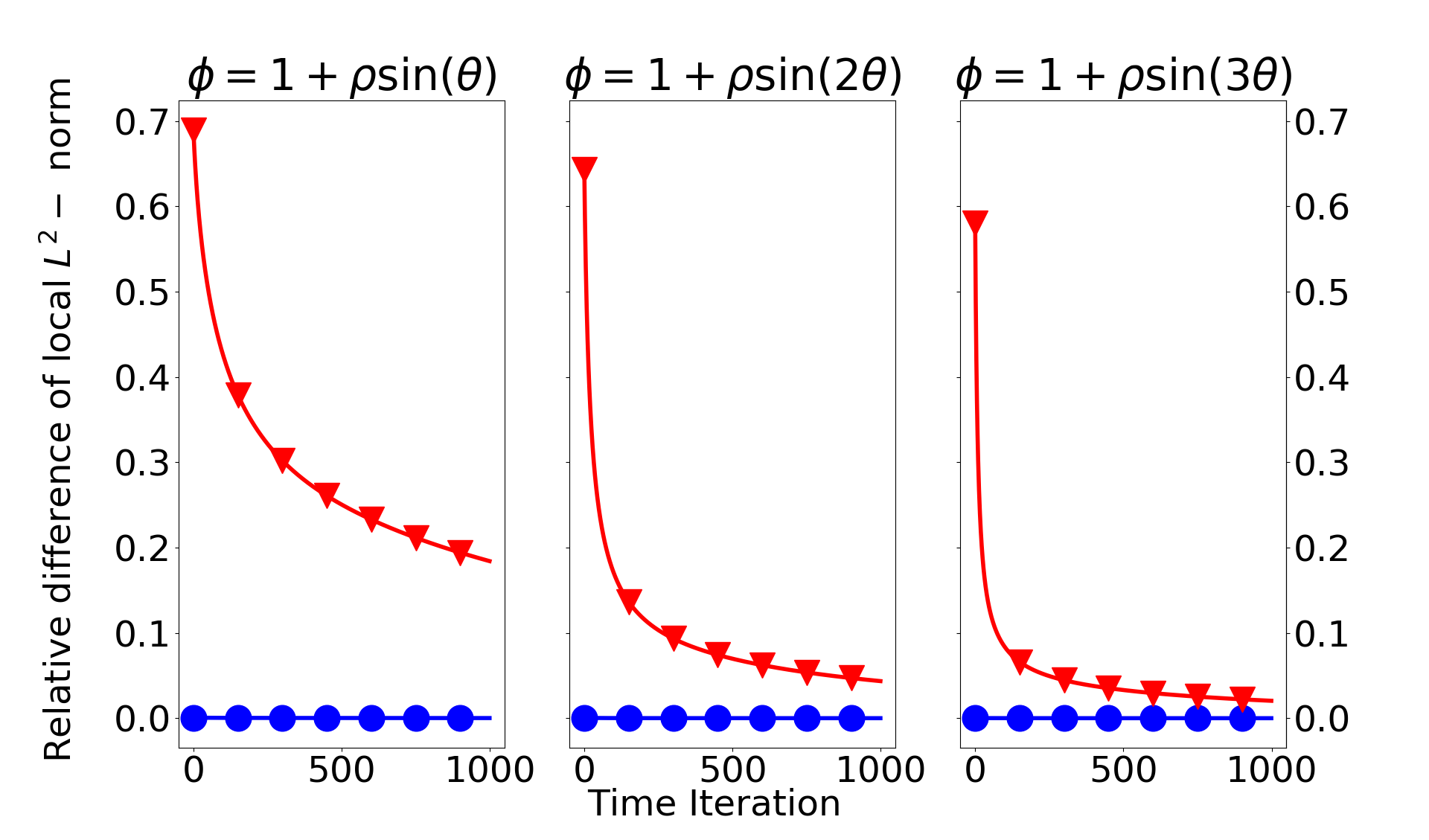}}
    \subfigure[$\displaystyle\mathcal{H}_{H^1}=\frac{\|u^{homo}_S(\boldsymbol{x}) - u^{inhomo}_S\|_{H^1(\Omega\setminus\Omega_C)}}{\|u^{homo}_S(\boldsymbol{x})\|_{H^1(\Omega\setminus\Omega_C)}}$]{
    \includegraphics[width = 0.48\textwidth]{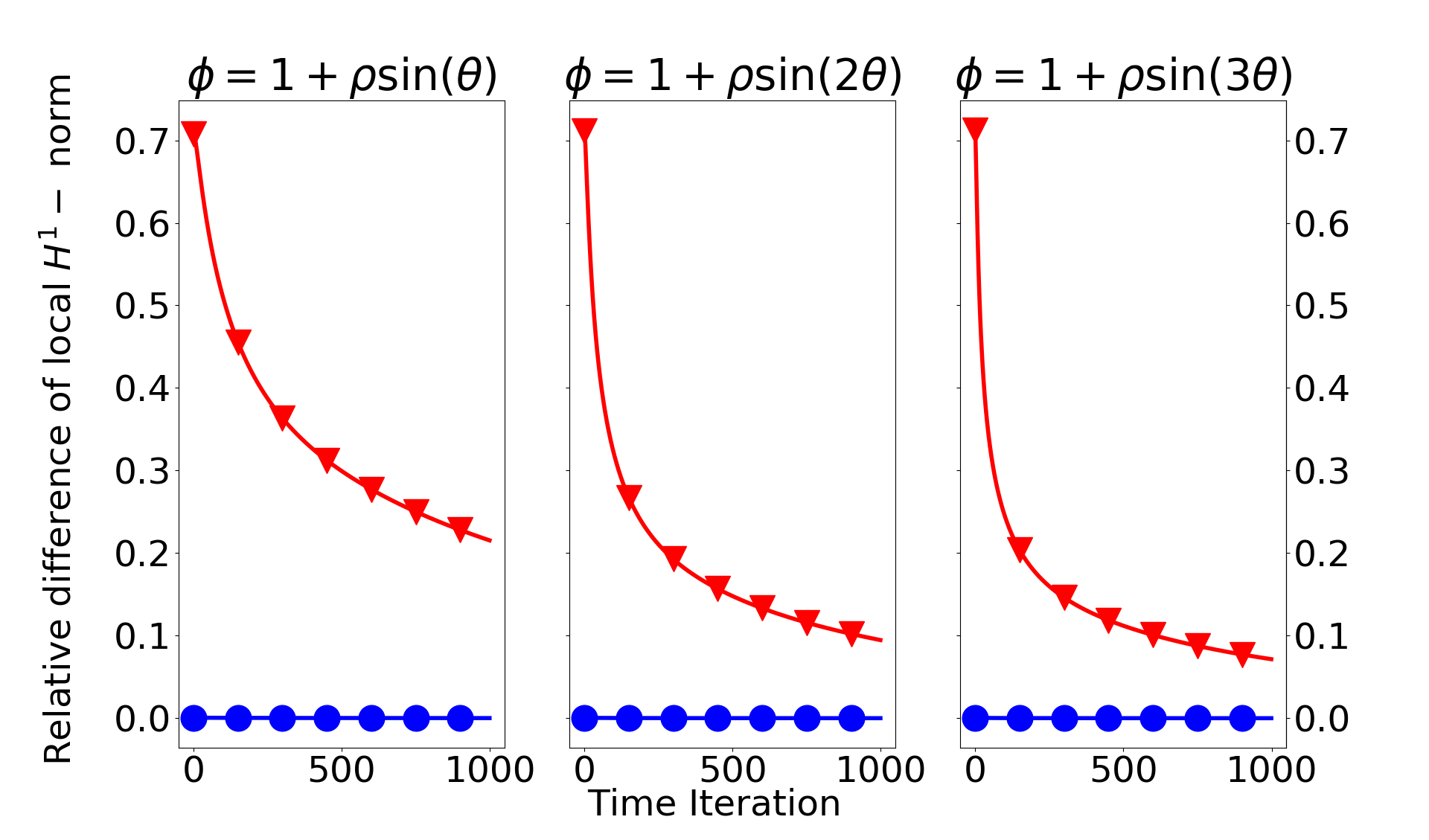}}
    \includegraphics[width = 0.3\textwidth]{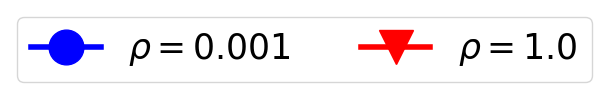}
    \caption{The local norm difference between the solution to $(\rm BVP_S)$ using the homogeneous flux density $\phi(\boldsymbol{x}) = 1$ and the inhomogeneous flux density $\phi(\boldsymbol{x}) = 1+\rho\sin(n\theta)$, where $n\in\{1,2,3\}$ and $\rho = 0.001, 1$. Here, we show the relative $L^2-$ and $H^1-$norm difference in Panel (a) and (b), respectively. In simulations, standard parameter values have been used from Table~\ref{tab:para_all}.}
    \label{Fig_hole_comp_D_1}
\end{figure}

When the frequency of the predefined flux density is large and the diffusion coefficient is relatively large, the inhomogeneity is also negligible. In Figure~\ref{Fig_hole_comp_ratio_1}, when $\rho$ is chosen to be $1$, then the diffusion coefficient $D$ and the parameter $n$ that determines the frequency of the fluctuation, are varied. According to the simulation results presented in Figure~\ref{Fig_hole_comp_ratio_1}, it can be seen that both a large $D$ and a large $n$ 
result in a small relative deviation. When $n$ is large, the gap between the point of high and low flux density is small, i.e. the distance between the location on the boundary of minimal and maximal flux density. With the same value of $D$, a smaller gap demands a shorter time for the homogenization. Furthermore, a larger diffusion coefficient accelerates this diffusing process in the gap. As a result, the inhomogeneity becomes less significant for a large value of $D$ and $n$.   
\begin{figure}[h!]
    \centering
    \subfigure[$\displaystyle\mathcal{H}_{L^2}=\frac{\|u^{homo}_S(\boldsymbol{x}) - u^{inhomo}_S\|_{L^2(\Omega\setminus\Omega_C)}}{\|u^{homo}_S(\boldsymbol{x})\|_{L^2(\Omega\setminus\Omega_C)}}$]{
    \includegraphics[width = 0.48\textwidth]{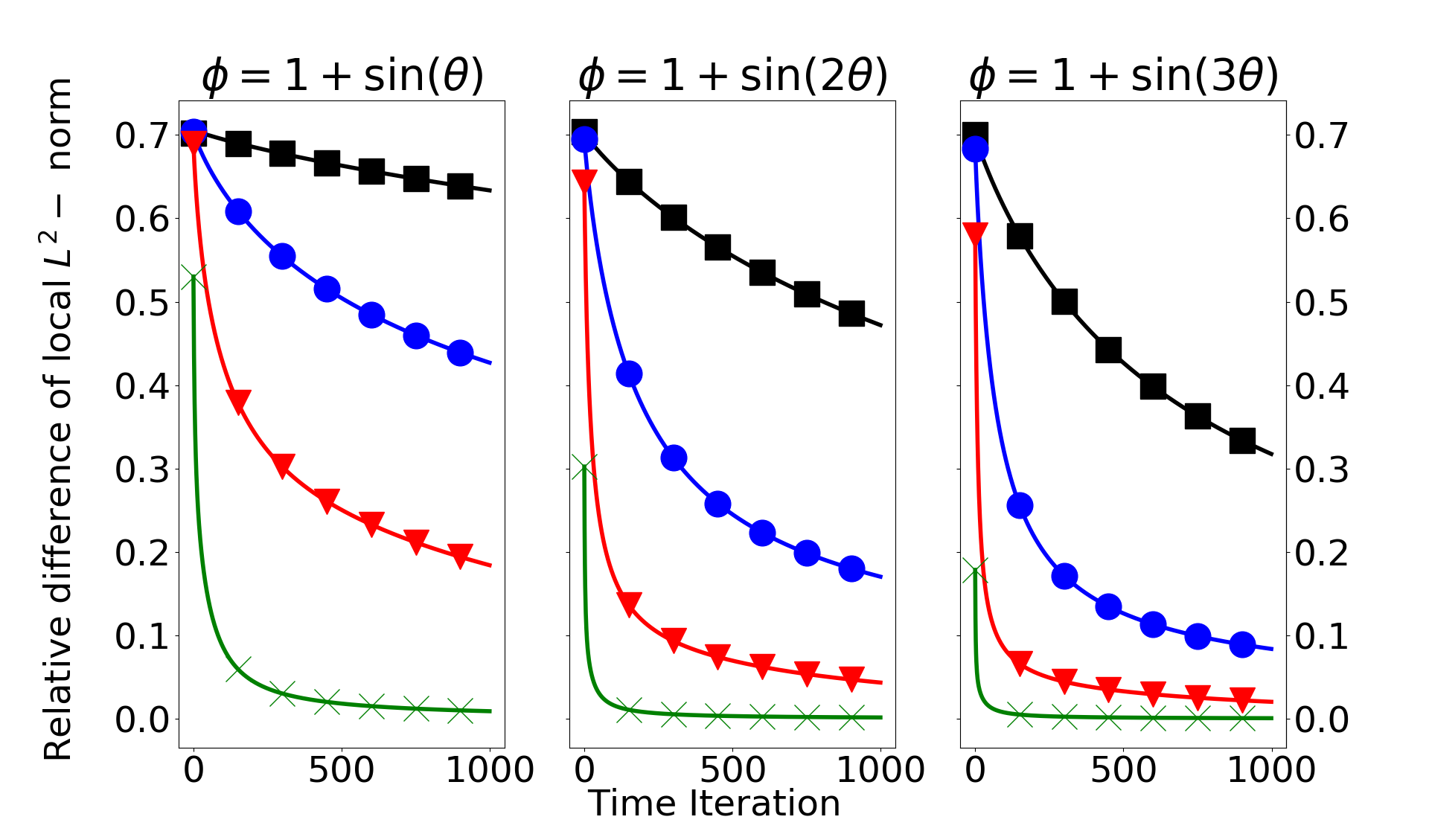}}
    \subfigure[$\displaystyle\mathcal{H}_{H^1}=\frac{\|u^{homo}_S(\boldsymbol{x}) - u^{inhomo}_S\|_{H^1(\Omega\setminus\Omega_C)}}{\|u^{homo}_S(\boldsymbol{x})\|_{H^1(\Omega\setminus\Omega_C)}}$]{
    \includegraphics[width = 0.48\textwidth]{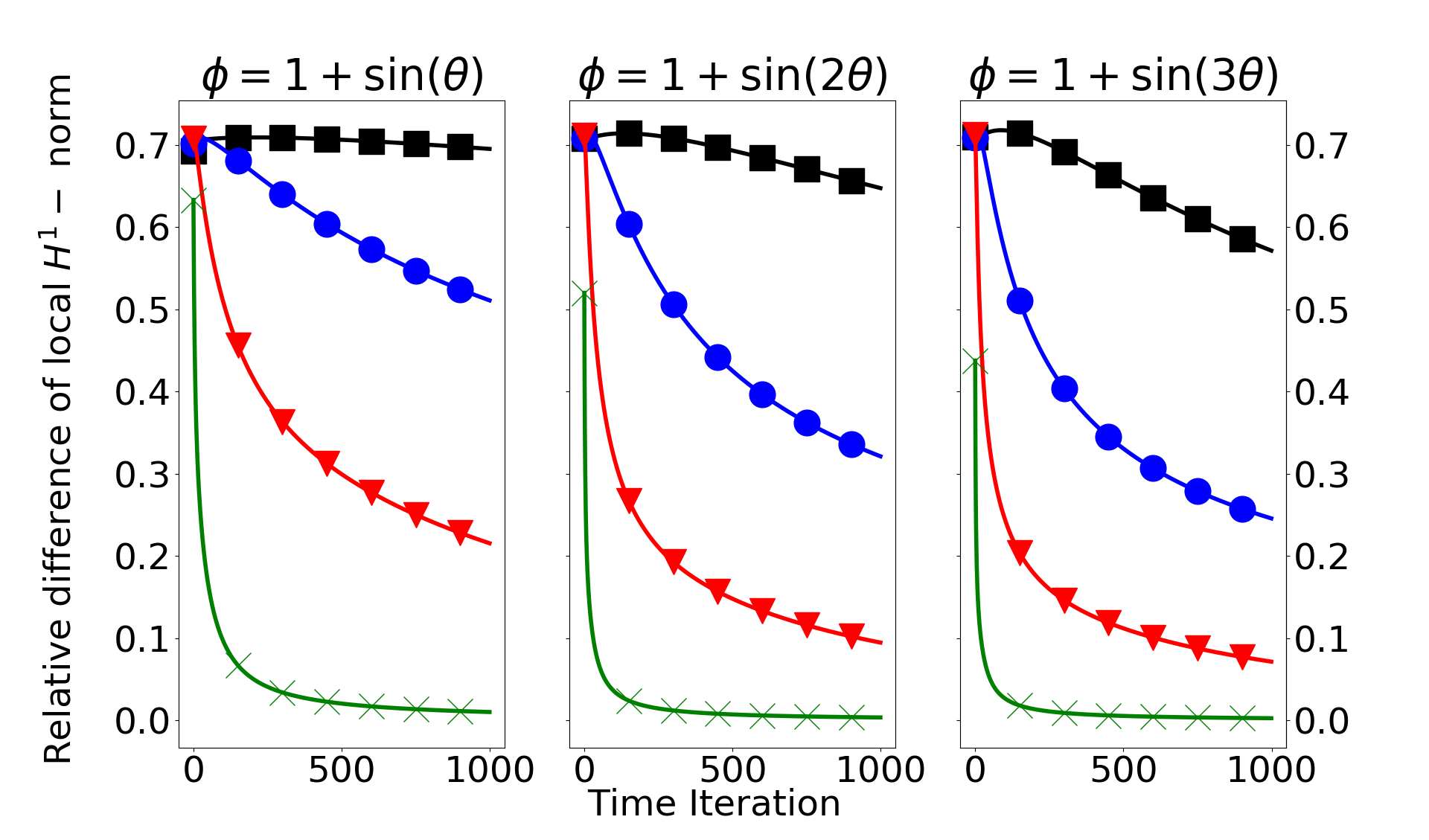}}
    \includegraphics[width = 0.55\textwidth]{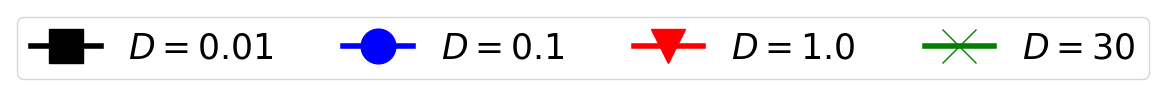}
    \caption{The local norm difference between the solution to $(\rm BVP_S)$ using the homogeneous flux density $\phi(\boldsymbol{x}) = 1$ and the inhomogeneous flux density $\phi(\boldsymbol{x}) = 1+\sin(n\theta)$, where $n\in\{1,2,3\}$ and $D$ varies in $\{0.01, 0.1, 1, 30\}$.Here, we show the relative $L^2-$ and $H^1-$norm difference in Panel (a) and (b), respectively. In simulations, standard parameter values have been used.}
    \label{Fig_hole_comp_ratio_1}
\end{figure}

Hence, in the spatial exclusion model, inhomogeneity in the flux density on the cell boundary may be ignored 
-- especially on the larger time scale -- when the diffusion coefficient $D$ is sufficiently large (roughly, $D\geqslant 0.1$), the fluctuation ratio $\rho$ is small (ca. $\rho\leqslant 0.1$) or the frequency of the fluctuation $n$ is large (ca. $n\geqslant 3$). Thus, an approximately circular cell shape and rather homogenous flux density over the cell boundary may be replaced in the spatial exclusion model effectively by a circular cell shape and homogeneous flux density, under these conditions too. The inhomogeneity of the flux density is either not significant or the extent of its effects decays relatively fast as time proceeds. A multi-Dirac approach in a points source approximation of the spatial exclusion model is not needed under those circumstances.

\subsection{Single-Dirac versus Multi-Dirac Approach}\label{Sec_Results}
\noindent
According to the results in the previous section, the spatial inhomogeneity of the flux density can be neglected in the spatial exclusion model for a larger diffusion coefficient, a smaller ratio $\rho$ and a larger value of $n$. In this section, we present the numerical results of different cases in which the spatial inhomogeneity is expected to be relevant. Here, we therefore only consider the predefined flux density of $\phi(\boldsymbol{x}) = 1+\rho\sin(n\theta)$, where $n = 1,2$. Whether the inhomogeneity is significant depends further on the value of $\rho$ and the diffusion coefficient $D$. The other parameter values used in the simulations, are listed in Table~\ref{tab:para_all}, unless they are specifically indicated otherwise.

The intensity of Dirac points (given by Equations~\eqref{Eq_dipole_sol} and~\eqref{Eq_tripole_sol}--\eqref{eq:phiC triple points}) may take large values. Therefore, the numerical solution obtained by solving $(\rm BVP_P)$ naively might be very large and unreliable. Solving $(\rm BVP_P)$ with the explicit Green's function approach that was discussed in Section~\ref{Sec_Green_approach}, can resolve this issue to some extent. Hence, in some figures of this section, we only show the graph obtained by the explicit Green's function approach for the multi-Dirac point source model, since the graph obtained by solving the BVP naively, is possibly inaccurate.
\vskip 2mm

\subsubsection{Polarized Flux Distribution: $n=1$}

\noindent According to Proposition~\ref{Prop_Condition}, it is of interest to compare the flux density over the (virtual) cell boundary, generated by the Dirac point sources, with the prescribed flux density $\phi_1(\boldsymbol{x}_\theta)$. So far, for the point source model, there are various methods computing the flux over the virtual cell boundary: the approximated flux defined in Equation~\eqref{Eq_1_Dirac_bnd_flux_t} and postprocessing the numerical solutions for the flux density by computing $D\nabla u_P(\boldsymbol{x}, t)\cdot\boldsymbol{n}$ numerically from these. For the latter, if the single-Dirac approach is used, $u_P(\boldsymbol{x},t)$ is solved directly by FEM; if multi-Dirac approach is used, due to the possible numerical instability, $u_P(\boldsymbol{x},t)$ is solved both directly and with the explicit Green's function approach. 

The flux densities over the (virtual) cell boundary in the point source model, computed in these various ways, and the predefined flux in the Spatial Exclusion Model are shown in Figure~\ref{Fig_1_Dirac_bnd_flux_ratio_1_D_1_r_001_ALL} at different times. Figure~\ref{Fig_1_Dirac_ratio_1_D_1} shows quantities that measure quality of approximation between the point source model (in variants and different methods of numerically solving the equations) and the spatial exclusion model, when the flux density is given by $\phi(\boldsymbol{x}) = 1+\sin(\theta)$.

\begin{figure}[h!]
    \centering
    \subfigure[$t = 0.04$]{
    \includegraphics[width=0.48\linewidth]{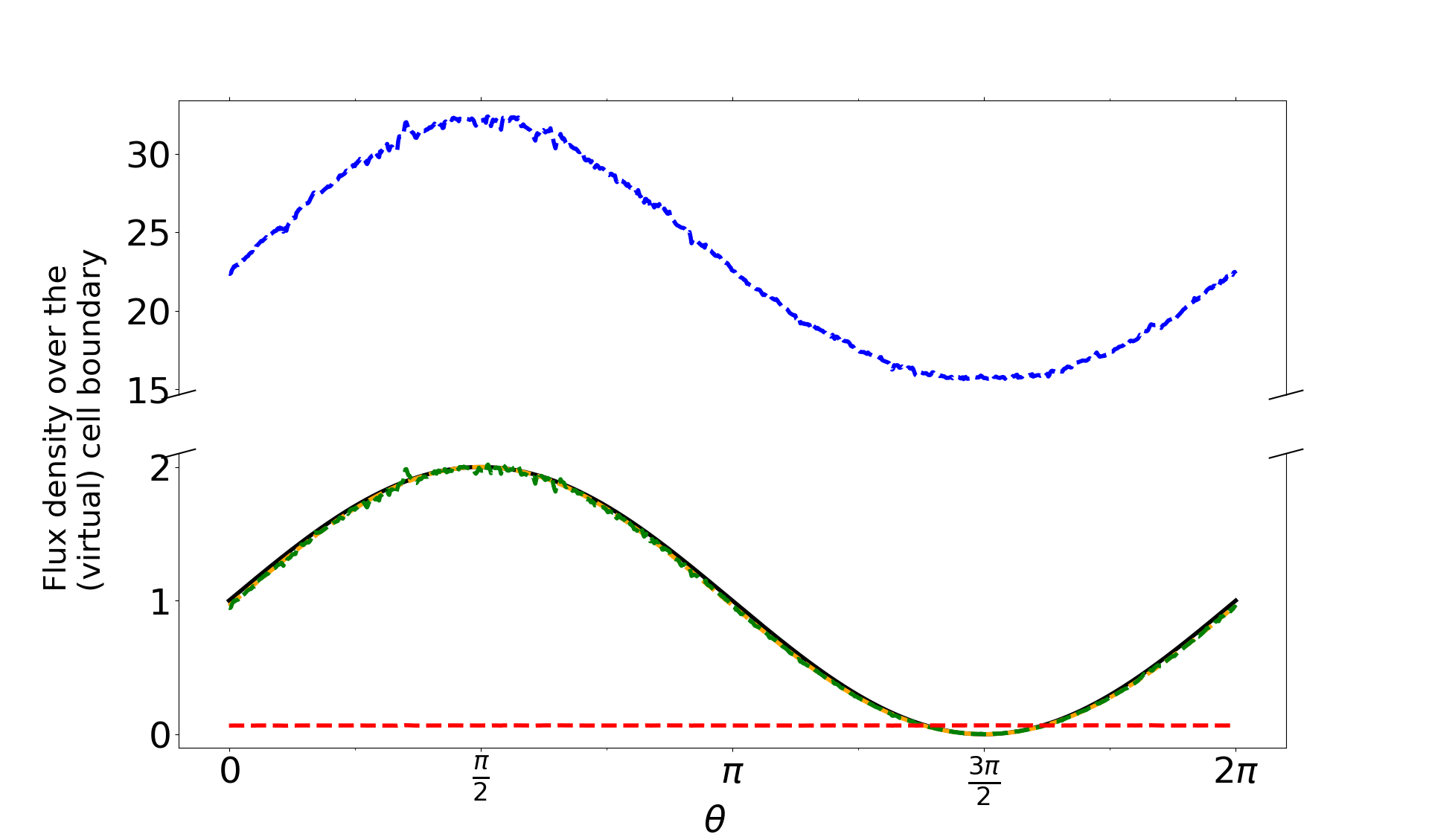}}
    \subfigure[$t = 0.8$]{
    \includegraphics[width=0.48\linewidth]{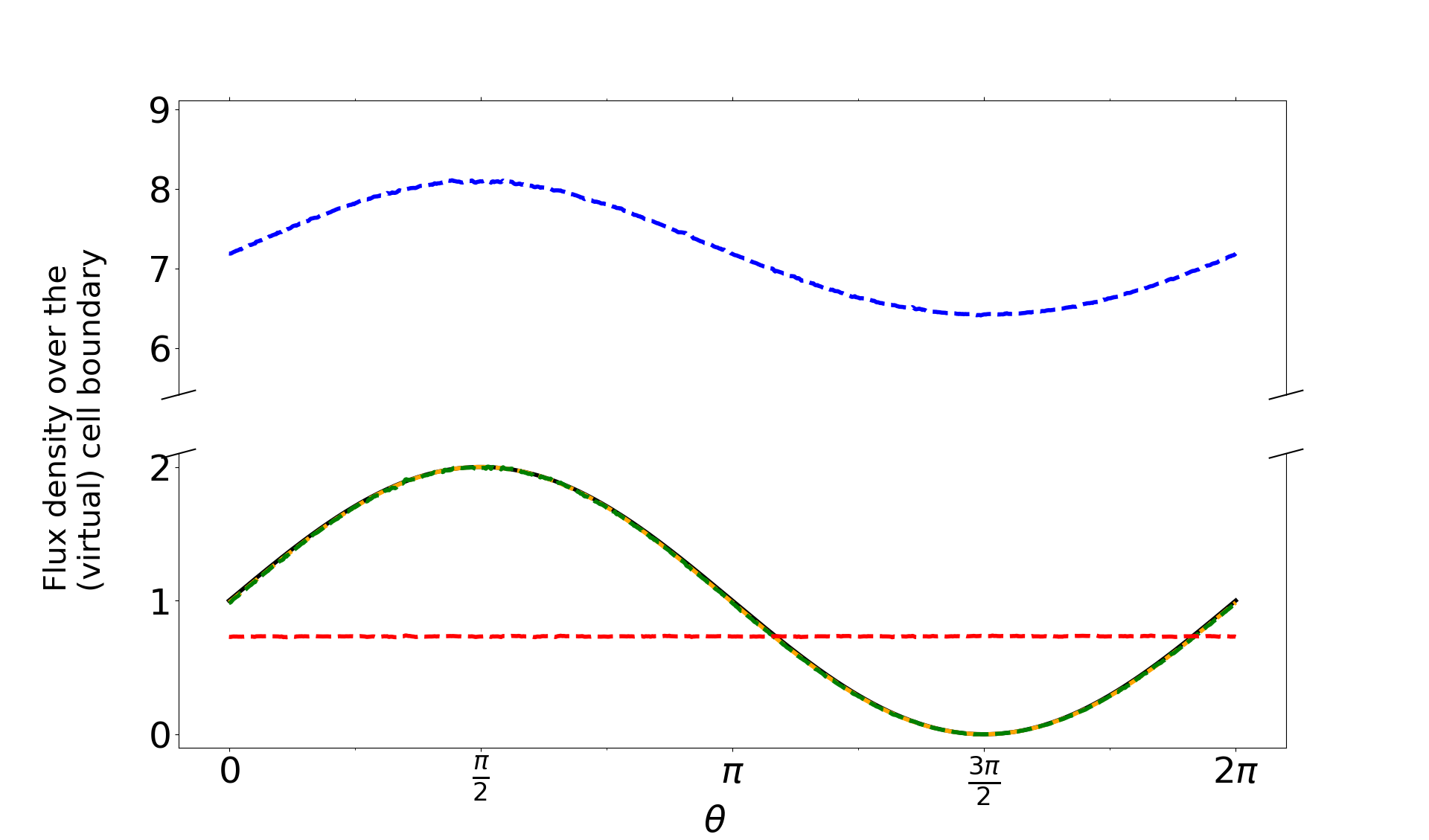}}
    \subfigure[$t = 2.0$]{
    \includegraphics[width=0.48\linewidth]{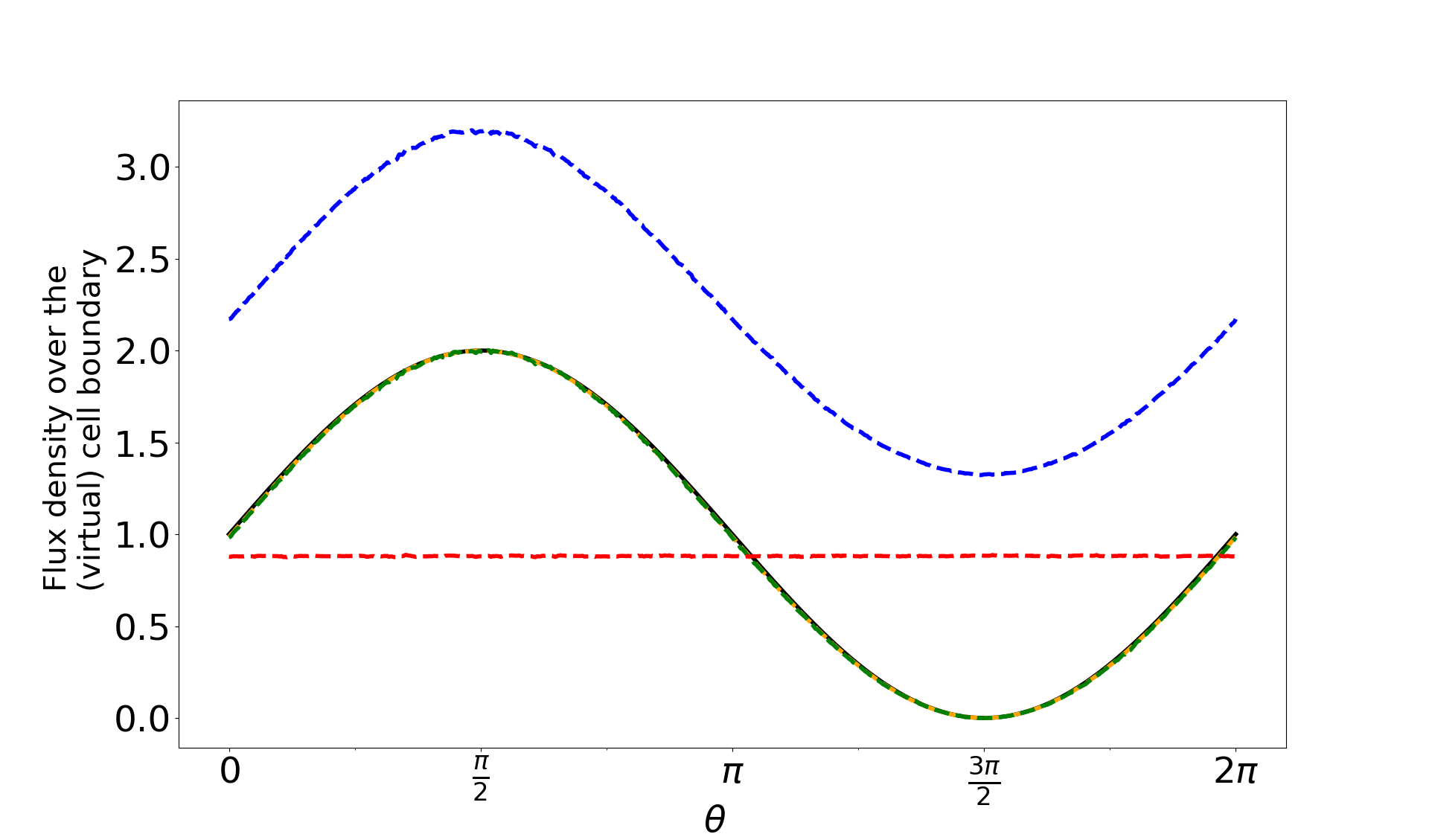}}
    \subfigure[$t = 4.0$]{
    \includegraphics[width=0.48\linewidth]{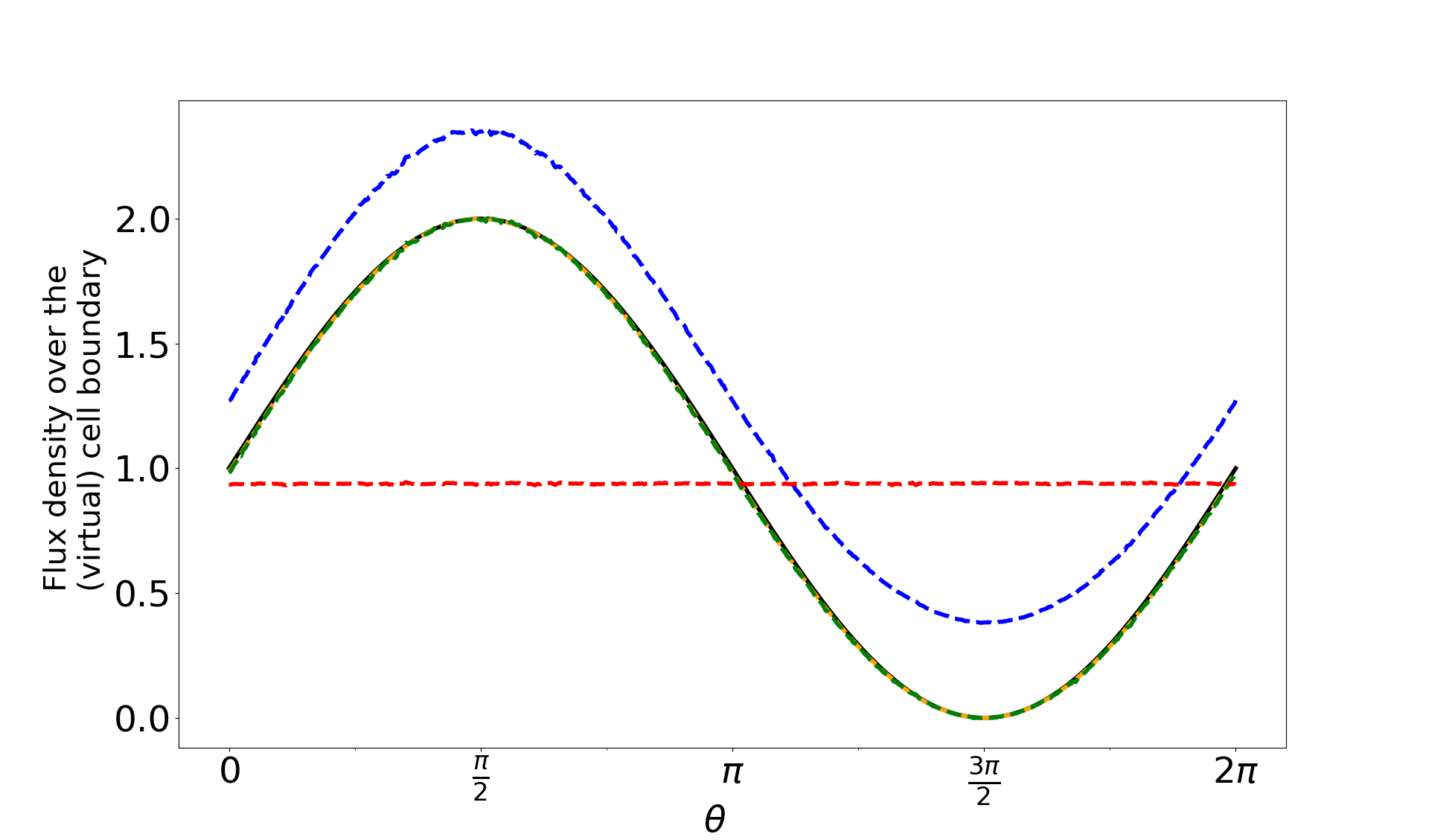}}
    \subfigure[$t = 8.0$]{
    \includegraphics[width=0.48\linewidth]{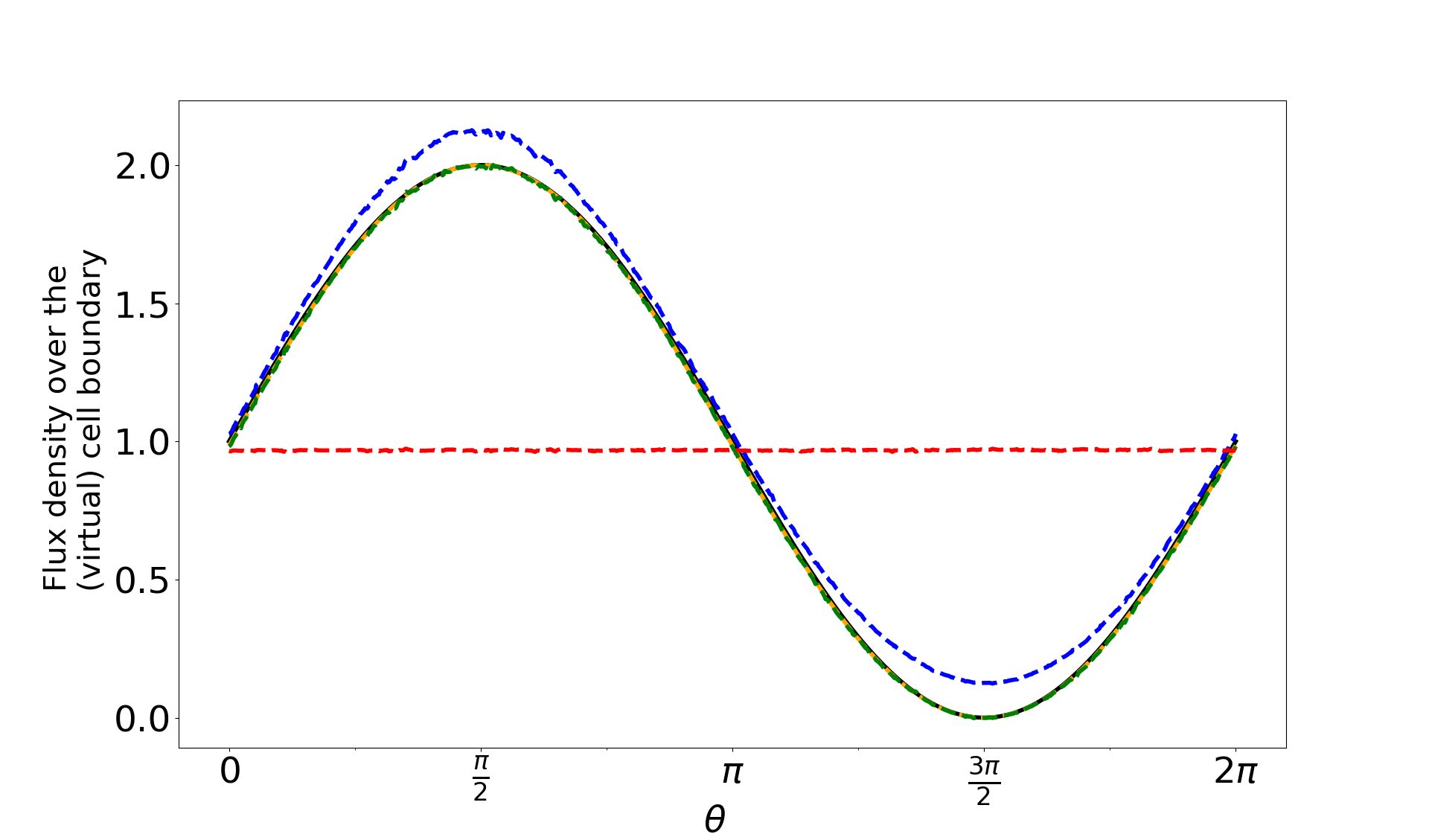}}
    \subfigure[$t = 40$]{
    \includegraphics[width=0.48\linewidth]{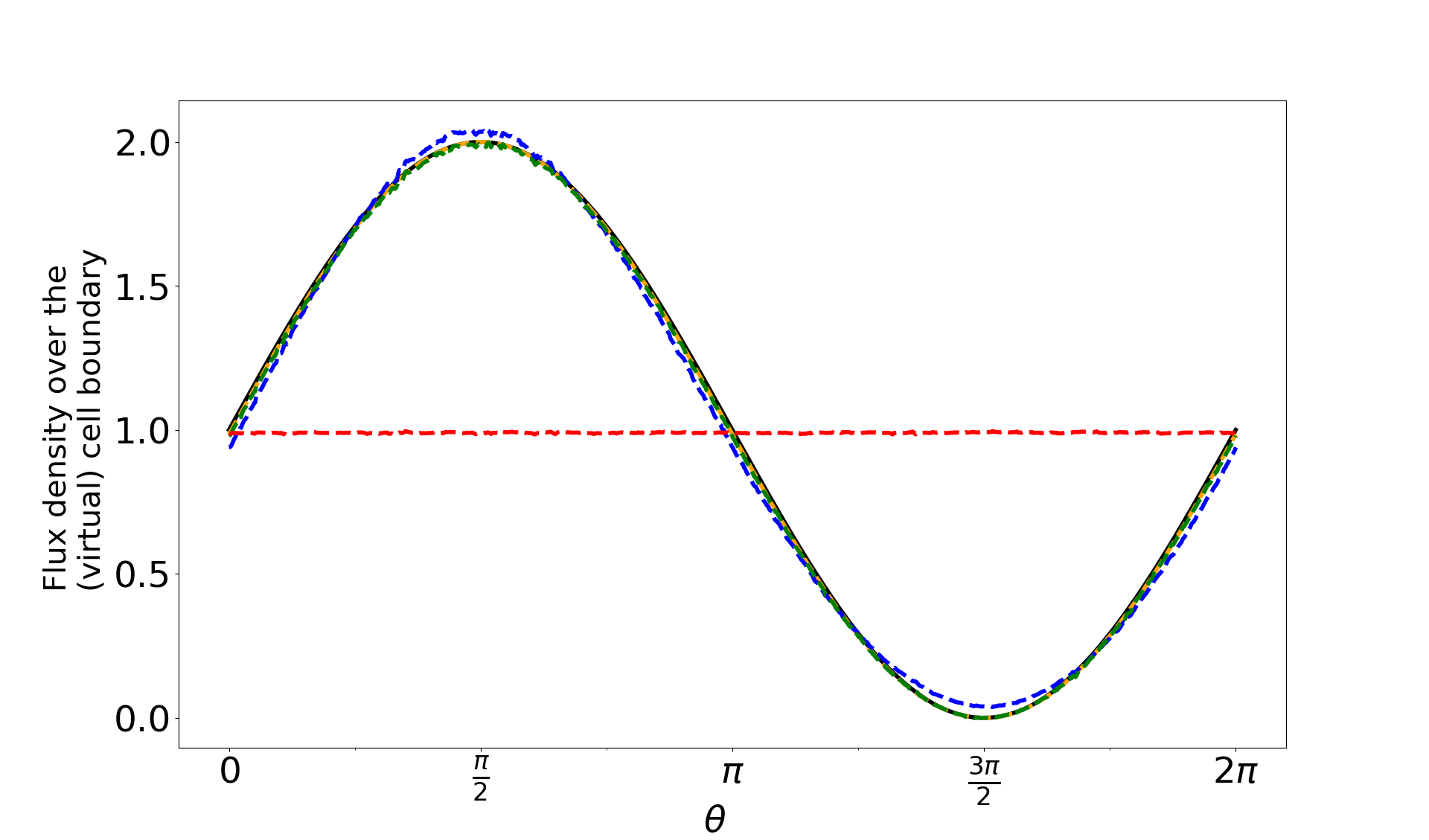}}
    \includegraphics[width = 0.96\linewidth]{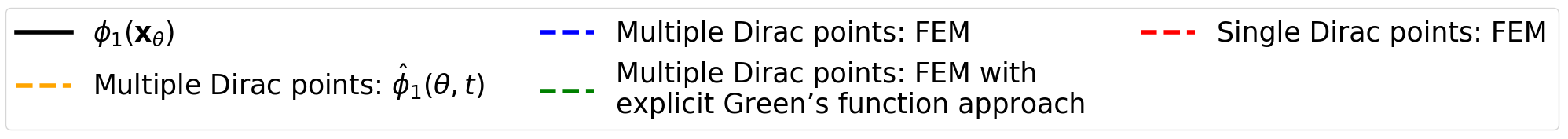}
    \caption{The flux over the (virtual) cell boundary $\partial\Omega_C$ at different time points is shown in various point source models. The black solid curve represents the predefined flux (given by Equation~\eqref{Eq_phi} with $n=1$) in the spatial exclusion model. The orange curve represents the approximated flux defined by Equation~\eqref{Eq_1_Dirac_bnd_flux_t}. The blue and green dashed curves are the results of numerical simulation of the point source model with multiple Dirac points: the orange one , the blue and green ones are solved by FEM directly and with the explicit Green's function approach (see Section~\ref{Sec_Green_approach}), respectively. The red dashed curve stands for the flux post-processed by the point source model using one centre Dirac point. Note that in Panel (a) the y-axis is discrete, and the upper and bottom parts are not on the same scale. Parameter values are taken from Table~\ref{tab:para_all}.}
    \label{Fig_1_Dirac_bnd_flux_ratio_1_D_1_r_001_ALL}
\end{figure}

From the beginning (Figure~\ref{Fig_1_Dirac_bnd_flux_ratio_1_D_1_r_001_ALL}(a)-(b)), we can already observe that for the multi-Dirac approach solved by FEM directly (blue dashed curve), even though we ensured the same flux density at the location of the extreme values, the boundary flux is significantly different from the graphs obtained by the other methods. To highlight this discrepancy, we split the vertical axis into two parts that are not at the same scale. As time proceeds, the blue dashed curve moves closer to and in the end mostly overlaps with the predefined flux density (black solid curve). This explains why in Figure~\ref{Fig_1_Dirac_ratio_1_D_1}, for all the interesting quantities, there exists a large peak in the $L^2$- and $H^1$-norm differences in the early part of the simulations. 

Contrary to the blue dashed curve in Figure~\ref{Fig_1_Dirac_bnd_flux_ratio_1_D_1_r_001_ALL}, the approximated flux density (orange dashed curve) over the (virtual) cell boundary defined in Equation~\eqref{Eq_1_Dirac_bnd_flux_t} and post-processed flux computed from the solution to the point source model with multi-Dirac points by the explicit Green's function approach (green dashed curve), already overlap with the predefined boundary flux (red solid curve). In other words, the approximated flux provides a high-quality approximation to the predefined flux over the time domain of the simulation. 

\begin{figure}[h!]
    \centering
    \subfigure[$\|u_S-u_P\|_{L^2(\Omega\setminus\Omega_C)}$]{
    \includegraphics[width = 0.48\textwidth]{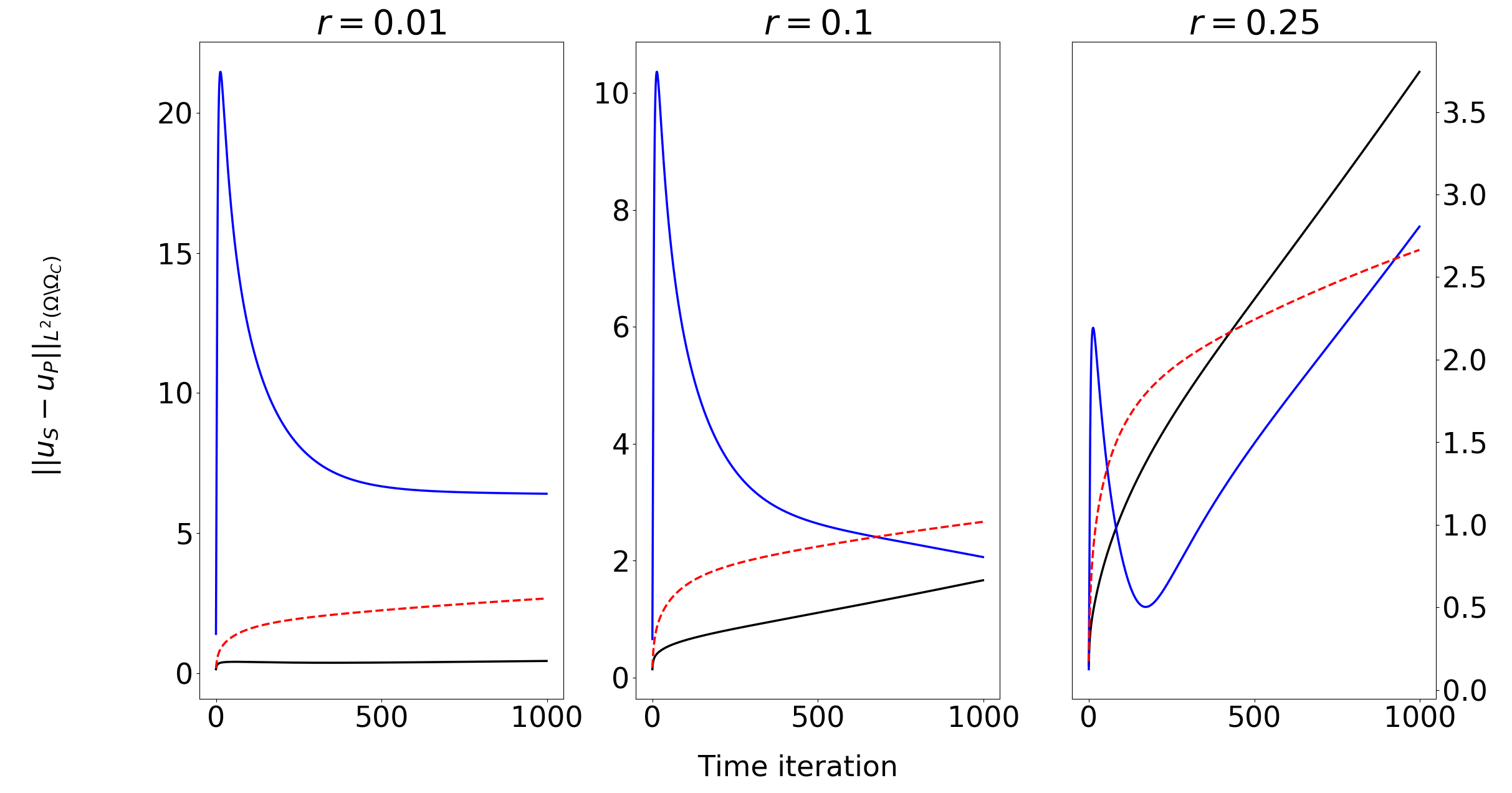}}
    \subfigure[$\|u_S-u_P\|_{H^1(\Omega\setminus\Omega_C)}$]{
    \includegraphics[width = 0.48\textwidth]{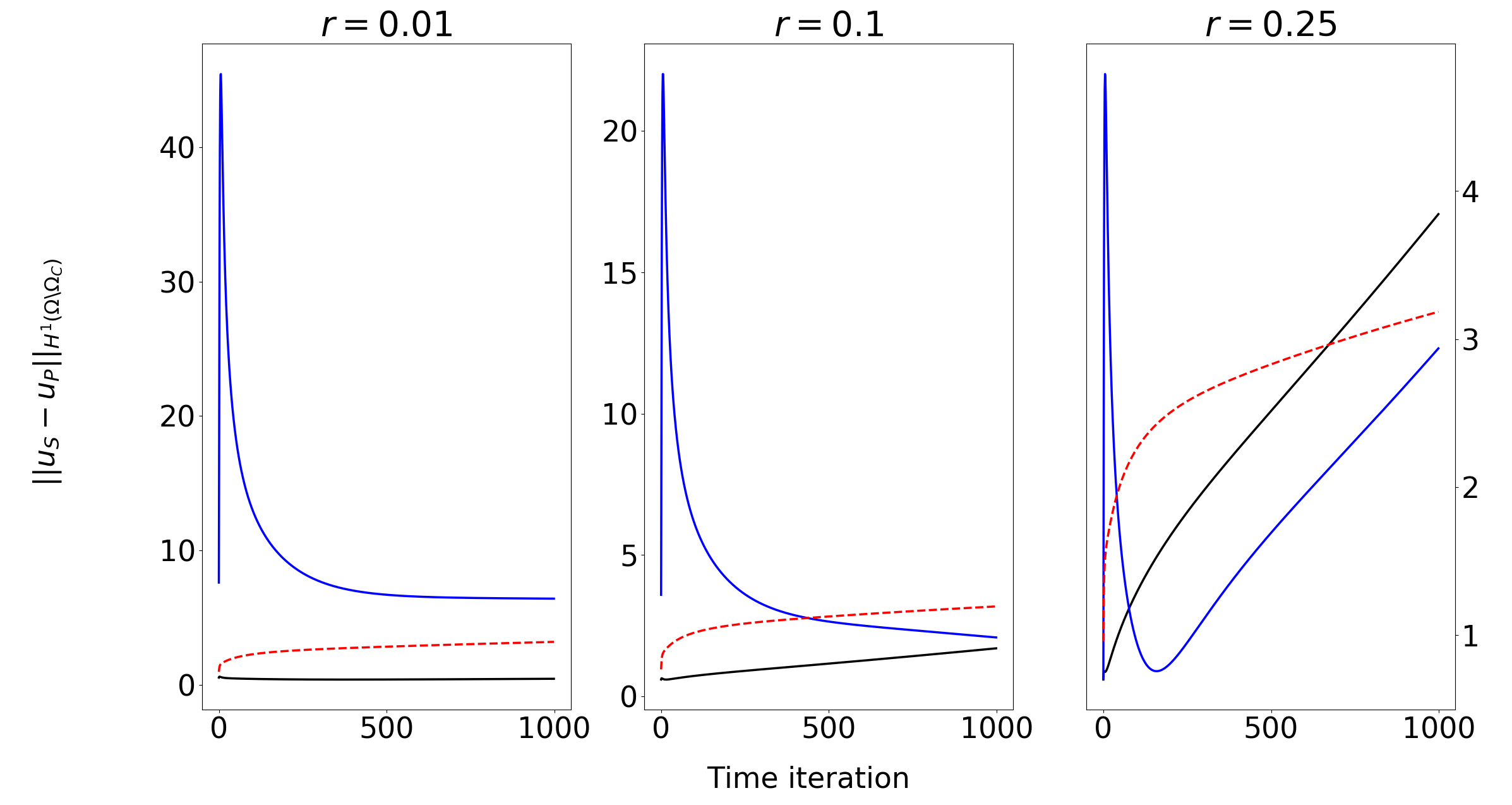}}
    \subfigure[$c^*(t)$]{
    \includegraphics[width = 0.48\textwidth]{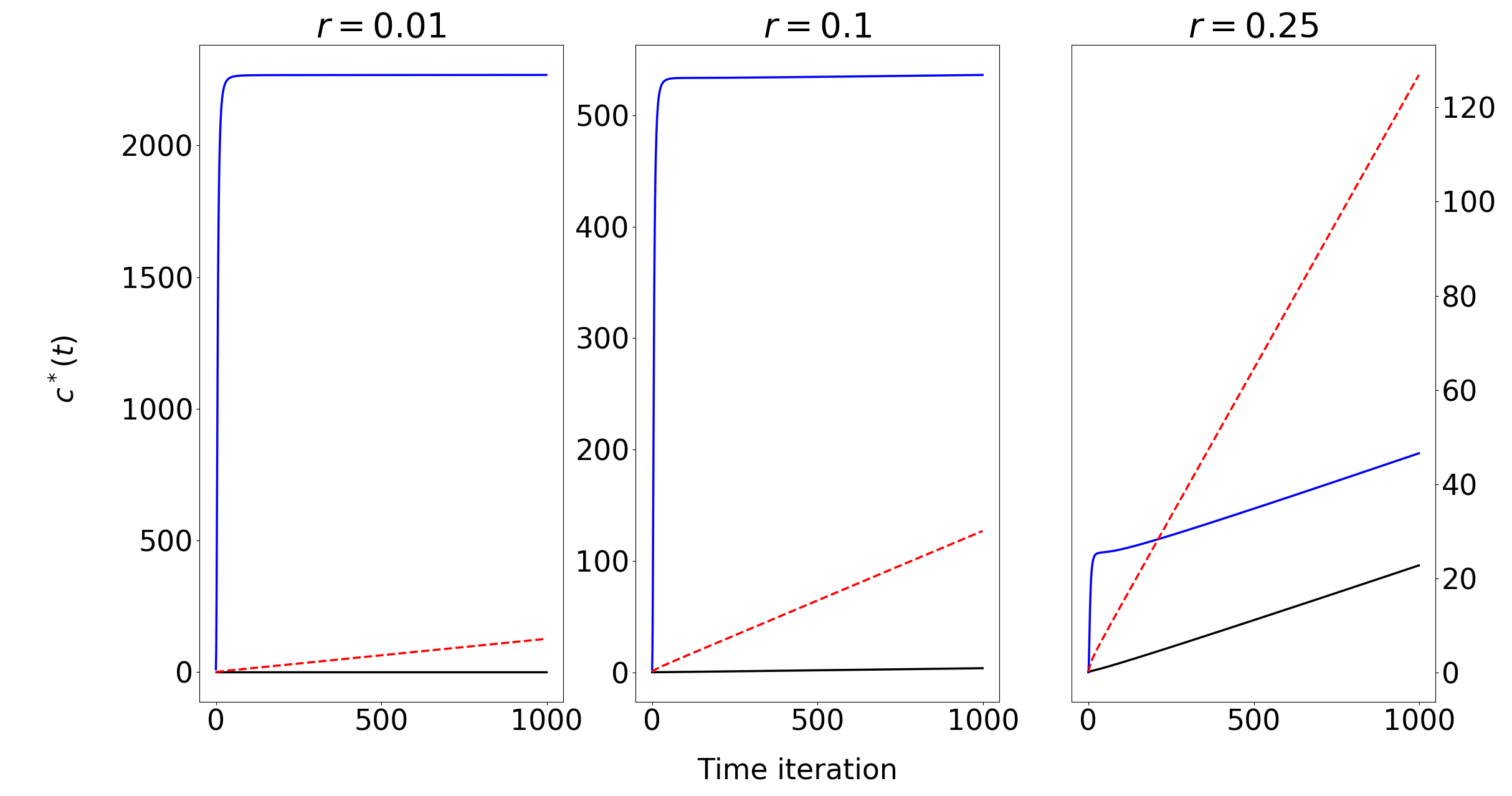}}
    \subfigure[$\|\phi(\boldsymbol{x}) - D\nabla u_P\cdot\boldsymbol{n}\|_{L^2(\partial\Omega_C)}$]{
    \includegraphics[width = 0.48\textwidth]{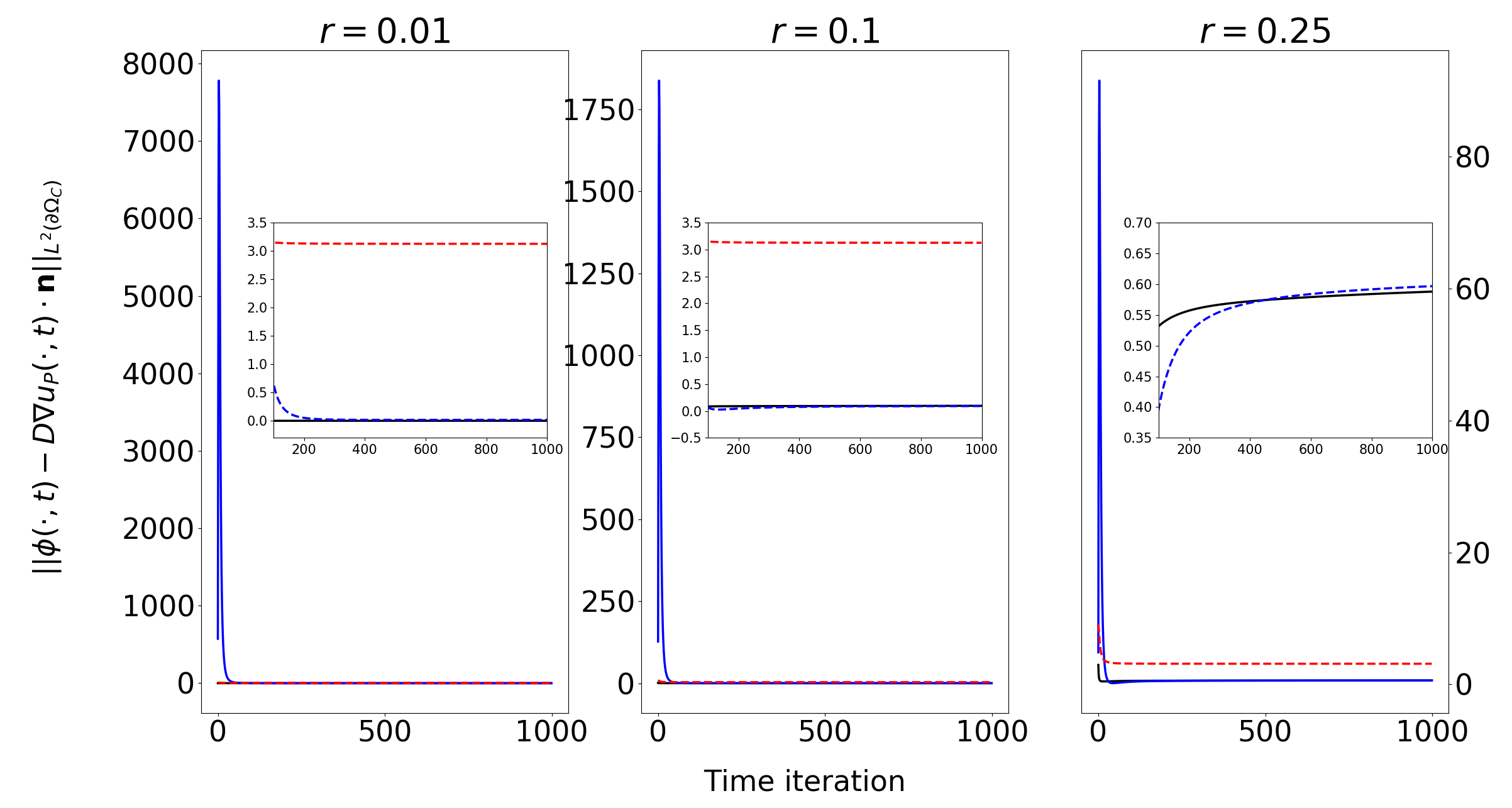}}
    \includegraphics[width = \textwidth]{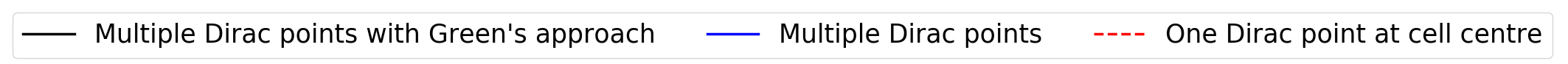}
    \caption{Comparison of single- and multi-Dirac approach to spatial exclusion model for the polarized flux density with large fluctuations $(\rho=1)$, i.e. $\phi(\boldsymbol{x}) = 1+\sin(\theta)$. The simulations were made with parameter values from Table~\ref{tab:para_all}. The local $L^2-$ (Panel (a)) and $H^1-$norm difference (Panel (b)) between the solutions to the spatial exclusion model and the point source model with one (red dashed curve) and multiple Dirac points (black -- explicit Green's function approach -- and blue solid curve -- direct FEM approach), respectively. The $c^*(t)$ and the boundary flux difference at each time step are shown in Panel (c) and (d), respectively.}
    \label{Fig_1_Dirac_ratio_1_D_1}
\end{figure}

Moreover, numerically, it is more stable to solve the multi-Dirac point source model with the explicit Green's function approach, particularly when the intensity of the Dirac points is very large,  as can be seen in Figure~\ref{Fig_1_Dirac_ratio_1_D_1}. One can observe that for $r=0.01$, solving the multi-Dirac point source model with this approach gives the smallest value in all the quantities. The single-Dirac approach tends to give worse results, especially for the smaller values of $r$. 

In Figure~\ref{Fig_1_Dirac_ratio_1_D_1}, with the value of diffusion coefficient $D=1$, the single-Dirac approach still consists of a systematic time delay, that is, the compounds take more than one time step to travel from the cell center to the cell boundary. As a result, in Figure~\ref{Fig_1_Dirac_ratio_1_D_1}(d), the boundary flux difference of the single-Dirac approach is relatively large and then it decays till the steady state, while in the multi-Dirac approach, the boundary flux difference already starts small. Furthermore, the rest of the subfigures confirms the benefit of using the multi-Dirac approach when the inhomogeneity of the flux density is significant: as long as $r$ is small and close to $0$, then the local norm difference in the multi-Dirac approach is always smaller than in the single-Dirac approach.

When the diffusion coefficient is sufficiently large, i.e. the compounds take less than one time step to travel from the cell centre to the cell boundary, then there is no error caused by this systematic time delay (see \citet{Peng2023} for more details). In the meantime, even though the inhomogeneous flux density leads to heterogeneity in the extracellular environment, thanks to the large diffusion coefficient, the concentration discrepancy can be smoothened quickly. As a result, the heterogeneity is not significant over the entire time range; see Figure~\ref{Fig_1_Dirac_ratio_1_D_30} as an example. For all the options of $r$, even though the multi-Dirac approach can well describe the inhomogeneous boundary flux: Figure~\ref{Fig_1_Dirac_ratio_1_D_30}(c)-(d) show the significant difference of $c^*(t)$ and $\|\phi(\boldsymbol{x},t) - D\nabla u\cdot\boldsymbol{n}\|_{L^2(\partial\Omega_C)}$ between the multi-Dirac and single-Dirac approach. However, regarding the concentration in the extracellular environment, even though in the beginning, the multi-Dirac approach always has a smaller error, this error increases linearly and faster than the error of the single-Dirac approach.   
\begin{figure}[h!]
    \centering
    \subfigure[$\|u_S-u_P\|_{L^2(\Omega\setminus\Omega_C)}$]{
    \includegraphics[width = 0.48\textwidth]{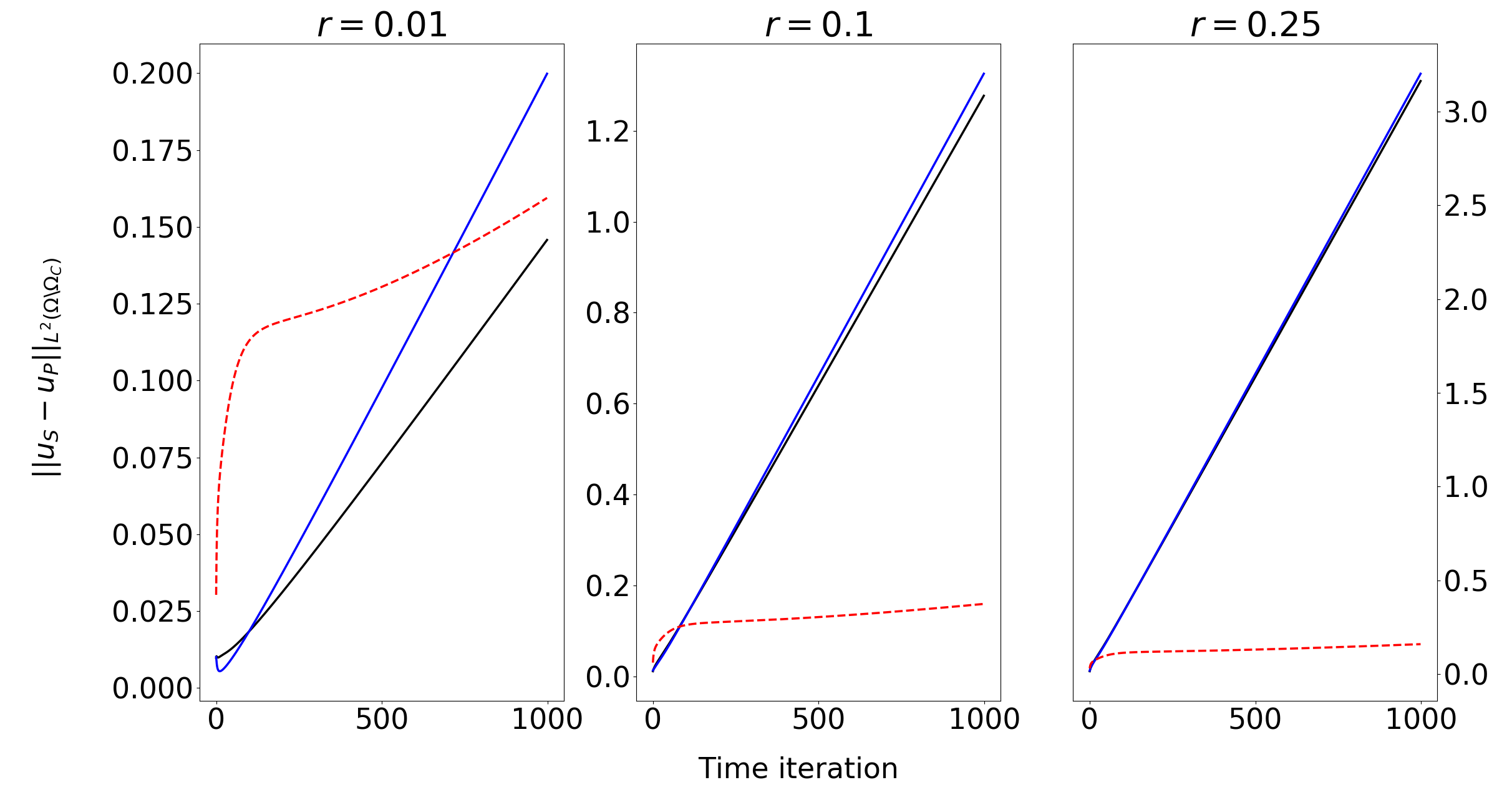}}
    \subfigure[$\|u_S-u_P\|_{H^1(\Omega\setminus\Omega_C)}$]{
    \includegraphics[width = 0.48\textwidth]{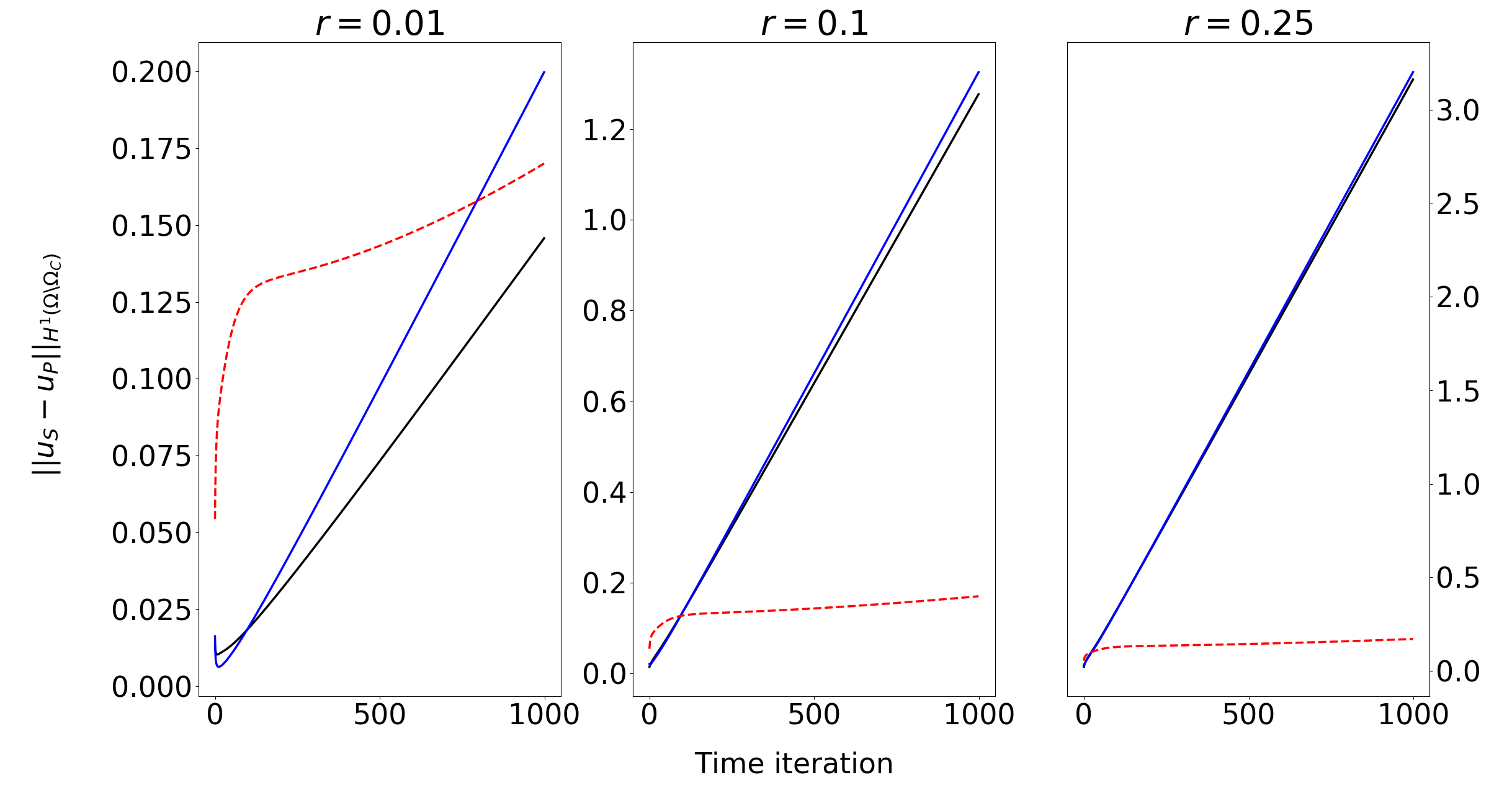}}
    \subfigure[$c^*(t)$]{
    \includegraphics[width = 0.48\textwidth]{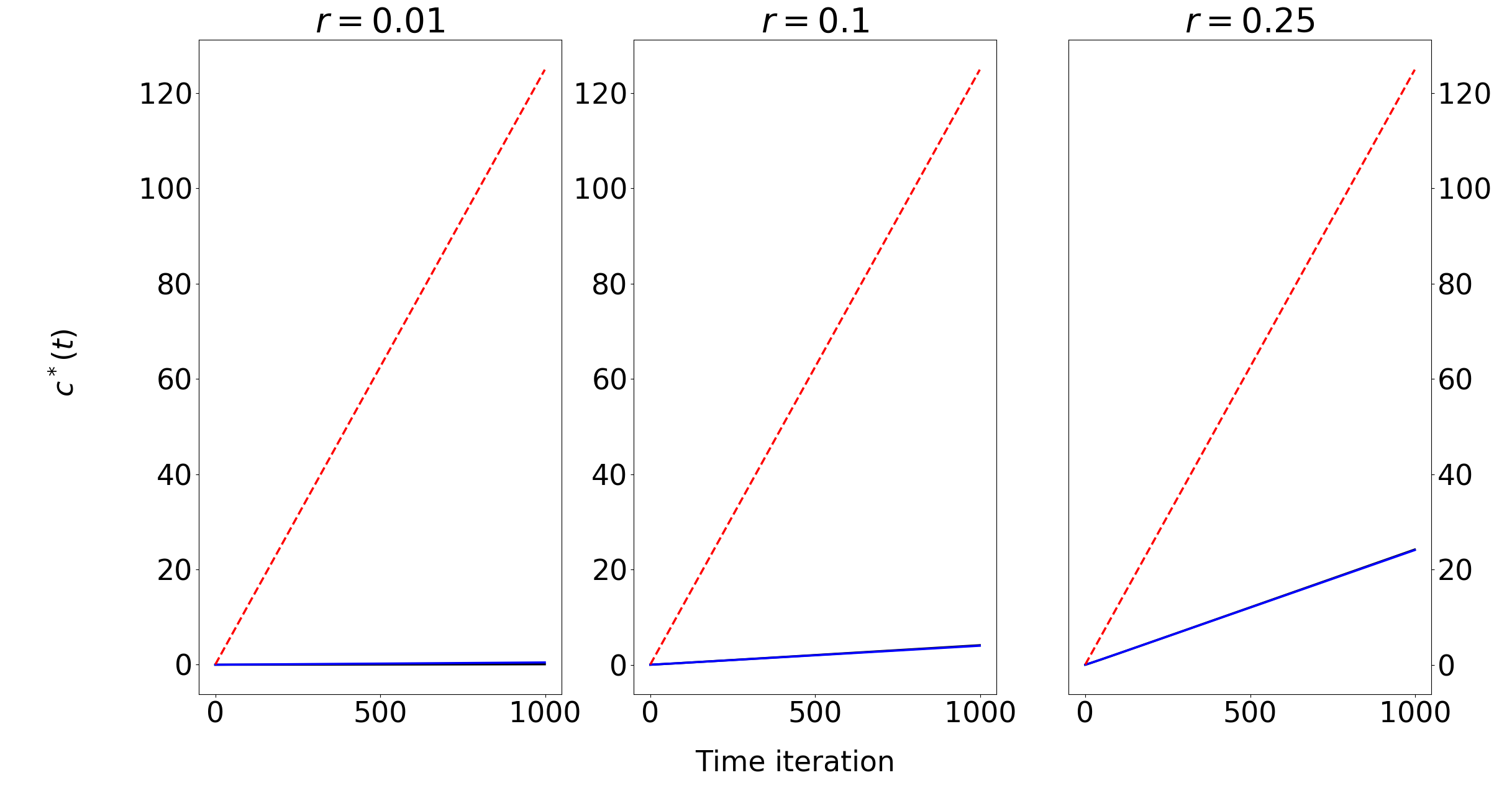}}
    \subfigure[$\|\phi(\boldsymbol{x}) - D\nabla u_P\cdot\boldsymbol{n}\|_{L^2(\partial\Omega_C)}$]{
    \includegraphics[width = 0.48\textwidth]{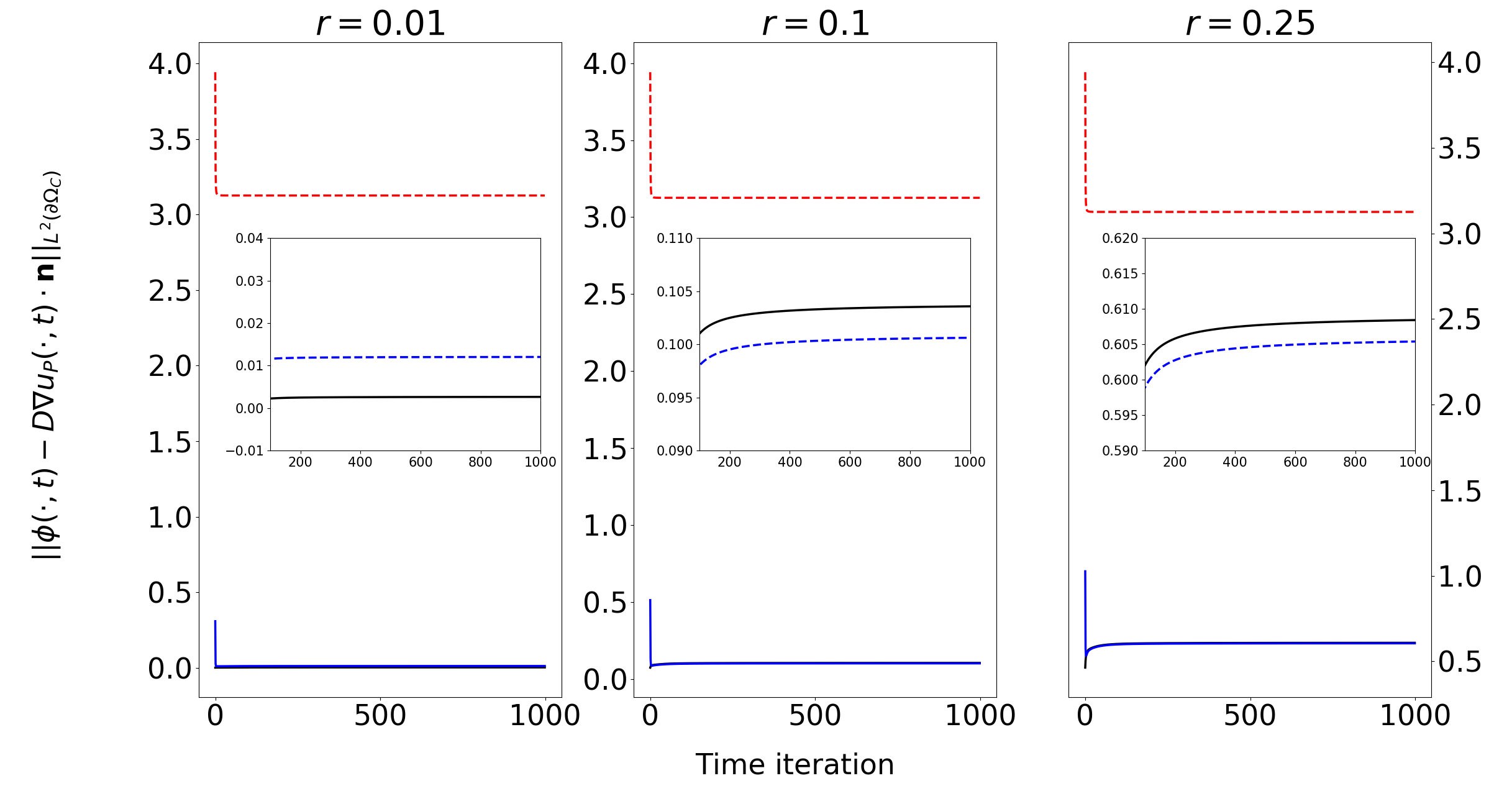}}
    \includegraphics[width = \textwidth]{3_curves_legend.png}
    \caption{Comparison of single- and multi-Dirac approach to spatial exclusion model for the polarized flux density with ratio $\rho=1$, i.e.  $\phi(\boldsymbol{x}) = 1+\sin(\theta)$. The simulations were made with a large diffusion coefficient $(D=30.0)$. The local $L^2-$ (Panel (a)) and $H^1-$norm difference (Panel (b)) between the solutions to the spatial exclusion model and the point source model with one (red dashed curve) and multiple Dirac points (black and blue solid curve), respectively. The $c^*(t)$ and the boundary flux difference at each time step are shown in Panel (c) and (d), respectively.}
    \label{Fig_1_Dirac_ratio_1_D_30}
\end{figure}

If $\rho$ in the predefined flux density is reduced, the inhomogeneity is less significant. We keep the same parameter values as in Table~\ref{tab:para_all} \textit{except for} modifying the value of $\displaystyle\rho$ to $0.01$, Figure~\ref{Fig_1_Dirac_ratio_001_D_1} shows that multi-Dirac approach is still favoured based on the local solution deviation between the spatial exclusion model and the point source model. Meanwhile, the difference between the multi-Dirac approach and the single-Dirac approach is less than in Figure~\ref{Fig_1_Dirac_ratio_1_D_1}, which verifies again the reduction of the homogeneity when $\displaystyle\rho$ decreases.
\begin{figure}[h!]
    \centering
    \subfigure[$\|u_S-u_P\|_{L^2(\Omega\setminus\Omega_C)}$]{
    \includegraphics[width = 0.48\textwidth]{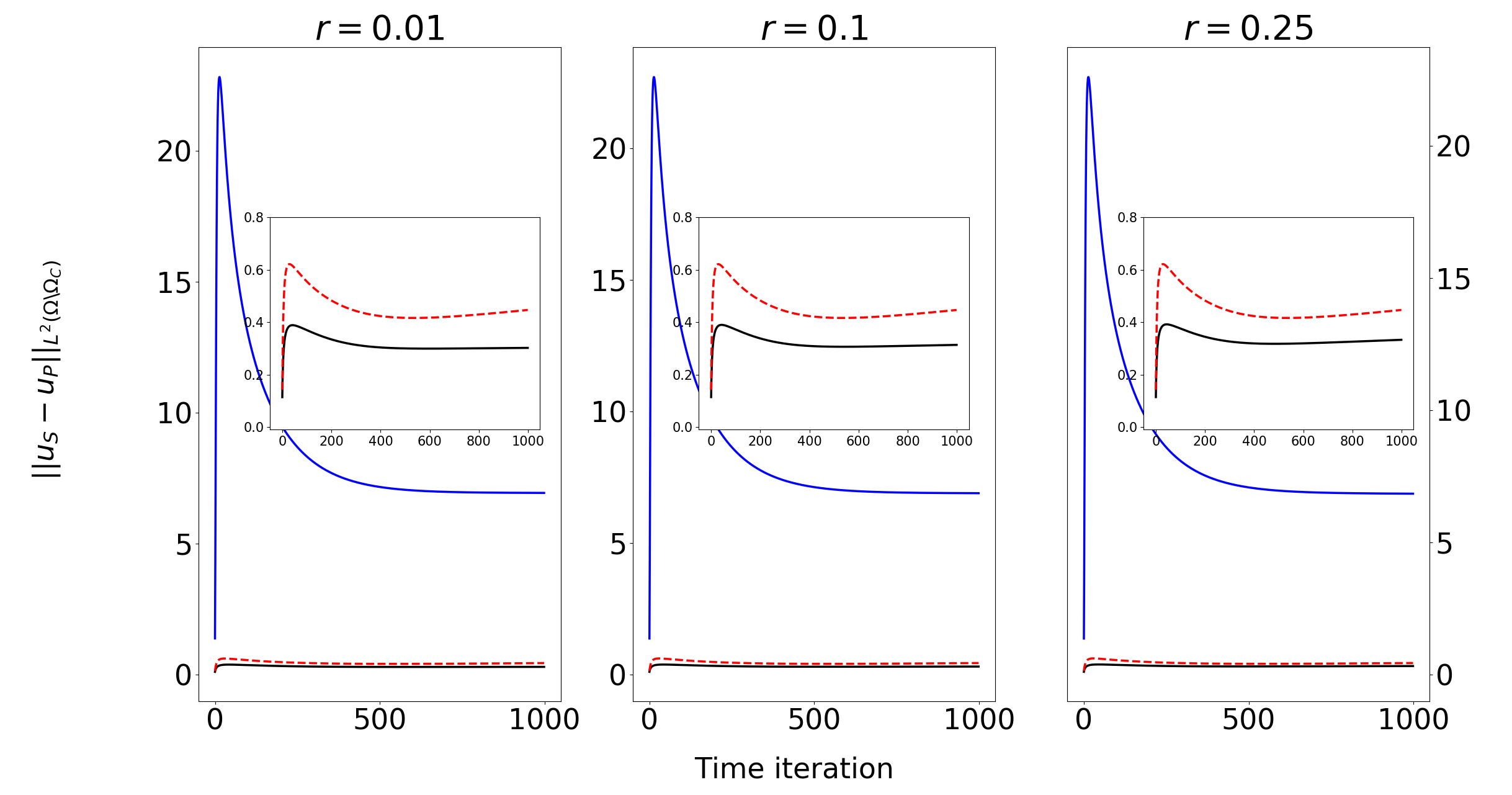}}
    \subfigure[$\|u_S-u_P\|_{H^1(\Omega\setminus\Omega_C)}$]{
    \includegraphics[width = 0.48\textwidth]{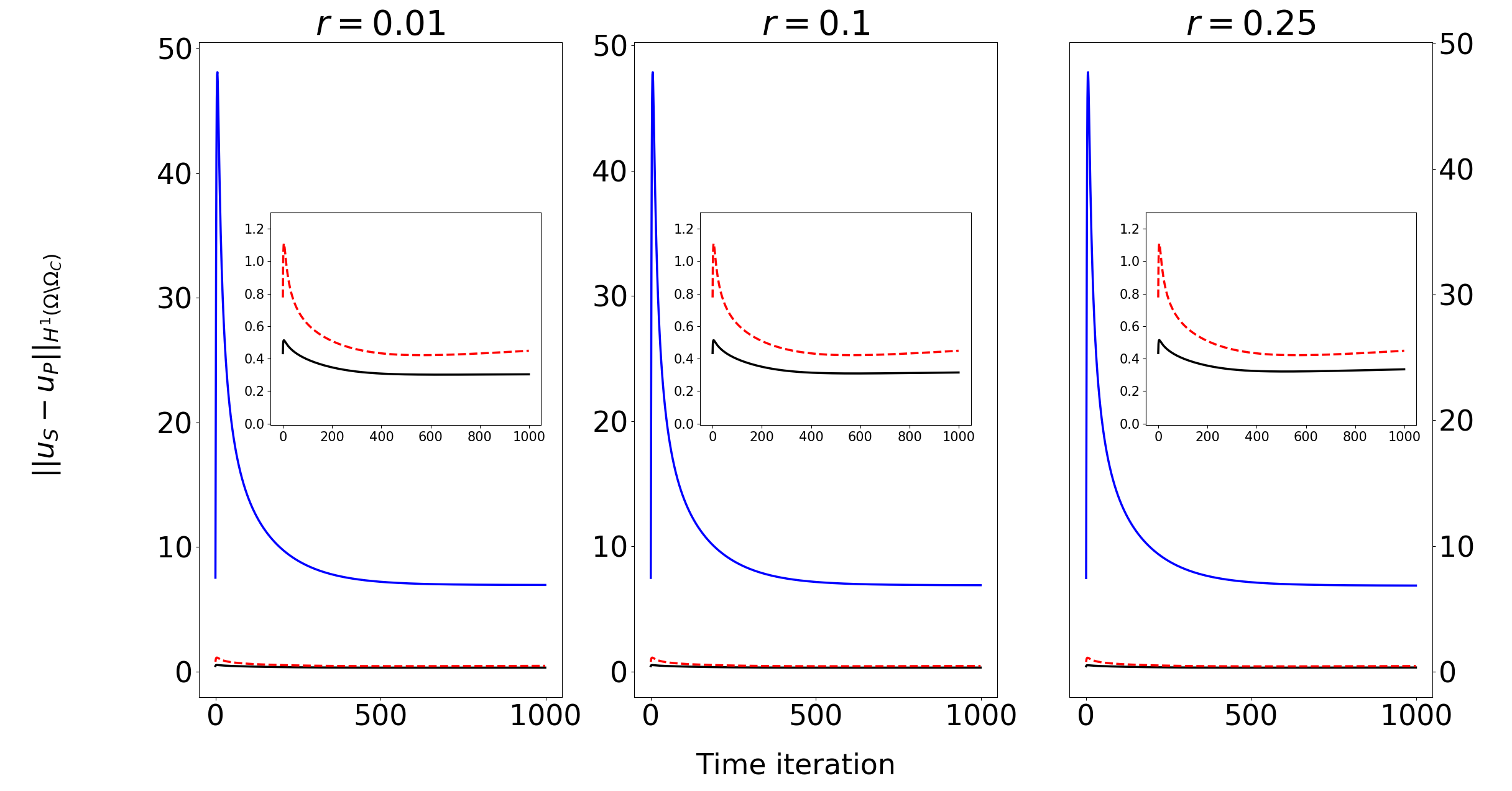}}
    \subfigure[$c^*(t)$]{
    \includegraphics[width = 0.48\textwidth]{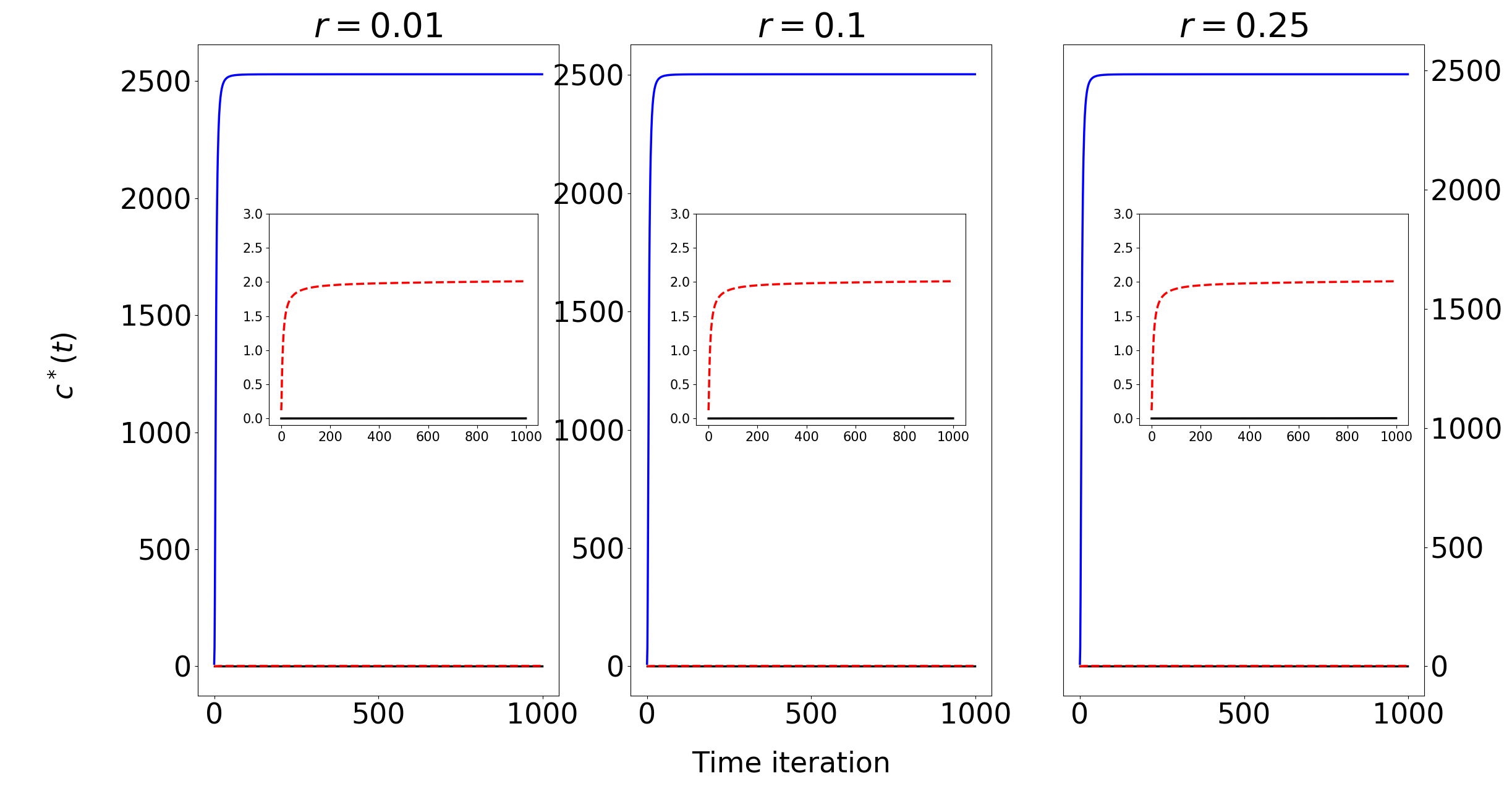}}
    \subfigure[$\|\phi(\boldsymbol{x}) - D\nabla u_P\cdot\boldsymbol{n}\|_{L^2(\partial\Omega_C})$]{
    \includegraphics[width = 0.48\textwidth]{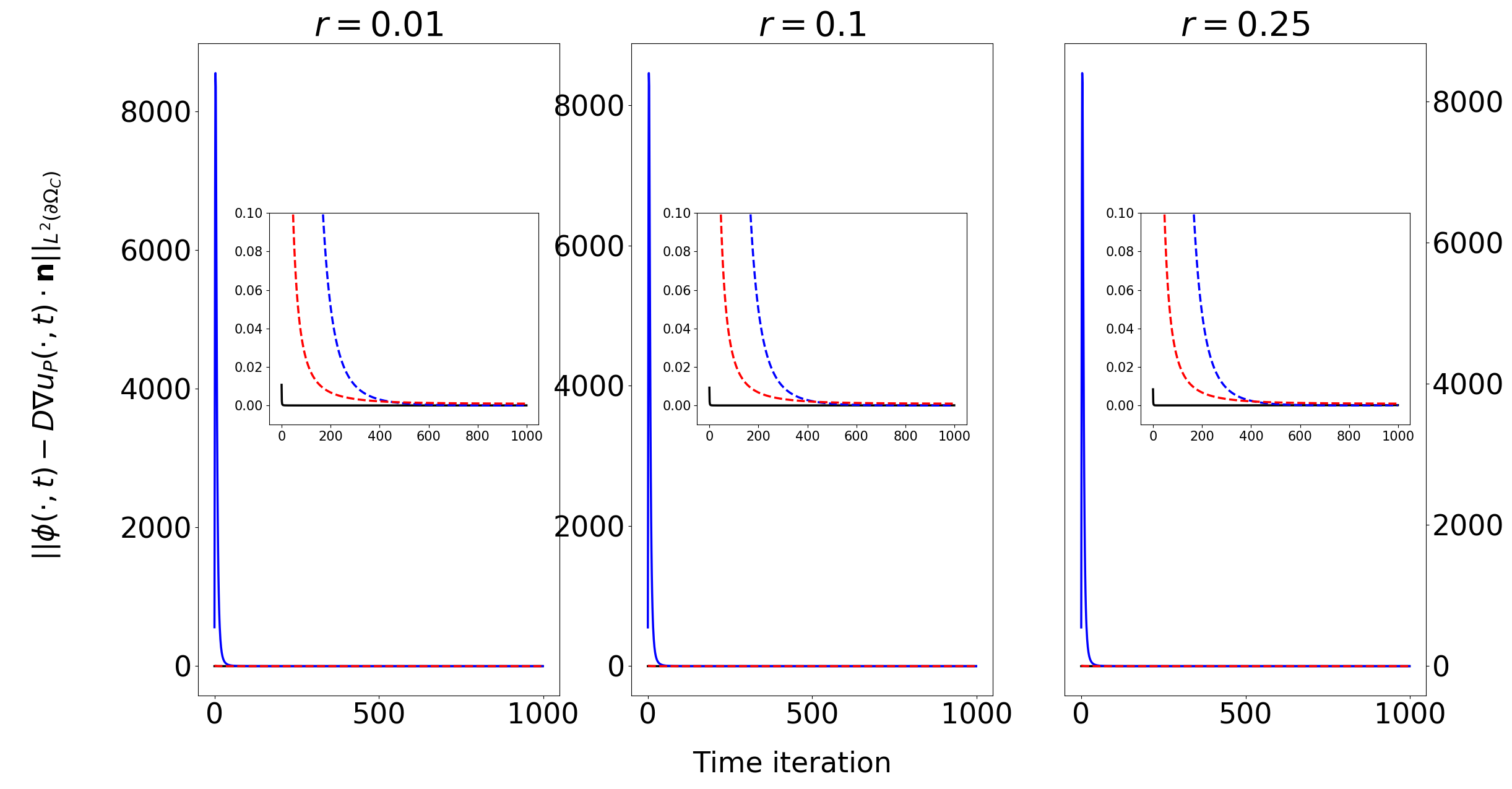}}
    \includegraphics[width = \textwidth]{3_curves_legend.png}
    \caption{Comparison of single- and multi-Dirac approach to spatial exclusion model for the polarized flux density with $\rho=0.01$, i.e. $\phi(\boldsymbol{x}) = 1+0.01\sin(\theta)$. The simulations were made with the rest of the parameters taken from Table~\ref{tab:para_all}. The local $L^2-$ (Panel (a)) and $H^1-$norm difference (Panel (b)) between the solutions to the spatial exclusion model and the point source model with one (red dashed curve) and multiple Dirac points (black and blue solid curve), respectively. The $c^*(t)$ and the boundary flux difference at each time step are shown in Panel (c) and (d), respectively.}
    \label{Fig_1_Dirac_ratio_001_D_1}
\end{figure}

\subsubsection{Axially Oriented Flux Distribution: $n=2$}
\noindent

In Figure~\ref{Fig_2_Dirac_ratio_1_D_1} we compare solutions of single-Dirac and multi-Dirac point source models, the latter computed with different methods, to the solution of the spatial exclusion model with the axially oriented flux distribution ($n=2$). We use the standard parameter values from Table~\ref{tab:para_all}. Similar to the observations from the case of $n=1$, as long as $r$ is small enough, the overall performance of the discrepancy between the two models by using the multi-Dirac approach is better than the single-Dirac approach. However, due to the large intensity of the Dirac point, explicit Green's function approach is required to have a trustworthy numerical solution, particularly when the flux over the cell boundary is computed; see Figure~\ref{Fig_2_Dirac_ratio_1_D_1}(c)-(d).

When we increase $n$ in the predefined flux density, the multi-Dirac approach is more favoured over the single-Dirac approach, using the same $r$ in the observed time domain, e.g. Figure~\ref{Fig_2_Dirac_ratio_1_D_1}(a) and~\ref{Fig_1_Dirac_ratio_1_D_1}(a) with $r=0.1$. This observation motivates the examination in further detail of the domain of effective application of the multi-Dirac approach.

\begin{figure}[h!]
    \centering
    \subfigure[$\|u_S-u_P\|_{L^2(\Omega\setminus\Omega_C)}$]{
    \includegraphics[width = 0.48\textwidth]{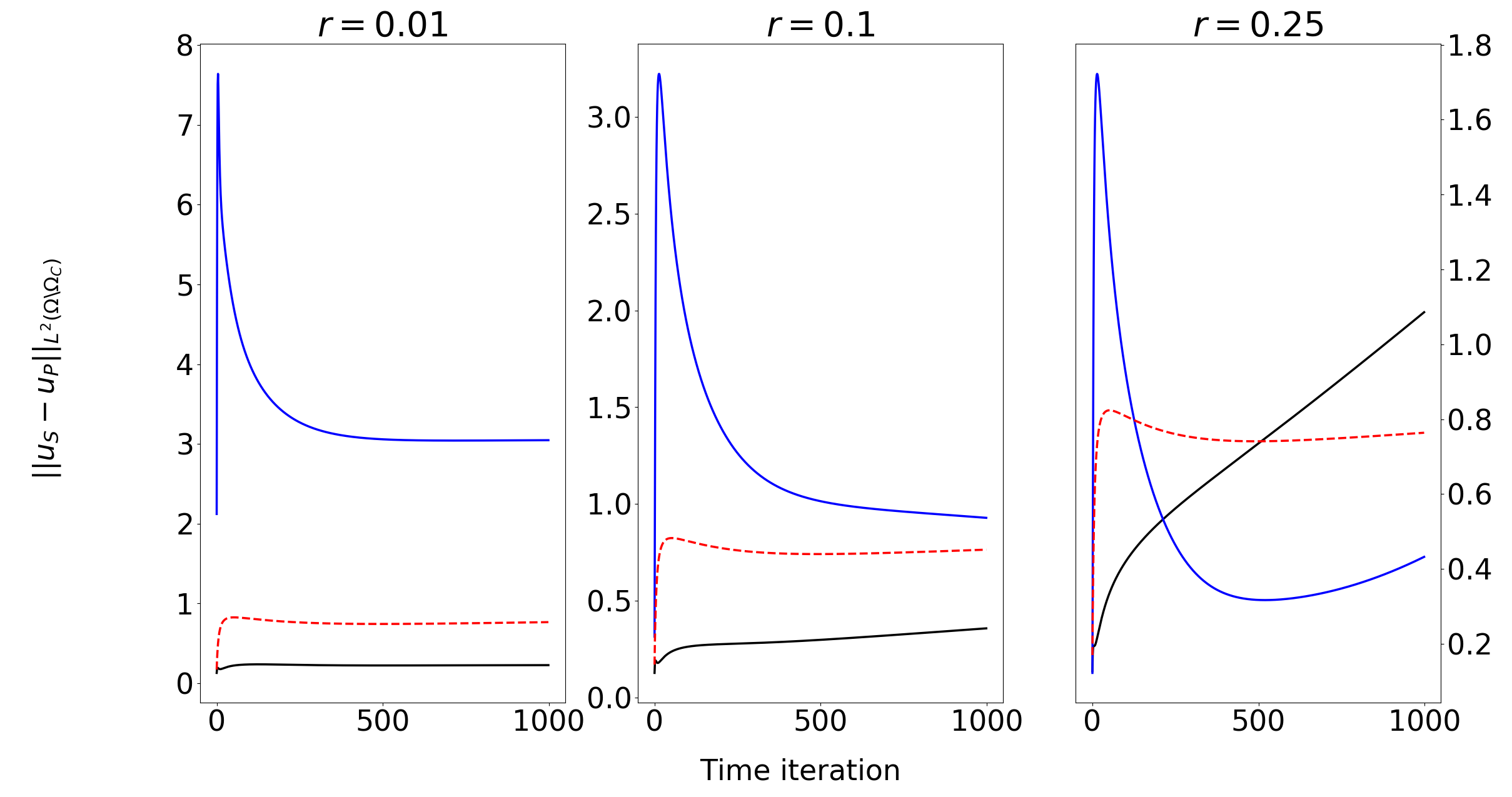}}
    \subfigure[$\|u_S-u_P\|_{H^1(\Omega\setminus\Omega_C)}$]{
    \includegraphics[width = 0.48\textwidth]{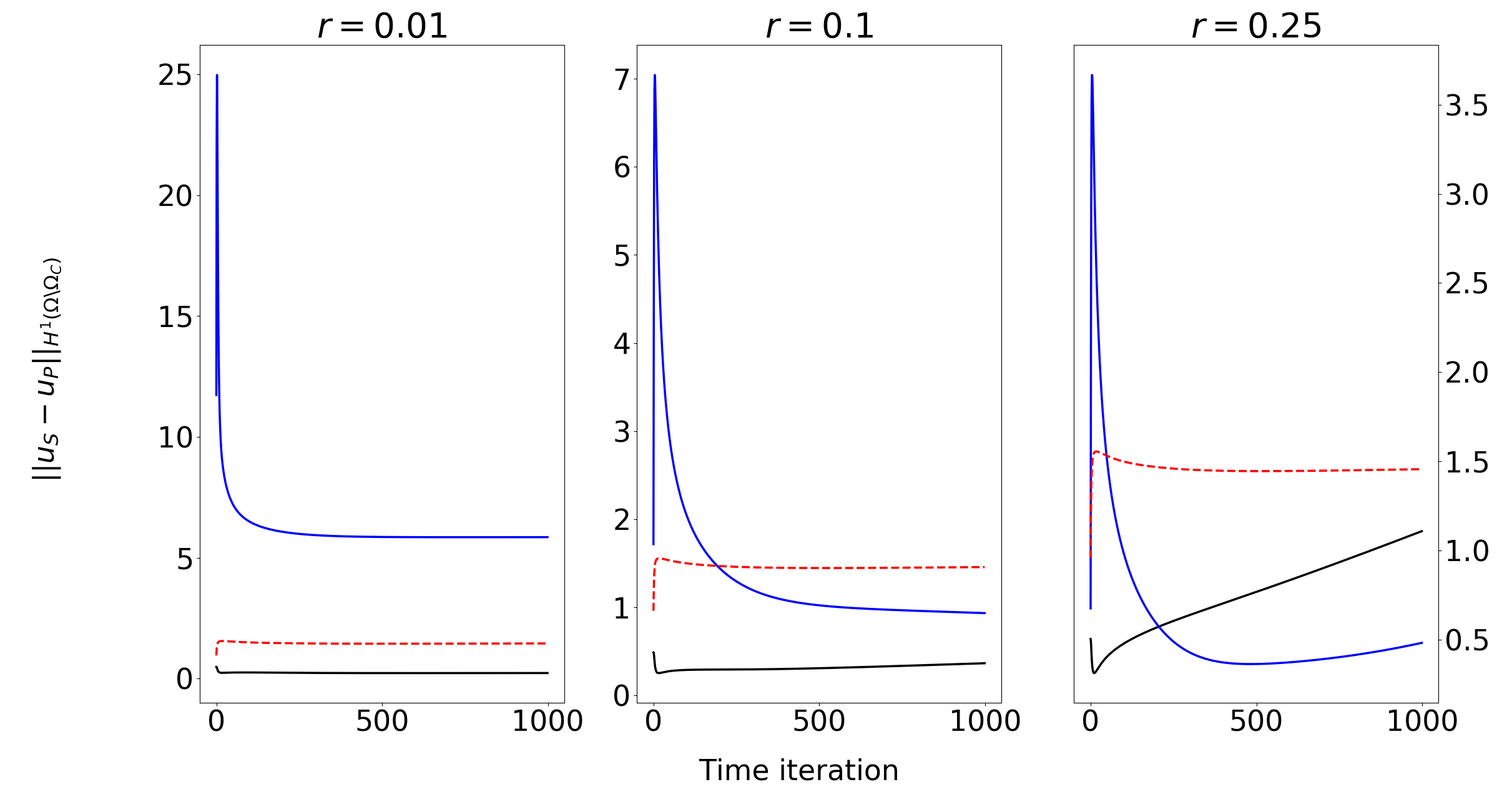}}
    \subfigure[$c^*(t)$]{
    \includegraphics[width = 0.48\textwidth]{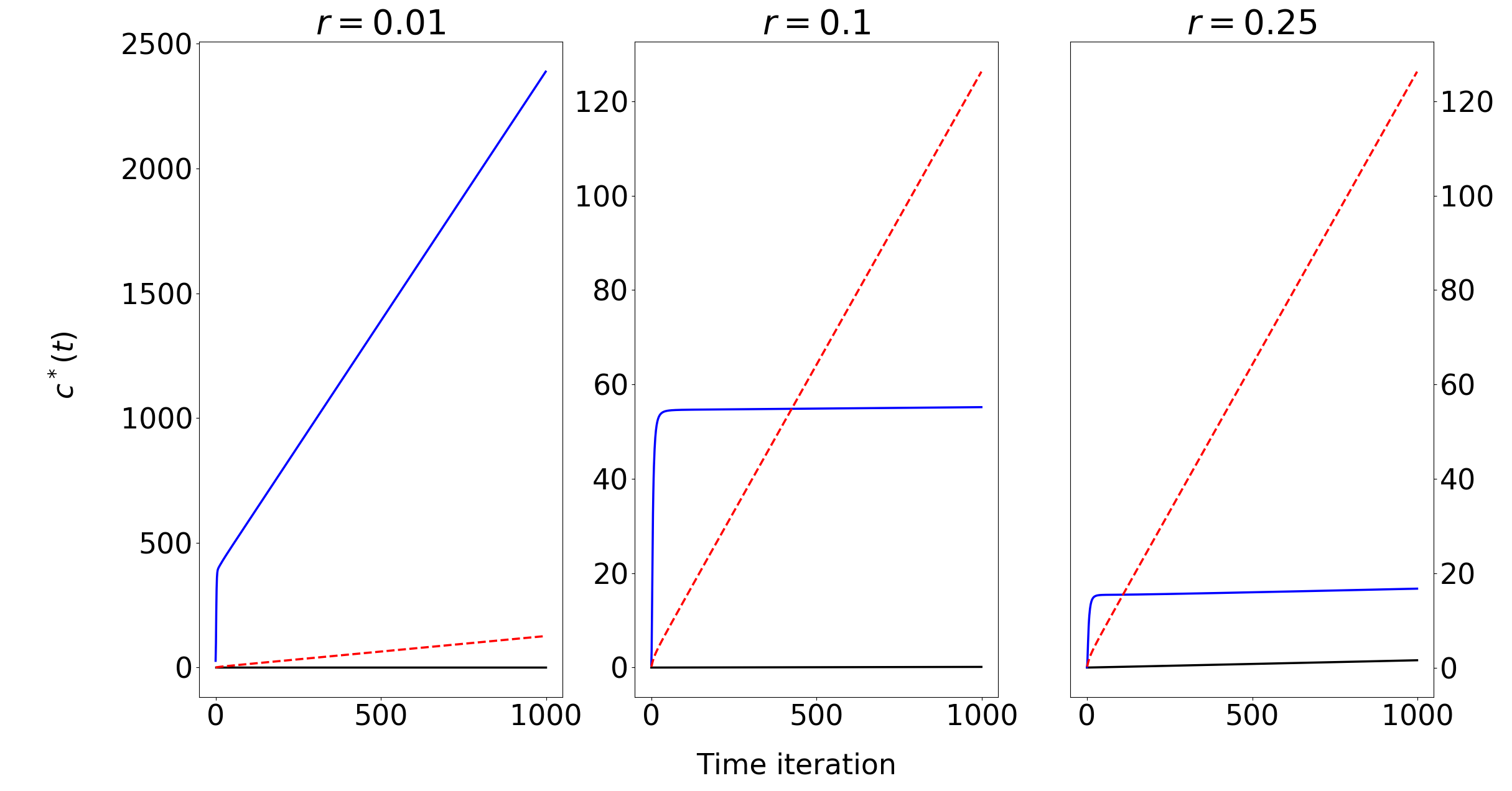}}
    \subfigure[$\|\phi(\boldsymbol{x}) - D\nabla u_P\cdot\boldsymbol{n}\|_{L^2(\partial\Omega_C)}$]{
    \includegraphics[width = 0.48\textwidth]{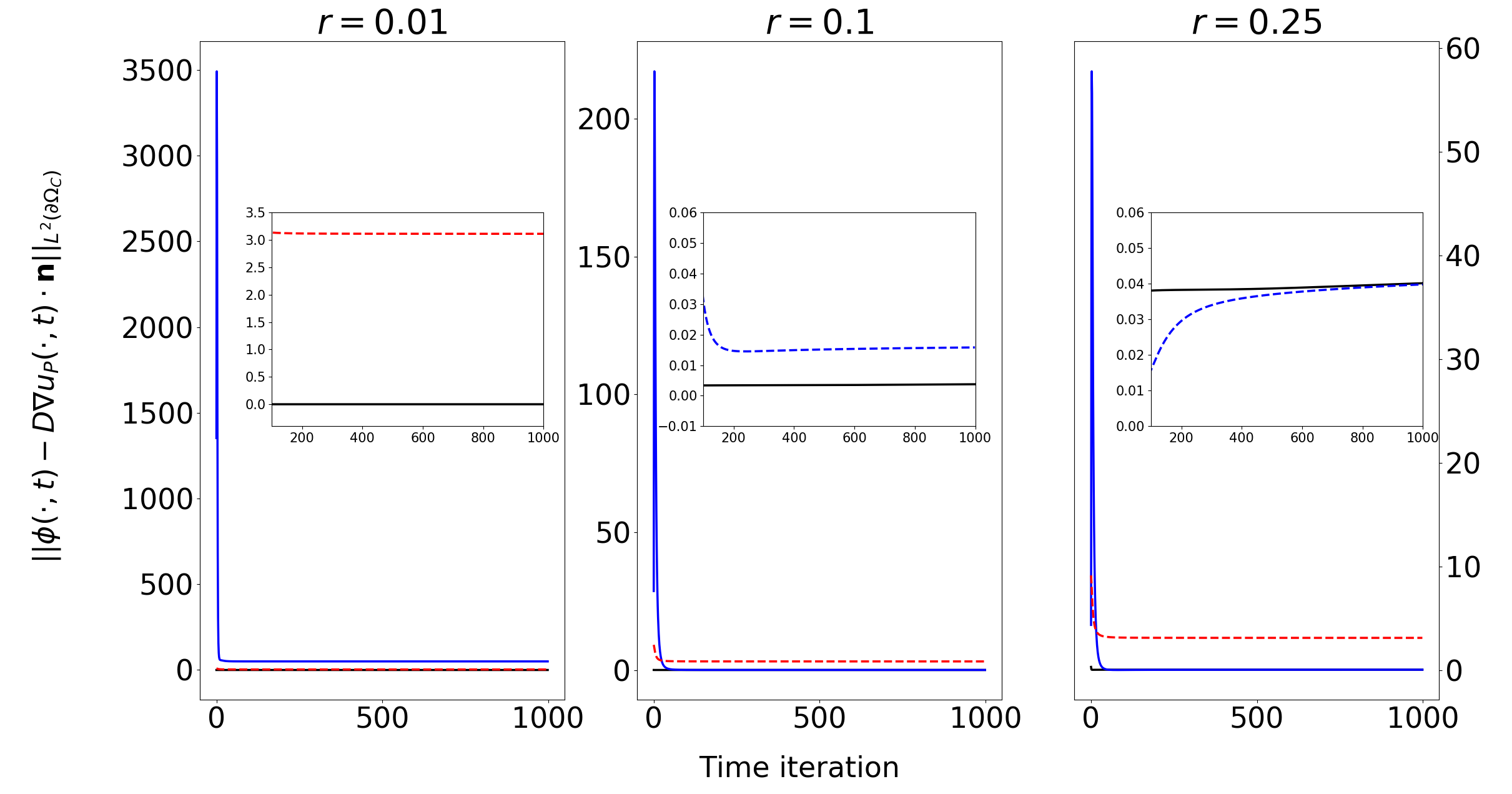}}
    \includegraphics[width = \textwidth]{3_curves_legend.png}
    \caption{Comparison of single- and multi-Dirac approach for the axially oriented flux density with ratio $\rho=1$, so $\phi(\boldsymbol{x}) = 1+\sin(2\theta)$. Simulations are made with the standard parameter values from Table~\ref{tab:para_all}. The local $L^2-$ (Panel (a)) and $H^1-$norm difference (Panel (b)) between the solutions to the spatial exclusion model and the point source model with one (red dashed curve) and multiple Dirac points (black and blue solid curve), respectively. The $c^*(t)$ and the boundary flux difference at each time step are shown in Panel (c) and (d), respectively.}
    \label{Fig_2_Dirac_ratio_1_D_1}
\end{figure}

\subsection{Application Domain of Multi-Dirac Approach}\label{Sec_MC_n_2}
\noindent
In the previous section, the results were obtained with specific and extreme parameter values, particularly the value of $\rho$ and $D$. In the selected cases, the simulation results indicate that the multi-Dirac approach gives better quality approximation than a single-Dirac approach. The question whether this property persists for all parameter combinations, is investigated now by means of Monte Carlo simulations with the diffusion constant $D$ and ratio $\rho$ as inputs. For $n=1$ and $n=2$ we take uniform distributions for both parameters on their order of magnitude:

\begin{itemize}
    \item $\log_{10}(D)\sim U(-3, 1.5)$;
    \item $\log_{10}(\rho)\sim U(-3,0)$.
\end{itemize}
The remaining parameters are fixed as in Table~\ref{tab:para_all}. To neglect the impact of the locations of the off-centre Dirac points, $r = 0.01$ is utilized. The distance in spatial $L^2$-norm between the numerical solution to the spatial exclusion model and the solution of the point source model with multi-Dirac approach and with single-Dirac approach is computed at the time steps of simulation. These two $L^2$-deviation curves are compared. The response of the Monte Carlo simulation is a label,  specified in Table~\ref{Tbl_MC_response} in more detail, representing to which approach in the point source model better approximates the spatial exclusion model in the time range $t\in(10, 40)$.
\begin{table}[h!]
    \centering
    \caption{The label description of the Monte Carlo simulations in Section~\ref{Sec_MC_n_2}.}
    \begin{tabular}{m{2cm}<{\centering}m{12cm}}
    \hline 
    \\[-1em]
    \textbf{Label}  & \textbf{Description}  \\
    \hline
    \\[-1em]
    $-1$ & One of the intensity functions $\widetilde{\Phi}_D(t)$ and $\widetilde{\Phi}_C(t)$ is too large for the numerical scheme, i.e. the entire numerical simulation fails. \\  
    \hline
    \\[-1em]
    $0$  &  The $L^2-$norm deviation between the multi-Dirac approach in the point source model and the spatial exclusion approach is \textit{smaller} than the one of the single-Dirac approach in the entire observation time domain.\\
    \hline
    \end{tabular}
    \label{Tbl_MC_response}
\end{table}

\subsubsection{Polarized Flux Distribution: $n=1$}
\noindent
We collected {$1213$} samples and the label obtained was either $-1$ or $0$; see Figure~\ref{Fig_n_1_MC_scatter}. This means that, as long as the numerical scheme does not fail, the multi-Dirac approach always performs better than the single-Dirac approach, in the perspective of the $L^2$-deviation of the solutions to the spatial exclusion model and the point source model. Moreover, the labels are clustered in two clusters that are separated by  $\log_{10}(D) = -2$, regardless of the value of $\rho$. It is mainly because a very small $D$ results in a huge intensity of the Dirac point, which causes numerical failure. 
\begin{figure}[h!]
    \centering
    \includegraphics[width = 0.75\textwidth]{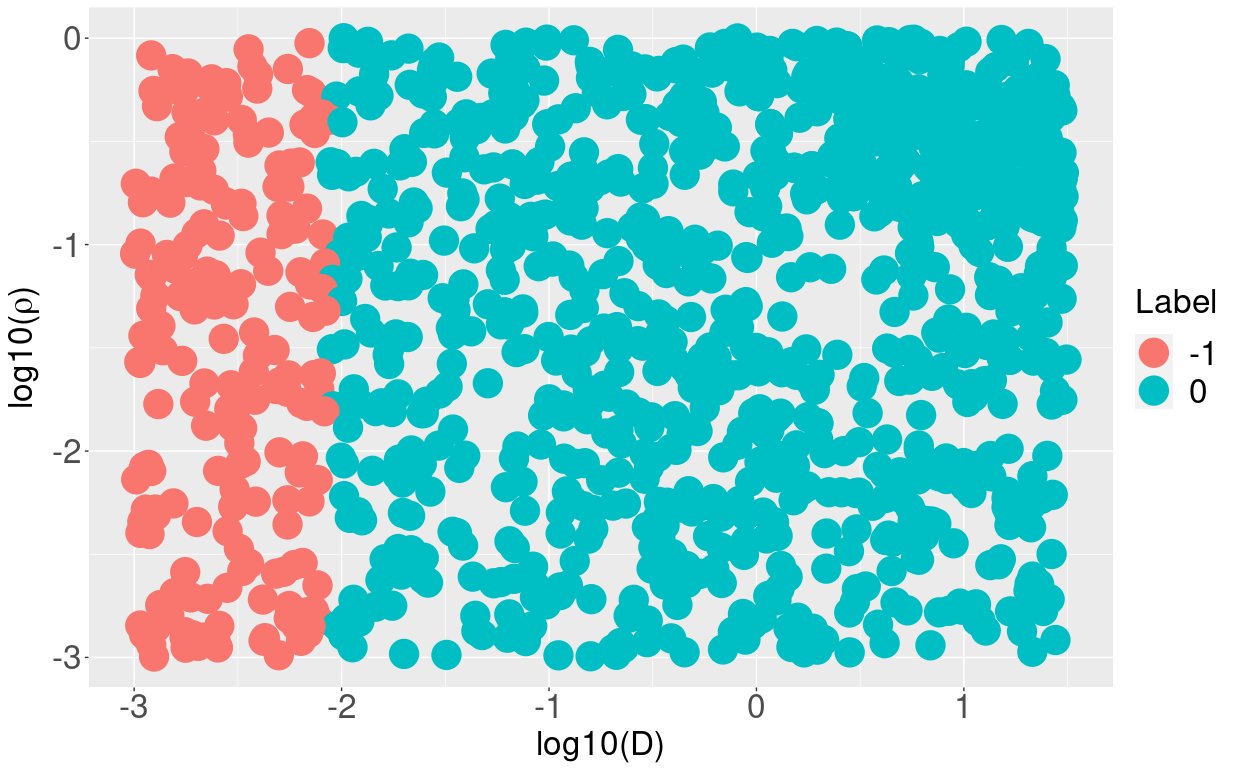}
    \caption{The scatter plot of $\log_{10}(\rho)$ and $\log_{10}(D)$. The color of the scatters represent the label shown in Table~\ref{Tbl_MC_response}.} 
    \label{Fig_n_1_MC_scatter}
\end{figure}

\subsubsection{Axially Oriented Flux Distribution: $n=2$}
\noindent
Figure~\ref{Fig_n_2_MC_scatter_logD_logratio} shows the scatter plot between $\log_{10}(D)$ and $\log_{10}(\rho)$ for the response labels for the axially oriented flux distribution (case $n=2$), with the labels marked by different colours. We collected {$1117$} samples in the dataset. The labels are again clearly divided into two connected clusters. The line $\log_{10}(D)=-2$ divides the clusters again as in the case $n=1$.

Similar to the issue in case $n=1$, a too-small $D$ results in a numerical failure due to the enormous value of the intensities computed by Equation~\eqref{Eq_tripole_sol}, regardless of the value of $\rho$.

\begin{figure}[h!]
    \centering
    \includegraphics[width = 0.75\textwidth]{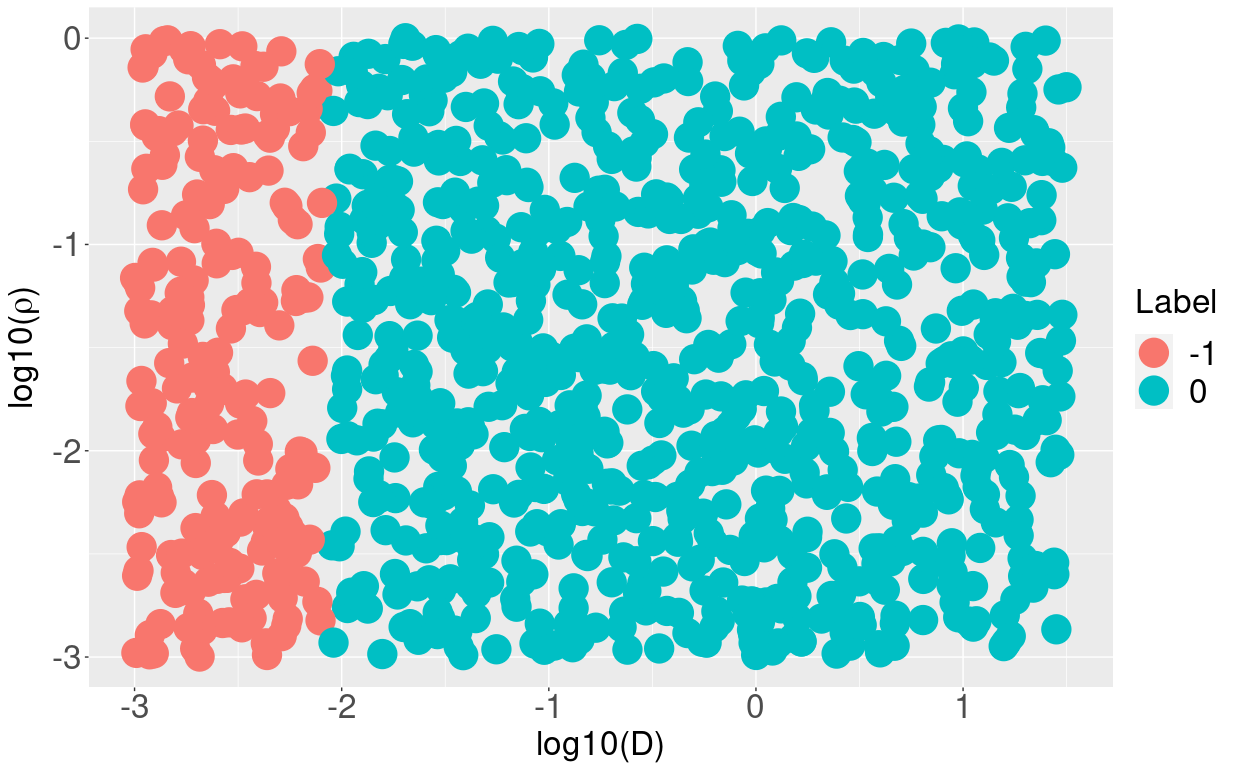}
    \caption{Scatter plot between $\log_{10}(D)$ and $\log_{10}(\rho)$. Here, the colours of the scatters represent the labels defined in Table~\ref{Tbl_MC_response}.}
    \label{Fig_n_2_MC_scatter_logD_logratio}
\end{figure}

\subsection{Numerical Scheme of Multi-Dirac Approach}\label{Sec_Mesh_Comparison}
\noindent

An important advantage of a point source model over a spatial exclusion model is that in the former a less complicated mesh is needed. In the latter, a fine mesh is required around the cell boundaries to handle well the diffusive spreading of released compounds. In the direct approach, a suitably fine mesh is still needed around the Dirac points, especially when such point has large intensity. In Section~\ref{Sec_Results} it was shown that one of the advantages of using the explicit Green's function approach is that the singularity from the Dirac point can be converted to the outer domain boundary. Hence, the numerical scheme can be stabilized compared to the direct approach. A coarser mesh is thus expected to provide a similar quality of approximation. We shall now investigate these claims further through numerical simulations. 

To compare the deviation when using different mesh size in the point source model with the explit Green's function approach and direct approach, we take the solution to the spatial exclusion model in a fine mesh as the reference solution, then the relative error in certain norm is defined by 
\begin{equation}
    \label{Eq_rela_error_mesh}
    \mathrm{r.e}(t) = \frac{\|u_S(t) - u_P(t)\|_{\Omega\setminus\Omega_C}}{\|u_S(t)\|_{\Omega\setminus\Omega_C}},
\end{equation}
where $u_S$ is the solution to the spatial exclusion model on the fine mesh. 

Figure~\ref{Fig_diff_mesh} presents the relative error defined in Equation~\ref{Eq_rela_error_mesh}, using the explicit Green's function approach and direct approach to solve $\mathrm{(BVP_P)}$, in different mesh size. Due to the large value of the intensity of the Dirac points, particularly in the direct approach, the relative error is large in the first few time steps. Hence, in the main figure, we show the time iterations from timestep $100$, while the entire time series of the relative error is shown in the embedded figures. Generally, using the same mesh structure, solving $\mathrm{(BVP_P)}$ with the explicit Green's function approach results in a less discrepancy to the solution to $\mathrm{(BVP_S)}$. Furthermore, there is a significant difference in relative error between the two approaches: for the explicit Green's function approach, using a coarse mesh only increases around $5\%$ relative error in the $L^2-$norm at the end of the simulation, whereas the increase of direct approach is around $25\%$. Regarding the $H^1-$norm, the difference in relative error between the mesh size is smaller in the explicit Green's function approach. Hence, this approach has the advantage of losing less numerical accuracy while using a coarse mesh, which can potentially increase the computational efficiency. 
\begin{figure}[h!]
    \centering
    \subfigure[Relative error of $L^2-$norm]{
    \includegraphics[width=0.48\textwidth]{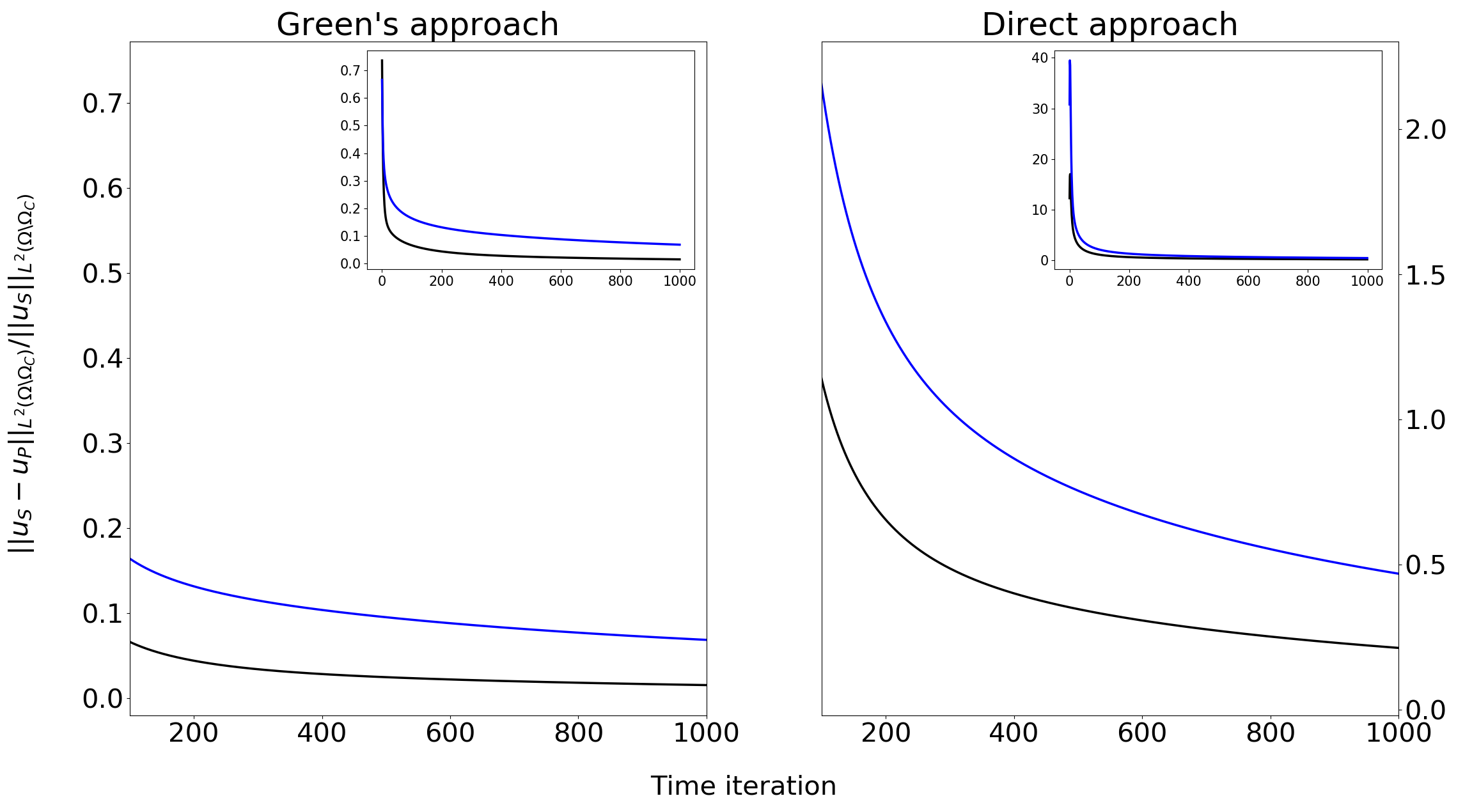}}
    \subfigure[Relative error of $H^1-$norm]{
    \includegraphics[width=0.48\textwidth]{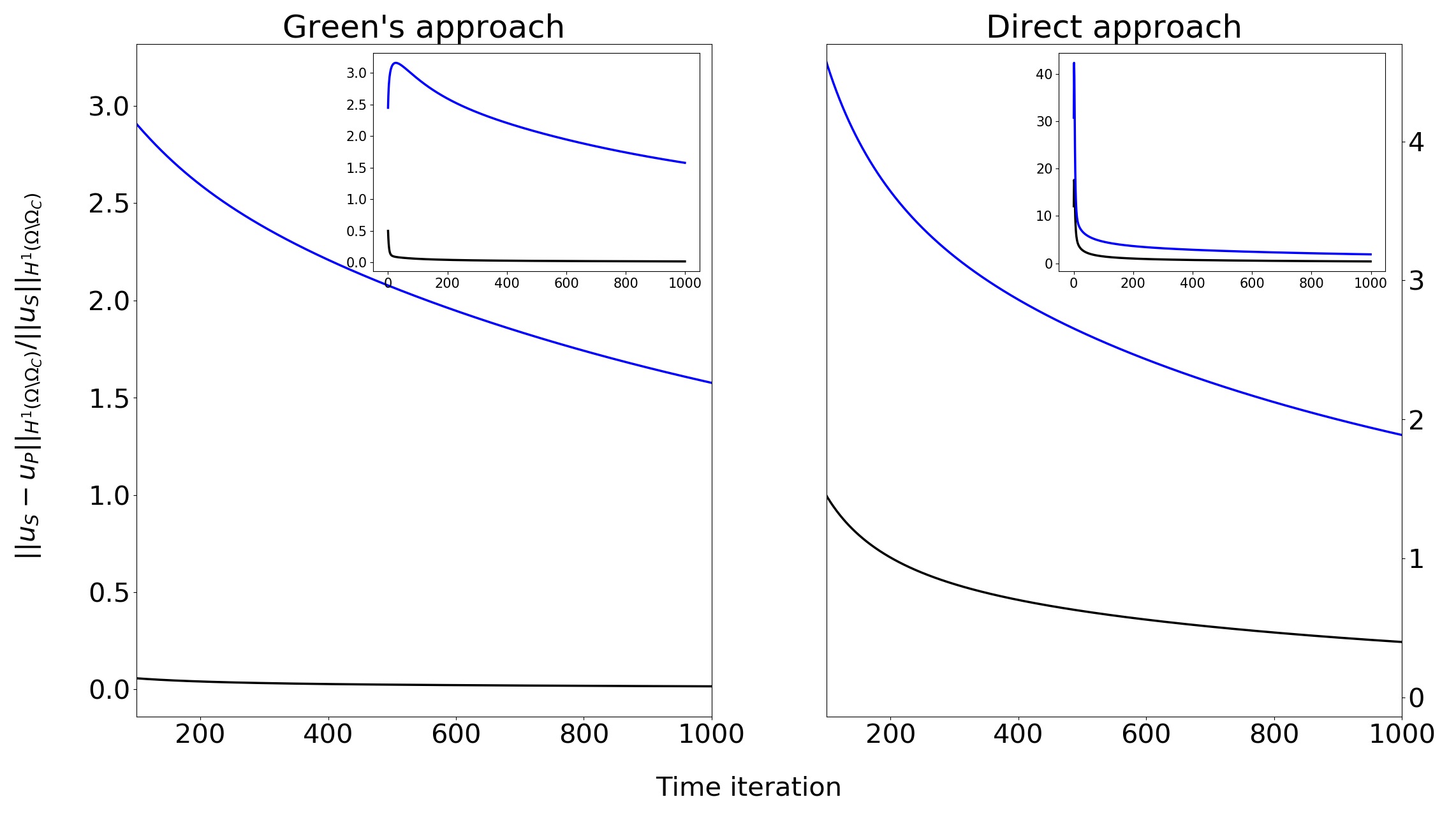}}
    \includegraphics[width = 0.6\textwidth]{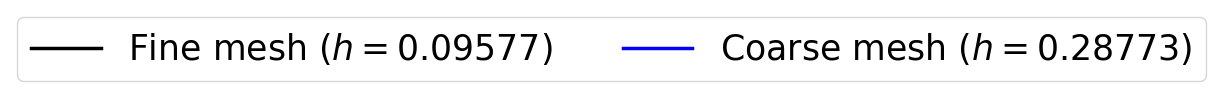}
    \caption{The relative error of $L^2-$ and $H^1-$ norm between the spatial exclusion and point source model, are shown in Panel (a) and (b), respectively. The point source model is solved by the explicit Green's function approach and the direct approach in a fine ($h = 0.09577$, black curve) and coarse ($h = 0.28773$, blue curve) mesh. The parameter values are the standard ones from Table~\ref{tab:para_all}.}
    \label{Fig_diff_mesh}
\end{figure}

\section{Conclusions and Discussion}\label{Sec_Conclu}
\noindent
In this study, we investigated how to use the point source model to approximate the spatial exclusion model for a single cell when the predefined flux density over the cell boundary is inhomogeneous, given by a specific sinusoidal function. This is a first explorative step towards settings in which there are more general flux inhomogeneities and more general cell shapes, and where there are many (and moving) cells. A point source model is expected to be computationally more efficient than a spatial exclusion model with FEM, which requires frequent remeshing in the latter setting. Forthcoming publications on results in these directions are in preparation. 

We proposed using several Dirac-delta point sources to represent the cell, in a specific configuration and with well-chosen intensities, such that the inhomogeneous flux over the cell boundary can be approximated. To compute the intensity of the point sources, we enforced that the extreme points of the prescribed flux density function agree with those of a modified flux expression for the point source model. We noticed numerically the convergence of this modified flux expression to the predefined flux density if $t\rightarrow+\infty$ and $r\rightarrow0^+$, then we also proved it in the analytical perspective in Appendix \ref{App_proof_Dirac} for $n=1$ and $n=2$. For the general case of higher $n$, we expect the same phenomenon to occur. An analytical proof of this result, which is not the main objective of this study, will be provided in a theoretically oriented article.

When the point source model was solved with the intensity of the Dirac points given by System~\eqref{eq:conditions for tilde_Phi}, a classical FEM approach could not deal with the large values of these intensities. Thus, the numerical solution is unstable. To solve this issue, we developed the explicit Green's function approach, which splits the numerical procedure into computing integrals that include the explicit Green's function (fundamental solution to the diffusion equation) on $\R^2$ and the computation of an adjustment that can be obtained  by solving a `regular' diffusion equation with prescribed boundary forcing, e.g. using standard FEM. 

There are several foreseen advantages of this approach:
\begin{enumerate}
    \item[(1)] The singularity that exists in $\hat{u}$ is now transferred to the computational domain boundary which is usually far away from the Dirac point. In other words, the smoothness inside the domain is guaranteed and hence, a fine mesh is not strictly required.
    \item[(2)] There is no need to reduce the time step significantly to maintain the stability of the numerical solutions.
    \item[(3)] Explicit Green's function approach is easier to implement in the case where there are many moving cells in the form of Dirac points, in particular, the cells are far away from the domain boundary.
    \item[(4)] The adjustment term does not need to be computed from the beginning of the simulation since it is very small compared with the fundamental solution. By doing this, the computational efficiency is improved.
\end{enumerate}
However, one possible drawback is the increasing computational expense when one implements the computation of the explicit Green's function partial solution $\hat{u}$ of Equation~\eqref{Eq_Green} for multiple Dirac points. Furthermore, this approach can only {\it extend} the range of the values for the intensity for which a numerically stable solution can be obtained, but it cannot completely resolve the numerical instability issue. It cannot deal still with the extremely large values. 

Further numerical simulation showed that the multi-Dirac approach performs better in terms of solution difference metrics between the two models, if $r$ is sufficiently small but it cannot be too small to result in a numerical failure. There are several significant parameters which influence the performance of the multi-Dirac approach, such as $n$, $r$, $\rho$ and $D$. In this manuscript, we investigated thoroughly the cases when $n=1$ and $n=2$. In both cases, as long as there is no numerical scheme failure, the multi-Dirac approach is always favoured in the point source model to obtain a more consistent numerical solution compared to the spatial exclusion model. 

In both $n=1$ and $n=2$, Monte Carlo simulations show that when $\log_{10}(D)<-2$ (i.e. $D<0.01$), regardless of the value of the ratio, the numerical scheme failed due to the enormous intensity of the Dirac points according to our approach. In other words, one can already tell whether the simulation will run into numerical issues or not based on the value of $D$. To tackle this issue, rescaling the PDE would be a possible solution such that the new value of $D$ is enlarged to a safe region. 

In \citet{Peng2023}, we acknowledged two types of error between the spatial exclusion model and the point source model, one of which is a \textit{systematic time delay}. That is, -- for a constant intensity for a single-Dirac point in a point source model -- if the diffusion coefficient is not large enough, then the compounds need a few time steps to reach the cell boundary before diffusing to the extracellular environment. Hence, this results in a time delay between the solution to the spatial exclusion model (with constant flux density) and the solution to the point source model. Here, we examined using time-dependent flux distributions and a multiple Dirac points approach. It seems that this systematic time delay is substantially reduced by enforcing that the flux density over the cell boundary in the point source model, resulting from the time-varying intensities at the Dirac points, to be the same in the time domain at the extreme point of the prescribed flux density.
\vskip 2mm

In summary, in this study, we focused on how to reproduce the inhomogeneous boundary condition of the cell in the spatial exclusion model by using multiple Dirac points as sink or source in the point source model. Depending on the parameter values, one can decide in advance whether it is necessary to use the multi-Dirac or single-Dirac approach. To reduce the numerical failure due to the large value of the intensity of Dirac points, we derived the explicit Green's function approach. However, we bare in mind that so far, for a very small diffusion constant, the numerical scheme will still fail.

\newpage
\bibliographystyle{abbrvnat}
\bibliography{main}

\begin{thebibliography}{28}
\providecommand{\natexlab}[1]{#1}
\providecommand{\url}[1]{\texttt{#1}}
\expandafter\ifx\csname urlstyle\endcsname\relax
  \providecommand{\doi}[1]{doi: #1}\else
  \providecommand{\doi}{doi: \begingroup \urlstyle{rm}\Url}\fi

\bibitem[Adams and Fournier(2003)]{Adams-Fournier:2003}
R.~A. Adams and J.~J. Fournier.
\newblock \emph{Sobolev Spaces}.
\newblock Elsevier Science Ltd, 2nd edition, 2003.

\bibitem[Alnæs et~al.(2015)Alnæs, Blechta, Hake, Johansson, Kehlet, Logg,
  Richardson, Ring, Rognes, and Wells]{AlnaesEtal2015}
M.~Alnæs, J.~Blechta, J.~Hake, A.~Johansson, B.~Kehlet, A.~Logg,
  C.~Richardson, J.~Ring, M.~E. Rognes, and G.~N. Wells.
\newblock The fenics project version 1.5.
\newblock \emph{Archive of Numerical Software}, Vol 3, 2015.
\newblock \doi{10.11588/ANS.2015.100.20553}.

\bibitem[Amann(2001)]{Amann:2001}
H.~Amann.
\newblock Linear parabolic equations involving measures.
\newblock \emph{{R}eal {A}cademia de {C}iencas {E}xactas, {F}isicas y
  {N}aturales. {R}evista. {S}erie {A}, {M}athematicas}, 356:\penalty0
  1045--1119, 2001.

\bibitem[Amann and Quittner(2004)]{Amann-Quittner:2004}
H.~Amann and P.~Quittner.
\newblock Semilinear parabolic equations involving measures and low regularity
  data.
\newblock \emph{Trans. AMS}, 356:\penalty0 1045--1119, 2004.

\bibitem[Binder(2020)]{Bender:2020}
B.~M. Binder.
\newblock Ethylene signaling in plants.
\newblock \emph{J. Biol. Chemistry}, 295:\penalty0 7710--7725, May 2020.

\bibitem[Buehler(2015)]{buehler2015cell}
L.~Buehler.
\newblock \emph{Cell membranes}.
\newblock Garland Science, 2015.

\bibitem[Cooper and Adams(2022)]{cooper2022cell}
G.~Cooper and K.~Adams.
\newblock \emph{The cell: a molecular approach}.
\newblock Oxford University Press, 2022.

\bibitem[Cooper and Adams(2023)]{cooper2023cell}
G.~M. Cooper and K.~Adams.
\newblock \emph{The cell: a molecular approach}.
\newblock Oxford University Press, 2023.

\bibitem[Ellingson(2018)]{Ellingson_2018}
S.~Ellingson.
\newblock \emph{Electromagnetics}.
\newblock {VT} Publishing, aug 2018.
\newblock \doi{10.21061/electromagnetics-vol-1}.

\bibitem[Evers et~al.(2015)Evers, Hille, and Muntean]{HMEvers2015}
J.~H. Evers, S.~C. Hille, and A.~Muntean.
\newblock Modelling with measures: Approximation of a mass-emitting object by a
  point source.
\newblock \emph{Mathematical Biosciences and Engineering}, 12\penalty0
  (2):\penalty0 357--373, 2015.
\newblock \doi{10.3934/mbe.2015.12.357}.

\bibitem[Fürstenberg-Hägg et~al.(2013)Fürstenberg-Hägg, Zagrobelny, and
  Bak]{Furstenberg_Hagg_2013}
J.~Fürstenberg-Hägg, M.~Zagrobelny, and S.~Bak.
\newblock Plant defense against insect herbivores.
\newblock \emph{International Journal of Molecular Sciences}, 14\penalty0
  (5):\penalty0 10242–10297, May 2013.
\newblock ISSN 1422-0067.
\newblock \doi{10.3390/ijms140510242}.

\bibitem[Gjerde et~al.(2019)Gjerde, Kumar, and Nordbotten]{Gjerde_2019}
I.~G. Gjerde, K.~Kumar, and J.~M. Nordbotten.
\newblock A singularity removal method for coupled 1d–3d flow models.
\newblock \emph{Computational Geosciences}, 24\penalty0 (2):\penalty0
  443–457, December 2019.
\newblock ISSN 1573-1499.
\newblock \doi{10.1007/s10596-019-09899-4}.

\bibitem[Griffiths(2005)]{griffiths2005introduction}
D.~J. Griffiths.
\newblock Introduction to electrodynamics, 2005.

\bibitem[Haleem et~al.(2023)Haleem, Javaid, Pratap~Singh, and
  Suman]{Haleem_2023}
A.~Haleem, M.~Javaid, R.~Pratap~Singh, and R.~Suman.
\newblock Exploring the revolution in healthcare systems through the
  applications of digital twin technology.
\newblock \emph{Biomedical Technology}, 4:\penalty0 28–38, December 2023.
\newblock ISSN 2949-723X.
\newblock \doi{10.1016/j.bmt.2023.02.001}.

\bibitem[Harvanová et~al.(2023)Harvanová, Duranková, and
  Bernasovská]{Harvanov__2023}
G.~Harvanová, S.~Duranková, and J.~Bernasovská.
\newblock The role of cytokines and chemokines in the inflammatory response.
\newblock \emph{Alergologia Polska - Polish Journal of Allergology},
  10\penalty0 (3):\penalty0 210–219, 2023.
\newblock ISSN 2353-3854.
\newblock \doi{10.5114/pja.2023.131708}.

\bibitem[Janeway et~al.(2001)Janeway, Travers, Walport, Shlomchik,
  et~al.]{janeway2001immunobiology}
C.~Janeway, P.~Travers, M.~Walport, M.~Shlomchik, et~al.
\newblock \emph{Immunobiology: the immune system in health and disease},
  volume~2.
\newblock Garland Pub. New York, 2001.

\bibitem[Karban(2015)]{Karban:2015}
R.~Karban.
\newblock \emph{Plant Sensing and Communication}.
\newblock University of Chigago Press, 2015.

\bibitem[Keski-Vakkuri et~al.(2022)Keski-Vakkuri, Montonen, and
  Panero]{keski2022mathematical}
E.~Keski-Vakkuri, C.~Montonen, and M.~Panero.
\newblock \emph{Mathematical Methods for Physics: An Introduction to Group
  Theory, Topology and Geometry}.
\newblock Cambridge University Press, 2022.

\bibitem[Lowry et~al.(1989)Lowry, Allen, and Shive]{Lowry_1989}
T.~Lowry, M.~B. Allen, and P.~N. Shive.
\newblock Singularity removal: A refinement of resistivity modeling techniques.
\newblock \emph{Geophysics}, 54\penalty0 (6):\penalty0 766–774, June 1989.
\newblock ISSN 1942-2156.
\newblock \doi{10.1190/1.1442704}.

\bibitem[Menini(2009)]{menini2009neurobiology}
A.~Menini.
\newblock \emph{The neurobiology of olfaction}.
\newblock CRC Press, 2009.

\bibitem[Mierke(2020)]{mierke2020cellular}
C.~T. Mierke.
\newblock \emph{Cellular Mechanics and Biophysics}.
\newblock Springer, 2020.

\bibitem[Neitzel and Rasband(2014)]{Neitzel2014Cell}
J.~Neitzel and M.~Rasband.
\newblock Cell communication.
\newblock https://www.nature.com/scitable/topic/cell-communication-14122659/,
  2014.

\bibitem[Peng and Hille(2023)]{Peng2023}
Q.~Peng and S.~C. Hille.
\newblock Quality of approximating a mass-emitting object by a point source in
  a diffusion model.
\newblock \emph{Computers \& Mathematics with Applications}, 151:\penalty0
  491–507, Dec. 2023.
\newblock ISSN 0898-1221.
\newblock \doi{10.1016/j.camwa.2023.10.034}.

\bibitem[Peng and Vermolen(2019)]{peng2019_DIAM_19_02}
Q.~Peng and F.~Vermolen.
\newblock Numerical methods to solve elasticity problems with point sources.
\newblock 2019.

\bibitem[Perbal(2003)]{Perbal_2003}
B.~Perbal.
\newblock \emph{Cell Communication and Signaling}, 1\penalty0 (1):\penalty0 3,
  2003.
\newblock \doi{10.1186/1478-811x-1-3}.

\bibitem[Sun et~al.(2023)Sun, He, and Li]{Sun_2023}
T.~Sun, X.~He, and Z.~Li.
\newblock Digital twin in healthcare: Recent updates and challenges.
\newblock \emph{Digital Health}, 9:\penalty0 205520762211496, January 2023.
\newblock ISSN 2055-2076.
\newblock \doi{10.1177/20552076221149651}.

\bibitem[van Kan et~al.(2005)van Kan, Segal, and Vermolen]{van2005numerical}
J.~van Kan, A.~Segal, and F.~J. Vermolen.
\newblock \emph{Numerical methods in scientific computing}.
\newblock VSSD, 2005.

\bibitem[Yang et~al.(2024)Yang, Peng, and Hille]{Yang-Peng-Hille:2024}
X.~Yang, Q.~Peng, and S.~C. Hille.
\newblock Approximation of a compound-exchanging cell by a dirac point, 2024.
\newblock URL \url{https://arxiv.org/abs/2410.09495}.

\end{thebibliography}

\newpage
\begin{appendix}
\section*{Appendices}\label{Sec_Appendix}
\section{Location of extremes of $\hat{\phi}$ that are fixed for all time}\label{App_proof_1_Dirac_extremum}
\label{app:same location extrema}

\noindent Recall that $\tilde{\phi}_P(\boldsymbol{x},t)$ represented the flux density over the (virtual) cell boundary $\partial\Omega_C$ in the point source model in de bounded domain $\Omega$, when the intensities at the Dirac points $\boldsymbol{x}^{(i)}$ were taken to be $\tilde{\Phi}_i(t)$ ($i=0,1,\dots, N$). In turn, $\tilde{\phi}_P(\boldsymbol{x},t)$ was approximated by the explicitly computable flux density over $\partial\Omega_C$ given by the point source model in $\R^2$. That is, (see Expression~\eqref{eq:expresion tilde Phi}):
\begin{equation}
    \hat{\phi}(\boldsymbol{x}_\theta,t) := \sum_{i=0}^N \frac{\widetilde{\Phi}_i(t)}{2\pi R}  \frac{(\boldsymbol{x}_\theta - \boldsymbol{x}_C)\cdot(\boldsymbol{x}_\theta-\boldsymbol{x}^{(i)})}{\|\boldsymbol{x}_\theta - \boldsymbol{x}^{(i)}\|^2}   \exp\left\{-\frac{\|\boldsymbol{x}_\theta - \boldsymbol{x}^{(i)}\|^2}{4Dt}\right\}.
    \tag{\ref{eq:expresion tilde Phi}}
\end{equation}

\begin{lemma}
    For $N+1$ Dirac points $\boldsymbol{x}^{(i)}$, $i=0,1,\dots,N$, located in $\Omega_C$ with $\boldsymbol{x}^{(0)}$ being the centre point $\boldsymbol{x}_C$ and for generic intensities $\tilde{\Phi}_i\neq 0$, the approximate flux density $\tilde{\phi}(\boldsymbol{x}_\theta,t)$ can have critical points located at a fixed position for all time $t>0$ if and only if 
    \begin{equation}\label{eq:condition fixed extrema}
        \frac{\partial\boldsymbol{x}_\theta}{\partial\theta}\cdot \bigl(\boldsymbol{x}_\theta -\boldsymbol{x}^{(i)}\bigr)=0,\quad \mbox{for all}\ i=0,1,\dots, N.
    \end{equation}
\end{lemma}
\begin{proof}
    One computes that
    \begin{align}
        \frac{\partial}{\partial\theta} \hat{\phi}(\boldsymbol{x}_\theta,t)\ &=\ \sum_{i=0}^N \frac{\tilde{\Phi}_i(t)}{2\pi R} \exp\left(-\frac{\bigl\| \boldsymbol{x}_\theta-\boldsymbol{x}^{(i)} \bigr\|^2}{4Dt} \right) \left\{  \frac{\partial}{\partial\theta} \left[  \frac{(\boldsymbol{x}_\theta - \boldsymbol{x}_C)\cdot(\boldsymbol{x}_\theta-\boldsymbol{x}^{(i)})}{\|\boldsymbol{x}_\theta - \boldsymbol{x}^{(i)}\|^2}\right] \right.\nonumber\\
        &\qquad\qquad -\ \left.\frac{1}{4Dt} \, 2\bigl(\boldsymbol{x}_\theta -\boldsymbol{x}^{(i)}\bigr)\cdot\frac{\partial \boldsymbol{x}_\theta}{\partial\theta} \, \frac{(\boldsymbol{x}_\theta - \boldsymbol{x}_C)\cdot(\boldsymbol{x}_\theta-\boldsymbol{x}^{(i)})}{\|\boldsymbol{x}_\theta - \boldsymbol{x}^{(i)}\|^2} \right\} \label{eq:expr derivative hat phi}
    \end{align}
    Since $\partial\Omega_C$ is a circle of radius $R$, centred at $\boldsymbol{x}_C=\boldsymbol{x}^{(0)}$,
    \begin{equation}\label{eq:derivative xC zero}
        \frac{\partial \boldsymbol{x}_\theta}{\partial\theta}\cdot \bigl(\boldsymbol{x}_\theta -\boldsymbol{x}_C\bigr)\ =\ \frac{1}{2} \frac{\partial}{\partial\theta} \bigl\| \boldsymbol{x}_\theta -\boldsymbol{x}_C\bigr\|^2\ =\ 0
    \end{equation}
    and Equation~\eqref{eq:expr derivative hat phi} simplifies to
    \begin{align}
        \frac{\partial}{\partial\theta} \hat{\phi}(\boldsymbol{x}_\theta,t)\ &=\ \sum_{i=0}^N \frac{\tilde{\Phi}_i(t)}{2\pi R\bigl\| \boldsymbol{x}_\theta-\boldsymbol{x}^{(i)} \bigr\|^2} \exp\left(-\frac{\bigl\| \boldsymbol{x}_\theta-\boldsymbol{x}^{(i)} \bigr\|^2}{4Dt} \right)
        \; \frac{\partial\boldsymbol{x}_\theta}{\partial\theta}\cdot \bigl(\boldsymbol{x}_\theta -\boldsymbol{x}^{(i)}\bigr)\nonumber\\
        &\qquad\qquad\times \left\{ 1 - 2\bigl(\boldsymbol{x}_\theta - \boldsymbol{x}_C\bigr)\cdot\bigl(\boldsymbol{x}_\theta-\boldsymbol{x}^{(i)}\bigr) \left( \frac{1}{4Dt} + \frac{1}{ \bigl\| \boldsymbol{x}_\theta-\boldsymbol{x}^{(i)} \bigr\|^2 }\right) \right\}.\label{eq:simplified derivative hat phi}
    \end{align}
    If there exists a critical point at a fixed location $\boldsymbol{x}_\theta$ for all $t>0$, one must have that $\frac{\partial}{\partial\theta} \hat{\phi}(\boldsymbol{x}_\theta,t)=0$ for all $t>0$. Assuming that all $\tilde{\Phi}_i\neq 0$, this condition can be satisfied for generic $\tilde{\Phi}_i$ if and only if $\theta$ is such that Condition~\eqref{eq:condition fixed extrema} holds.
\end{proof}

Condition~\eqref{eq:condition fixed extrema} imposes a strong constraint on the possible Dirac point configuration. Note that for $i=0$ the condition is always satisfied (for the circular cell) according to Expression~\eqref{eq:derivative xC zero}. For the other $i$, note that the vector $\frac{\partial\boldsymbol{x}_\theta}{\partial\theta}$ is tangential to $\partial\Omega_C$ at $\boldsymbol{x}_\theta$. Thus, the condition requires all vectors pointing from an a-central Dirac point to the point on the boundary to be orthogonal to the boundary. Such points exist only if $N=1$ or if $N=2$ and the two a-central points are located on the same line through the centre. Thus,
\begin{corollary}\label{clry:fixed nextremes one off-centre Dirac}
    If $N=1$, then $\hat{\phi}(\boldsymbol{x}_\theta,t)$ attains extreme values for all $t>0$ at some fixed $\boldsymbol{x}_\theta$ for generic non-zero $\tilde{\Phi}_i$ if $\theta$ is such that $\boldsymbol{x}_\theta$ is on the line through centre $\boldsymbol{x}_C$ and $\boldsymbol{x}^{(1)}$.
\end{corollary}

For the case $N=1$ and the configuration of Dirac points as in Figure~\ref{Fig_sine_1_Dirac}(a) we compute the derivative explicitly as
\begin{align*}
    \frac{\partial\hat{\phi}_1}{\partial\theta}(\boldsymbol{x}_\theta,t) &= \frac{C_1(t)r\cos(\theta)}{R^2+r^2-2Rr\sin(\theta)}\exp\left(-\frac{R^2+r^2-2Rr\sin(\theta)}{4Dt}\right) \\
    &\qquad \times \left[\frac{R^2-r^2}{R^2+r^2-2Rr\sin(\theta)}+\frac{R(R-r\sin(\theta))}{2Dt}\right], 
\end{align*}
where $C_1(t)$ is given by Equation~\eqref{eq: C1}. It is clear that the numerator of $C_1(t)$ is always non-negative, since $r\in(0,R)$. The denominator of $C_1(t)$ is also always non-negative as the numerator can be rewritten as
\begin{align*}
    \exp\left\{-\frac{(R-r)^2}{4Dt}\right\}\left((R+r)-(R-r)\exp\left\{-\frac{Rr}{Dt}\right\}\right)&\geqslant \exp\left\{-\frac{(R-r)^2}{4Dt}\right\}((R+r) - (R-r))\\
    & = 2r\exp\left\{-\frac{(R-r)^2}{4Dt}\right\}>0.
\end{align*}
Thus, we can conclude that $C_1(t)$ is always non-negative. All the terms except for $\cos(\theta)$ in the expression for $\displaystyle\frac{\partial\hat{\phi}_1}{\partial\theta}$ are strictly positive for $r\in(0, R)$. Hence, the critical points are $\displaystyle\frac{\pi}{2}$ and $\\displaystylefrac{3\pi}{2}$ in accordance with Corollary~\ref{clry:fixed nextremes one off-centre Dirac}. By the first-order derivative test, the maximum occurs at $\frac{\pi}{2}$ and the minimum occurs at $\displaystyle\frac{3\pi}{2}$, which are the same as $\phi(\boldsymbol{x}_\theta) = \phi_0+A\sin(\theta)$.

\begin{corollary}
    If $N=2$, then $\hat{\phi}(\boldsymbol{x}_\theta,t)$ attains extreme values for all $t>0$ at some fixed $\boldsymbol{x}_\theta$ for generic $\tilde{\Phi}_i$ if $\boldsymbol{x}^{(1)}$ and $\boldsymbol{x}^{(2)}$ lie on a line through the centre $\boldsymbol{x}^{(0)}$. Then, $\theta$ is such that $\boldsymbol{x}_\theta$ lies on this line.
\end{corollary}

In case $N=2$ with $\boldsymbol{x}^{(1)}$ and $\boldsymbol{x}^{(2)}$ on a line through the centre, at equal distance on opposite sides thereof, such as in Figure~\ref{Fig_general_Dirac}\subref{Fig_2_Diracs}, there exists a fixed location of an extreme value of $\hat{\phi}(\boldsymbol{x}_\theta,t)$ for all $t>0$ also on the line through the centre, orthogonal to the line through $\boldsymbol{x}^{(1)}$ and $\boldsymbol{x}^{(2)}$, provided $\tilde{\Phi}_1=\tilde{\Phi}_2$. Indeed, for $\boldsymbol{x}_\theta$ on this orthogonal line, with the symmetric configuration of $\boldsymbol{x}^{(1)}$ and $\boldsymbol{x}^{(2)}$,
\[
    (\boldsymbol{x}_\theta - \boldsymbol{x}_C\bigr)\cdot\bigl(\boldsymbol{x}_\theta-\boldsymbol{x}^{(1)}\bigr)=(\boldsymbol{x}_\theta - \boldsymbol{x}_C\bigr)\cdot\bigl(\boldsymbol{x}_\theta-\boldsymbol{x}^{(2)}\bigr)
\]
while
\[
    \frac{\partial\boldsymbol{x}_\theta}{\partial\theta}\cdot \bigl(\boldsymbol{x}_\theta -\boldsymbol{x}^{(1)}\bigr) = -\frac{\partial\boldsymbol{x}_\theta}{\partial\theta}\cdot \bigl(\boldsymbol{x}_\theta -\boldsymbol{x}^{(2)}\bigr).
\]
Thus, the location of the fixed extremes for all time of $\hat{\phi}$ for the given set-up are the same as those of $\phi(\boldsymbol{x}_\theta) = \phi_0+A\sin(2\theta)$.

For $N\geq 3$, there cannot be a configuration of Dirac points, other than all on the same line through the centre, such that Condition~\eqref{eq:condition fixed extrema} holds. In that case, the $\tilde{\Phi}_i$ cannot be arbitrary (non-zero). Then, symmetry of the configuration of Dirac points together with relations among the $\tilde{\Phi}_i$ may still result in $\frac{\partial}{\partial\theta} \hat{\phi}(\boldsymbol{x}_\theta,t)=0$ for all $t>0$. We shall not pursuit this here any further.

\section{Convergence of $\hat{\phi}_\infty$ to $\phi$ in Section~\ref{Sec_Multi}}\label{App_proof_Dirac}
Figure~\ref{Fig_sine_1_Dirac} and Figure~\ref{Fig_sine_2_Dirac} indicate the convergence of $\hat{\phi}_\infty$ to $\phi$, when $r\rightarrow0^+$. The propositions below verifies this convergence analytically for $n=1$ and $n=2$, respectively. 
\begin{proposition}
    Given $\phi_1(\theta) = \phi_0+A\sin(\theta)$, where $\phi_0$ and $A$ are positive constants. Then the flux over the cell boundary $\partial\Omega_C$ from $(\rm BVP_p)$ at the steady state, computed by Equation~\eqref{Eq_flux_stst_1}, converges to $\phi_1(\theta)$ when $r\rightarrow0^+$.
\end{proposition}
\begin{proof}
Substituting the expressions of $\widetilde{\Phi}_D$ and $\widetilde{\Phi}_C$ in Equation~\eqref{Eq_dipole_sol}, we obtain Equation~\eqref{Eq_flux_stst_1}:
\begin{equation}
    \hat{\phi}_\infty(\theta; r) = \phi_0 + A - \frac{A(R+r)^2(1-\sin(\theta))}{R^2+r^2-2Rr\sin(\theta)}. \tag{\ref{Eq_flux_stst_1}}
\end{equation}
Then, we proceed further by sending $r\rightarrow0^+$:
\begin{align*}
    \lim_{r\rightarrow0^+}\hat{\phi}_\infty(\theta; r) &= \phi_0+A-\frac{AR^2(1-\sin(\theta))}{R^2}\\
    &= \phi_0 + A\sin(\theta) 
\end{align*}
Hence, we proved the proposition.
\end{proof}

Numerically it can be seen that the theorem above also holds for any integer $n$ in the predefined flux density $\phi = \phi_0+A\sin(n\theta)$; see Figure~\ref{Fig_sine_2_3} for $n = 2$ and $n=3$ as examples. 
\begin{figure}
    \centering
    \subfigure[$\phi = \phi_0+A\sin(3\theta)$]{
    \includegraphics[width = 0.75\textwidth]{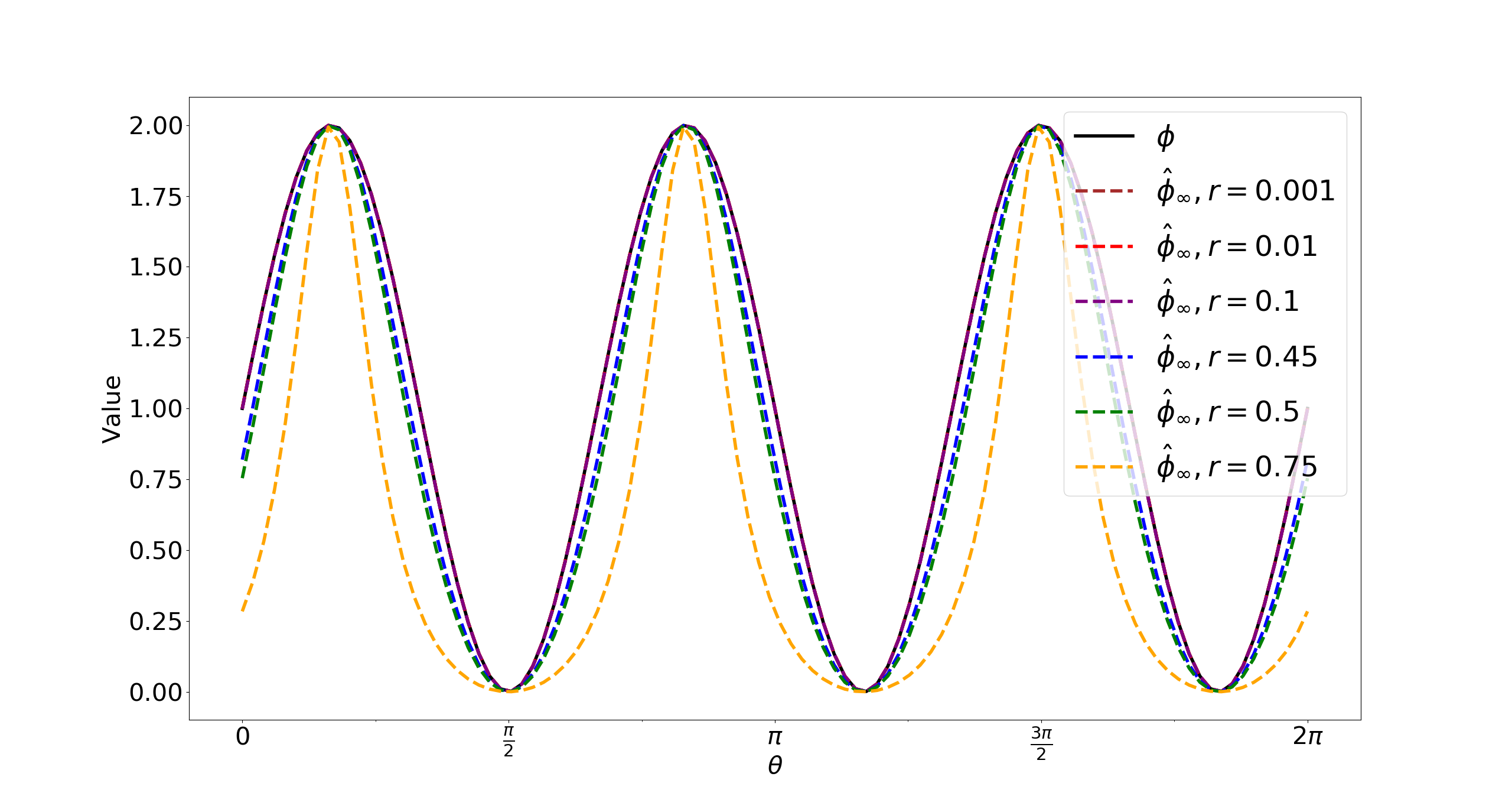}}
    \subfigure[$\phi = \phi_0+A\sin(4\theta)$]{
    \includegraphics[width = 0.75\textwidth]{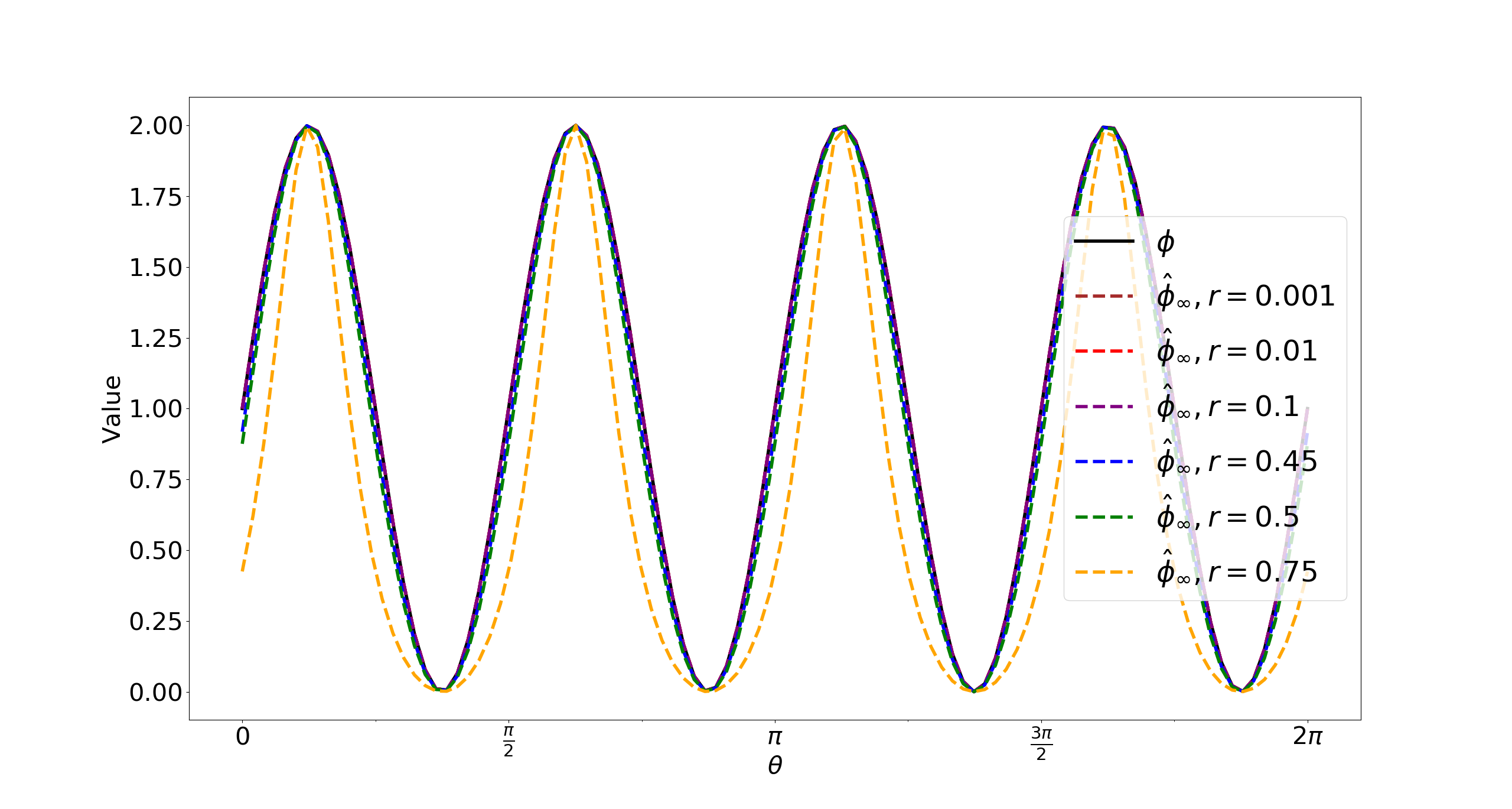}}
    \caption{Examples of $n=3$ and $n=4$ with $\widetilde{\Phi}_D$ and $\widetilde{\Phi}_C$ computed by System~\eqref{eq:conditions for tilde_Phi}.}
    \label{Fig_sine_2_3}
\end{figure}

For $n=2$, the convergence can be claimed by the proposition:
\begin{proposition}
    Given $\phi_2(\theta) = \phi_0+A\sin(2\theta)$, where $\phi_0$ and $A$ are positive constants. When $r\rightarrow0^+$, the flux over the cell boundary $\partial\Omega_C$ from $(BVP_P)$ at steady state (i.e. Equation~\eqref{Eq_2_Dirac_bnd_flux_stst}) converges to $\phi_2(\theta)$
\end{proposition}
\begin{proof}
    We combine the fractions in Equation~\eqref{Eq_2_Dirac_bnd_flux_stst}. Then it gives 
    \begin{align*}
        \hat{\phi}_{2,\infty}(\theta;r) = \phi_0-A+\frac{Ar^4-AR^2r^2+Ar^4\sin(2\theta)-AR^2r^2+AR^4+AR^4\sin(2\theta)-2AR^2r^2\sin(2\theta)}{R^4+r^4-2R^2r^2\sin(2\theta)}.
    \end{align*}
    Taking $r\rightarrow0^+$ yields point-wise in $\theta$:
    \begin{align*}
        \lim_{r\rightarrow0^+}\hat{\phi}_{2,\infty}(\theta;r) = \phi_0-A+\frac{AR^4+AR^4\sin(2\theta)}{R^4} = \phi_0+A\sin(2\theta).
    \end{align*}
\end{proof}

\section{Shape of the intensity of Dirac points for $n=1$ in Section~\ref{Subsec_1_Dirac}}\label{App_Phi_D_Phi_C}
\noindent
Recall the solution to System~\eqref{Eq_dipole} that is stated in Equation~\eqref{Eq_dipole_sol}:
\begin{equation}
    \left\{
    \begin{aligned}
        \widetilde{\Phi}_D(t) &= \frac{4\pi A}{\frac{1}{R-r}\exp\left\{-\frac{(R-r)^2}{4Dt}\right\}-\frac{1}{R+r}\exp\left\{-\frac{(R+r)^2}{4Dt}\right\}},\\
        \widetilde{\Phi}_C(t) &= 2\pi R\exp\left\{\frac{R^2}{4Dt}\right\}\left(\phi_0+A-\frac{2A(R+r)}{(R+r)-(R-r)\exp\left\{-\frac{Rr}{Dt}\right\}}\right). 
    \end{aligned}
    \right.
    \tag{\ref{Eq_dipole_sol}}
\end{equation}
We start with $\widetilde{\Phi}_D$. Regarding the location of the off-centre Dirac point, it is clear that $r$ cannot be $0$ or $R$, which will result in a singular solution and $\widetilde{\Phi}_D(t) = 0$ for any $t$, respectively. Moreover, when $t \rightarrow 0$, $\Phi_D$ becomes infinity. The typical shape of $\widetilde{\Phi}_D$ is exhibited in Figure~\ref{Fig_Phi_D_Appendix}. When $t$ is close to zero, the value of $\widetilde{\Phi}_D$ is very large then as $t$ increases, the value drops down rapidly to the minimum at $\displaystyle t_{min} = D\frac{\ln\frac{R+r}{R-r}}{Rr}$. Afterwards, the value increases again and gradually reaches the steady state with $t$ large.
\begin{figure}[h!]
    \centering
    \includegraphics[width = 0.75\textwidth]{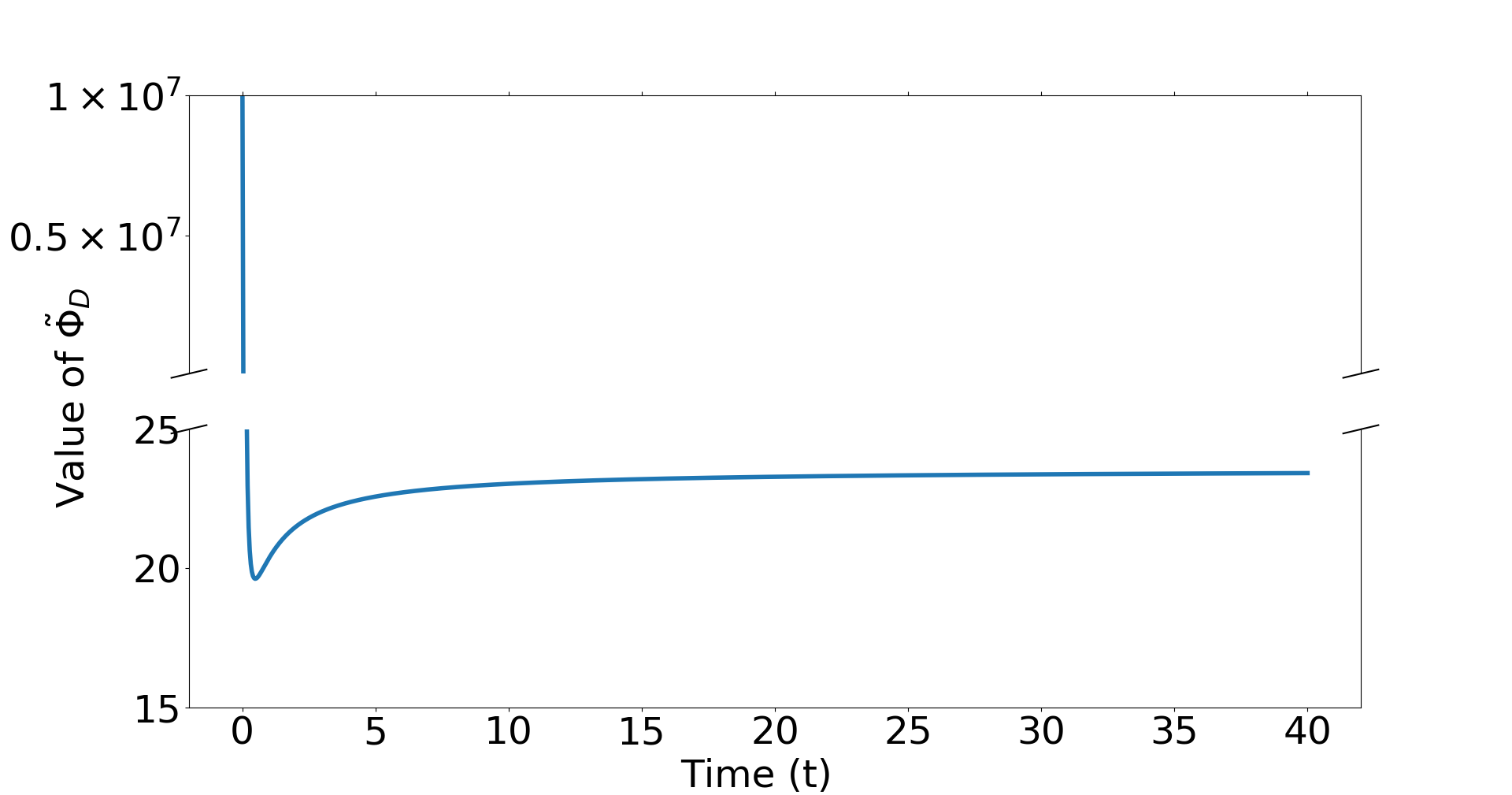}
    \caption{A typical shape of $\widetilde{\Phi}_D$ as a function of time $t$. Here, $r=0.25$ and $A/\phi_0 = 1$.}
    \label{Fig_Phi_D_Appendix}
\end{figure}

As for $\widetilde{\Phi}_C(t)$, there are two possible cases: (1) the function is a monotonically decreasing function; or (2) the function contains one local minimum and one local maximum, respectively. In Figure~\ref{Fig_Phi_C_Appendix}, we show the examples of each case. According to the analysis, the value of $\beta:=r/R$ is decisive for the shape of $\widetilde{\Phi}_C(t)$:
\begin{enumerate}
    \item[(1)] When $\beta\geqslant\frac{1}{4}$ and for any value of $\rho$, $\widetilde{\Phi}_C$ is a decreasing function over $(0, +\infty)$.
    \item[(2)] When $\beta<\frac{1}{4}$ and $\frac{8\beta}{16\beta^2+1}<\rho<1<\frac{1}{4\beta}$, then $\widetilde{\Phi}_C$ decreases first in $(0, t_1)$ then increases in $(t_1, t_2)$, afterwards it increases again in $(t_2, +\infty)$, where $t_1$ is the local minimum and $t_2$ is the local maximum. 
\end{enumerate}
\begin{figure}[h!]
    \centering
    \subfigure[Monotonically decreasing: $\rho=A/\phi_0 = 0.1$ and $\beta = \frac{r}{R} = 0.6$]{
    \includegraphics[width = 0.75\textwidth]{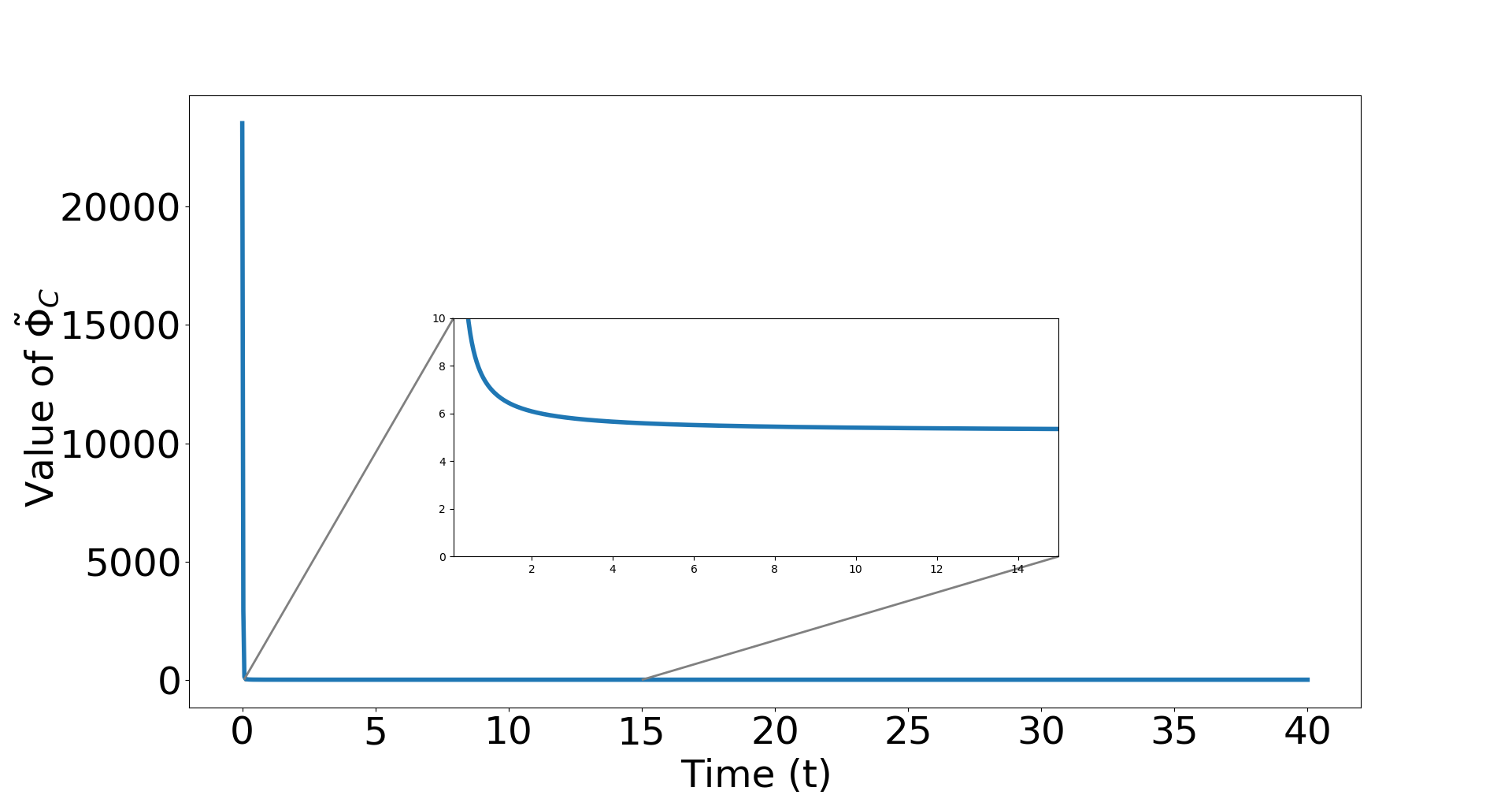}}
    \subfigure[Containing two local extrema: $\rho=A/\phi_0 = 0.1$ and $\beta = \frac{r}{R} = 0.01$]{
    \includegraphics[width = 0.75\textwidth]{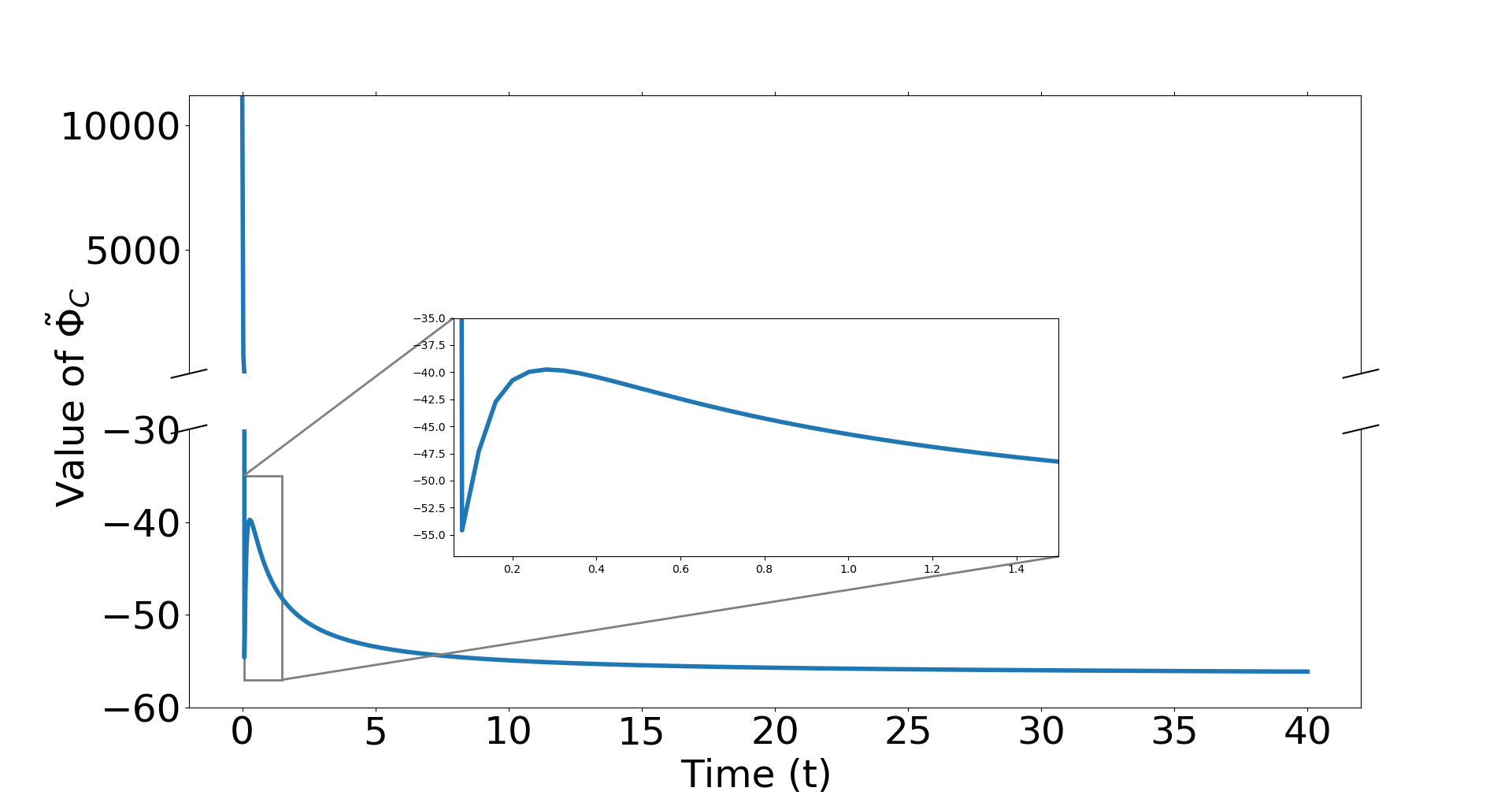}
    }
    \caption{The typical shape of $\widetilde{\Phi}_C(t)$ with different sets of parameter values. Here, the diffusion coefficient $D = 1.0$, the cell radius $R=1.0$ and the constant in the flux density $\phi_0 = 1.0$.}
    \label{Fig_Phi_C_Appendix}
\end{figure}

\end{appendix}

\end{document}